\documentclass[]{article}

\usepackage{graphicx}
\usepackage{amsmath}
\usepackage{amssymb}
\usepackage{amsthm} 
\usepackage{bm}
\usepackage{pxfonts}
\usepackage{enumerate}
\usepackage{color}
\usepackage{mathdots}
\usepackage{sectsty}
\usepackage[hidelinks]{hyperref}
\usepackage{tikz}
\usepackage{caption}
\usepackage{adjustbox}
\usepackage[title]{appendix}
\allowdisplaybreaks

\sectionfont{\scshape\centering\fontsize{11}{14}\selectfont}
\subsectionfont{\scshape\fontsize{11}{14}\selectfont}
\usepackage{fancyhdr}

\newcommand\shorttitle{On the singular values of complex matrix Brownian motion with a matrix drift}
\newcommand\authors{T. Assiotis}

\newtheorem{thm}{Theorem}[section]
\newtheorem{cor}[thm]{Corollary}
\newtheorem{lem}[thm]{Lemma}
\newtheorem{defn}[thm]{Definition}
\newtheorem{rmk}[thm]{Remark}
\newtheorem{prop}[thm]{Proposition}

\title{\large \bf ON THE SINGULAR VALUES OF COMPLEX MATRIX BROWNIAN MOTION WITH A MATRIX DRIFT}
\author{\small THEODOROS ASSIOTIS}
\date{}

\begin{document}

\maketitle

\begin{abstract}
Let $\textnormal{Mat}_{\mathbb{C}}(K,N)$ be the space of $K\times N$ complex matrices. Let $\mathbf{B}_t$ be Brownian motion on $\textnormal{Mat}_{\mathbb{C}}(K,N)$ starting from the zero matrix and $\mathbf{M}\in\textnormal{Mat}_{\mathbb{C}}(K,N)$. We prove that, with $K\ge N$, the $N$ eigenvalues of $\left(\mathbf{B}_t+t\mathbf{M}\right)^*\left(\mathbf{B}_t+t\mathbf{M}\right)$ form a Markov process with an explicit transition kernel. This generalizes a classical result of Rogers and Pitman \cite{RogersPitman} for multidimensional Brownian motion with drift which corresponds to $N=1$. We then give two more descriptions for this Markov process. First, as independent squared Bessel diffusion processes in the wide sense, introduced by Watanabe \cite{Watanabe} and studied by Pitman and Yor \cite{PitmanYor}, conditioned to never intersect. Second, as the distribution of the top row of interacting squared Bessel type diffusions in some interlacting array. The last two descriptions also extend to a general class of one-dimensional diffusions.
\end{abstract}

\tableofcontents

\section{Introduction}

The fact that the square radial part of multidimensional Brownian motion is a Markov process is well-known. From the point of view of the theory of Markov processes this is a simple application of Dynkin's criterion \cite{RevuzYor}. Alternatively it can be proven using It\^{o}'s formula which produces a closed stochastic differential equation (SDE) for the squared Bessel process \cite{RevuzYor,SurveyBessel}. 

When we incorporate a drift vector however these simple approaches no longer work; for example It\^{o}'s formula does not produce a closed equation. Nevertheless, it is still the case that the resulting radial process is Markovian with an explicit transition kernel. This is a theorem of Rogers and Pitman from \cite{RogersPitman}. It was proven using a celebrated criterion (which they introduced in the same paper), that involves intertwinings, of when a function of a Markov process is itself Markovian.

In this work we consider a matrix generalization of this result. We look at a rectangular complex matrix Brownian motion with an arbitrary matrix drift, $\left(\mathbf{B}_t+t\mathbf{M};t\ge 0\right)$. The corresponding process $\left(\mathbf{X}_t^{\mathbf{M}};t\ge 0\right)=\left(\left(\mathbf{B}_t+t\mathbf{M}\right)^*\left(\mathbf{B}_t+t\mathbf{M}\right);t\ge 0\right)$ on positive definite matrices generalizes the much-studied Laguerre matrix diffusion, see \cite{OConnell,Demni}, which is the special case $\mathbf{M}\equiv 0$. Its distribution at a fixed time $T\ge 0$ is given by the non-central complex Wishart distribution. This distribution seems to have been first studied in detail in the statistics literature by James \cite{James2}. A key tool in the analysis is the theory of zonal polynomials developed independently by James and Hua, see \cite{James1,Hua,Muirhead,Macdonald}. In the mathematical physics literature this random matrix is also usually called the chiral Gaussian ensemble with an external source, see for example \cite{QCD,QCD2}. In the past couple of decades there has been significant interest also in asymptotic (as the matrix size goes to infinity) questions related to this ensemble, in particular about global and local eigenvalue statistics, large deviations and connections to free probability, see for example \cite{DozierSilverstein,ForresterLiu,HardyKuijlaars,Benaych} and the references therein. Returning to the discussion of the dynamical picture, our first main result in this paper is that the square singular values of $\left(\mathbf{B}_t+t\mathbf{M};t\ge 0\right)$, equivalently the eigenvalues of $\left(\mathbf{X}_t^{\mathbf{M}};t\ge 0\right)$, form a Markov process with an explicit transition kernel.

We then prove that this Markov process can be realised as independent squared Bessel diffusion processes in the wide sense conditioned to never collide, in particular it consists of non-intersecting paths. The Bessel diffusion process in the wide-sense was introduced by Watanabe \cite{Watanabe} in his investigation of invariance of one-dimensional diffusions under time inversion and then studied in detail by Pitman and Yor \cite{PitmanYor}. The prototypical result of the kind we prove here is a theorem of Grabiner \cite{Grabiner} for Dyson Brownian motion \cite{Dyson,AGW}, the evolution of eigenvalues of Brownian motion on Hermitian matrices, which realises it as independent Brownian motions conditioned to never intersect. An analogous interpretation exists for the eigenvalues of the Laguerre matrix diffusion discovered by K\"onig and O'Connell in \cite{OConnell}. The result of Grabiner for Brownian motions was extended to the case with drifts in \cite{BBO}, see also \cite{JonesOConnell,Puchala,Katori1,Katori2} for further details on conditioned to never intersect drifting Brownian motions. Finally, we show in this paper that the non-collision probability and the transition kernel of the corresponding conditioned process can be computed explicitly for a more general class of one-dimensional diffusions.

We then give a third and final description of the Markov process under consideration as the distribution of the top row of interacting squared Bessel type diffusions in some (non-triangular) interlacing array. The dynamics we consider here are in analogy to interacting, via reflections, Brownian motions in a triangular interlacing array, see \cite{Warren}. In the Brownian case the projection on the autonomous edge particle system gives the well-known model of Brownian motions with one-sided reflections, equivalently Brownian last passage percolation. An analogous, albeit more complicated, autonomous interacting particle system also appears at the edge of the array in our dynamics. Connections to random matrices for Brownian motions with one-sided reflections, at a fixed time, were first discovered by Baryshnikov \cite{Baryshnikov} and Gravner-Tracy-Widom \cite{GTW} and extended to the process level by O'Connell-Yor \cite{OConnellYor} and Bougerol-Jeulin \cite{BougerolJeulin}. Analogous results in the setting of the present paper are a consequence of our investigation. The Brownian model \cite{Warren} can also be extended to include drifts for the interacting Brownian motions \cite{FerrariFrings,InterlacingDiffusions}. The fact that it is possible to involve drifts in the autonomous edge particle system of Brownian motions with one-sided reflections, while retaining some of the integrability of the model, also plays an important role in a recent characterization theorem of the KPZ fixed point from KPZ universality, see \cite{Virag}. Finally, this construction involving squared Bessel type diffusions that we present here also extends to a more general class of interacting one-dimensional diffusions, via collisions, in an interlacing array so that the distribution of the top row matches the one of the conditioned process mentioned at the end of the previous paragraph. We conclude by noting that a consequence of our results (that we will not explore further here, see Appendix \ref{SectionDet}) is that there is an underlying determinantal point process \cite{AGW,BorodinRains} structure in all the models we consider.

In the next three subsections of the introduction we introduce the necessary notations and terminology to state precisely these three descriptions of the Markov process described above. The proofs (in a more general setting) are then given in Sections \ref{SectionMatrix}, \ref{SectionConditioned} and \ref{SectionInteractingDiffusions} respectively.

\subsection{The matrix process and its eigenvalues}

Let $\textnormal{Mat}_{\mathbb{C}}(K,N)$ be the space of $K\times N$ complex matrices.
Let $\textnormal{Her}_\mathbb{C}(N)$ be the space of $N\times N$ complex Hermitian matrices and $\textnormal{Her}^+_\mathbb{C}(N)$ be the subset of positive definite ones. We define the map $\mathsf{eval}$ on Hermitian matrices that takes a matrix $\mathbf{H}\in \textnormal{Her}_\mathbb{C}(N)$ to its ordered, in a non-decreasing fashion, eigenvalues $\lambda_1(\mathbf{H})\le \dots \le \lambda_N(\mathbf{H})$
\begin{align*}
\mathsf{eval}\left(\mathbf{H}\right)=\left(\lambda_1(\mathbf{H}),\dots,\lambda_N(\mathbf{H})\right).
\end{align*}
Throughout we denote by $\left(\mathbf{B}_t;t\ge 0\right)$ the following stochastic process on $\textnormal{Mat}_{\mathbb{C}}(K,N)$:
\begin{align*}
 \left(\mathbf{B}_t\right)_{ij}=\mathsf{w}_{ij}(t)+\textnormal{i}\tilde{\mathsf{w}}_{ij}(t),
\end{align*}
where $\left(\mathsf{w}_{ij}\right)_{\substack{i=1,\dots, K \\ j=1,\dots, N}}$ and $\left(\tilde{\mathsf{w}}_{ij}\right)_{\substack{i=1,\dots, K \\ j=1,\dots, N}}$ are independent standard (real) Brownian motions. We call any matrix process on $\textnormal{Mat}_{\mathbb{C}}(K,N)$ which has the same distribution as $\left(\mathbf{B}_t;t \ge 0\right)$ a matrix Brownian motion on $\textnormal{Mat}_{\mathbb{C}}(K,N)$.

Suppose $K\ge N$ and consider the following stochastic process $\left(\mathbf{X}_t;t\ge 0\right)=\left(\mathbf{B}_t^*\mathbf{B}_t;t\ge 0\right)$ on $\textnormal{Her}^+_\mathbb{C}(N)$. This is called the Laguerre matrix process and was first studied in \cite{OConnell}. Some of its properties were then investigated in detail in \cite{Demni}. Its analogue on real symmetric matrices is called the Wishart process and was studied earlier in \cite{Bru}. Using  It\^{o}'s formula, see \cite{Demni}, it can be seen that the Laguerre process solves the following closed matrix SDE
\begin{align*}
    d\mathbf{X}_t=\sqrt{\mathbf{X}_t}d\mathbf{\Gamma}_t+d\mathbf{\Gamma}_t^*\sqrt{\mathbf{X}_t}+2K\mathbf{I}_Ndt,
\end{align*}
where $\mathbf{\Gamma}_t$ is a matrix Brownian motion on $\textnormal{Mat}_{\mathbb{C}}(N,N)$ obtained from $\mathbf{B}_t$ and $\mathbf{I}_N$ is the $N\times N$ identity matrix.

It is a remarkable fact, first proven in \cite{OConnell}, that the evolution of its eigenvalues $\left(\mathsf{eval}\left(\mathbf{X}_t\right);t\ge 0\right)$, equivalently the square singular values of $\left(\mathbf{B}_t;t\ge0\right)$, is Markovian. This evolution consists of non-intersecting paths living\footnote{The eigenvalues are almost surely strictly ordered and non-vanishing for all $t>0$, see \cite{Bru,Demni,Noncolliding}.} in the (strictly) ordered chamber
\begin{align*}
 \mathbb{W}_N=\{x=(x_1,\dots,x_N)\in (0,\infty)^N : x_1<\cdots<x_N \}.
\end{align*}
Moreover, the transition density with respect to the Lebesgue measure in $\mathbb{W}_N$ of this Markov process is explicit and given by\footnote{This is exactly the transition density of $N$ independent squared Bessel diffusions of index $\nu$ conditioned to never intersect, as proven in \cite{OConnell}.}, see \cite{OConnell,Demni}, with the parameter $\nu=K-N$:
\begin{align}\label{LaguerreEvalSemi}
    q_t^{(\nu),N}\left(x,y\right)=\frac{\mathsf{\Delta}_N(y)}{\mathsf{\Delta}_N(x)}\det\left(q_t^{(\nu)}(x_i,y_j)\right)_{i,j=1}^N.
\end{align}
Here, and throughout this paper, $\mathsf{\Delta}_N$ denotes the Vandermonde determinant
\begin{align*}
 \mathsf{\Delta}_N(z)=\det\left(z_i^{j-1}\right)_{i,j=1}^N=\prod_{1\le i<j \le N}(z_j-z_i)  
\end{align*}
and $q_t^{(\nu)}$ is the transition density of the squared Bessel (one-dimensional) diffusion with index $\nu$, see \cite{RevuzYor,SurveyBessel},
\begin{align}\label{BESQdensity}
    q_t^{(\nu)}(x,y)=\frac{1}{2t}\left(\frac{y}{x}\right)^{\frac{\nu}{2}}e^{-\frac{(x+y)}{2t}}I_\nu\left(\frac{\sqrt{xy}}{t}\right),
\end{align}
where $I_\nu$ is the modified Bessel function of the first kind. We also need a final piece of notation. For $\lambda>0$ and $\nu \in \mathbb{R}$ we define the function
\begin{align*}
\phi_{\lambda}^{(\nu)}(x)=(2\lambda x)^{-\frac{\nu}{2}}I_{\nu}\left(\left(2\lambda x\right)^{\frac{1}{2}}\right).
\end{align*}

Let $\mathbf{M}\in \textnormal{Mat}_{\mathbb{C}}(K,N)$ be arbitrary. We are interested in the evolution  of the eigenvalues of the matrix process $\left(\mathbf{X}_t^{\mathbf{M}};t\ge 0\right)=\left(\left(\mathbf{B}_t+t\mathbf{M}\right)^*\left(\mathbf{B}_t+t\mathbf{M}\right);t\ge 0\right)$, equivalently the square singular values of $\left(\mathbf{B}_t+t\mathbf{M};t\ge 0\right)$. For $N=1$ this is precisely the setting of the result of Rogers-Pitman \cite{RogersPitman} for multidimensional Brownian motion with drift (for even real dimension since we consider complex Brownian motion here). We have the following theorem.

\begin{thm}\label{EigenvalueTheorem}
Let $K\ge N$ and consider $\mathbf{M}\in \textnormal{Mat}_{\mathbb{C}}(K,N)$ with $\mu=\mathsf{eval}\left(\mathbf{M}^*\mathbf{M}\right)$. Assume that the matrix process $\left(\mathbf{X}_t^{\mathbf{M}};t\ge 0\right)=\left(\left(\mathbf{B}_t+t\mathbf{M}\right)^*\left(\mathbf{B}_t+t\mathbf{M}\right);t\ge 0\right)$ is started from the zero matrix. Then, its eigenvalue evolution $\left(\mathsf{eval}\left(\mathbf{X}_t^{\mathbf{M}}\right);t\ge 0\right)$ is Markovian with transition density with respect to Lebesgue measure in $\mathbb{W}_N$ given by, where $\nu=K-N$:
\begin{align}\label{CondBesselDriftsTransDens}
 q_t^{(\nu),N,\mathbf{M}}(x,y)=q_t^{(\nu),N,\mu}(x,y)=e^{-\frac{1}{2}\sum_{i=1}^N\mu_i t}\frac{\det\left(\phi^{(\nu)}_{\frac{\mu_i}{2}}(y_j)\right)_{i,j=1}^N}{\det\left(\phi^{(\nu)}_{\frac{\mu_i}{2}}(x_j)\right)_{i,j=1}^N}\det\left(q_t^{(\nu)}(x_i,y_j)\right)_{i,j=1}^N.
\end{align}
\end{thm}

\begin{rmk}
In case of coinciding coordinates for $\mu=\mathsf{eval}\left(\mathbf{M}^*\mathbf{M}\right)$ the ratio of determinants in formula (\ref{CondBesselDriftsTransDens}) is interpreted using L'H\^{o}pital's rule. This is the interpretation we use for ratios of vanishing determinants that we may encounter in the sequel and we do not mention it explicitly again.
\end{rmk}

\begin{rmk}
We will also study the evolution of the matrix process $\left(\mathbf{X}_t^{\mathbf{M}};t\ge 0\right)$ itself in Proposition \ref{MatrixTheorem} in Section \ref{SectionMatrix}. This gives yet another generalization of the Rogers-Pitman theorem \cite{RogersPitman}. In the real symmetric case, namely for the Wishart process, an analogue of this result was proven in \cite{WishartDrifts}. As far as we are aware in the real symmetric case there is no analogue of Theorem \ref{EigenvalueTheorem} for the evolution of the eigenvalues.
\end{rmk}

\subsection{Diffusions conditioned to never intersect}
Let $\delta \ge 2$. 
The squared Bessel process with dimension $\delta$ and index $\nu=\frac{\delta}{2}-1$ is given by the unique strong solution $\left(\mathsf{x}(t);t\ge 0\right)$ of the SDE in $(0,\infty)$:
\begin{align*}
 d\mathsf{x}(t)=2\sqrt{\mathsf{x}(t)}d\mathsf{w}(t)+\delta dt=2\sqrt{\mathsf{x}(t)}d\mathsf{w}(t)+2\left(\nu+1\right) dt,
\end{align*}
where $\mathsf{w}$ is a standard Brownian motion. In this paper we call this the $\textnormal{BESQ}(\delta)$ process. It is well-known, see \cite{RevuzYor,SurveyBessel,PitmanYor}, that $0$ is an entrance boundary point (we can start the process there and is never reached again) for the $\textnormal{BESQ}(\delta)$ process for dimension $\delta\ge 2$. Its transition density with respect to the Lebesgue measure in $(0,\infty)$ is given by (\ref{BESQdensity}).

It is possible to consider a certain generalization of the $\textnormal{BESQ}(\delta)$ process depending on a parameter $\lambda\ge 0$ so that for $\lambda=0$ we get back $\textnormal{BESQ}(\delta)$ process. This can be defined as the Doob $h$-transform \cite{Doob} of the $\textnormal{BESQ}(\delta)$ process by $\phi_\lambda^{(\nu)}$, see for example \cite{Watanabe,PitmanYor}, or equivalently as the unique strong solution $\left(\mathsf{x}^{\lambda}(t);t\ge 0\right)$ to the SDE in $(0,\infty)$:
\begin{align*}
   d\mathsf{x}^{\lambda}(t)=2\sqrt{\mathsf{x}^{\lambda}(t)}d\mathsf{w}(t)+\left[2\left(\nu+1\right)+2\sqrt{2\lambda\mathsf{x}^{\lambda}(t)}\frac{I_{\nu+1}\left(\sqrt{2\lambda\mathsf{x}^{\lambda}(t)}\right)}{I_{\nu}\left(\sqrt{2\lambda\mathsf{x}^{\lambda}(t)}\right)}\right]dt. 
\end{align*}
As before, $0$ is an entrance boundary point for this diffusion (as long as $\delta\ge 2$, equivalently $\nu\ge 0$). We call this the $\textnormal{BESQ}_{\lambda}(\delta)$ process and call $\lambda$ the (generalized) drift parameter. These diffusions (or more precisely their square root analogue) were introduced in a seminal paper \cite{Watanabe} by Watanabe investigating invariance of one-dimensional diffusions under the time inversion $t\rightsquigarrow \frac{1}{t}$. They were called Bessel diffusion processes in the wide sense there. A deep study of these processes was then undertaken by Pitman and Yor in \cite{PitmanYor}. They also come up in the study of exponential functionals of Brownian motion by Matsumoto and Yor \cite{MatsumotoYor}, see also \cite{RiderValko,OConnellPosDef} for some interesting matrix extensions of closely related processes.

In this paper we study the probability of collision of independent $\textnormal{BESQ}_{\lambda_i}(\delta)$ diffusions. More precisely we show that if $N$ independent $\textnormal{BESQ}_{\lambda_i}(\delta)$ diffusions with strictly ordered drift parameters ($\lambda_1<\cdots<\lambda_N$) are started from strictly ordered ($x_1<\cdots<x_N$) initial locations then the probability that they never intersect is strictly positive and given by a very simple explicit formula. This turns out to be a special case of a general phenomenon for a certain class of one-dimensional diffusions that we study in Section \ref{SectionConditioned}. From this we will easily deduce the following interpretation for the eigenvalues of the matrix process considered in Theorem \ref{EigenvalueTheorem}.

\begin{thm}\label{ConditionedBesselProcess}
Let $x,\mu \in \mathbb{W}_N$ and $\delta\ge 2$. Consider $N$ independent $\textnormal{BESQ}_{\frac{\mu_1}{2}}(\delta), \dots, \textnormal{BESQ}_{\frac{\mu_N}{2}}(\delta)$ diffusions conditioned to never intersect starting from $x$. Then, the transition density with respect to Lebesgue measure in $\mathbb{W}_N$ of the conditioned process is given by $q_t^{(\nu),N,\mu}$ from (\ref{CondBesselDriftsTransDens}) where $\nu=\frac{\delta}{2}-1$.
\end{thm}

\begin{rmk}
It is essential for the argument we use that the parameters $\lambda_i$ are strictly ordered. One might expect that the interpretation in Theorem \ref{ConditionedBesselProcess} should still hold when the ordering is not strict. This is indeed the case when the $\lambda_i\equiv 0$ from the results of \cite{OConnell}. On the other hand, when the parameters are no longer ordered the situation is much more involved. In the case of Brownian motions with drifts this has been investigated in \cite{Puchala,Garbit}. It would be interesting to extend these results to the $\textnormal{BESQ}$ case or even more ambitiously to the general setting of Section \ref{SectionConditioned}.
\end{rmk}

Finally, we have the following symmetry between the drift parameters and the starting point coordinates. An analogous result holds for conditioned Brownian motions with drifts \cite{JonesOConnell}. This property appears to be rather special and does not seem to extend to the general setting of one-dimensional diffusions considered in Section \ref{SectionConditioned}.
\begin{cor}
Let $x,\mu \in \mathbb{W}_N$ and $\delta\ge 2$. Then, the distribution at fixed time $t=1$ of $N$ independent $\textnormal{BESQ}_{\frac{\mu_1}{2}}(\delta), \dots, \textnormal{BESQ}_{\frac{\mu_N}{2}}(\delta)$ diffusions conditioned to never intersect starting from $x$ is the same as the distribution of $N$ independent $\textnormal{BESQ}_{\frac{x_1}{2}}(\delta), \dots, \textnormal{BESQ}_{\frac{x_N}{2}}(\delta)$ diffusions conditioned to never intersect starting from $\mu$.
\end{cor}

\begin{proof}
The symmetry $x\leftrightarrow \mu$ for $t=1$, namely that $q_1^{(\nu),N,\mu}(x,y)=q_1^{(\nu),N,x}(\mu,y)$, is immediate from the explicit formula for the transition density, where $\nu=\frac{\delta}{2}-1$:
\begin{align*}
 q_t^{(\nu),N,\mu}(x,y)=\left(\frac{1}{2t}\right)^{N} \exp\left(-\frac{1}{2}\sum_{i=1}^N\left(\mu_it+\frac{x_i+y_i}{t}\right)\right)  \frac{\det\left(I_\nu\left(\sqrt{\mu_iy_j}\right)\right)_{i,j=1}^N \det\left(I_\nu\left(\frac{\sqrt{x_iy_j}}{t}\right)\right)_{i,j=1}^N}{\det\left(I_\nu\left(\sqrt{\mu_ix_j}\right)\right)_{i,j=1}^N}.
\end{align*}
\end{proof}

\subsection{Interacting diffusions construction}

We now give a third interpretation of the Markov process with transition kernel $q_t^{(\nu),N,\mu}$ as the law of the top row of interacting diffusions in some interlacing array. The type\footnote{There is yet another type of dynamics on interlacing arrays which has been intensely studied in the literature. This has its roots in the combinatorial algorithm of the RSK correspondence, its continuous analogue and generalizations, see for example \cite{OConnellYor,OConnellTams,Toda}. As far as we are aware there is no analogue of these dynamics related to the Laguerre diffusion and its generalization studied here.} \footnote{There is also a natural analogue of these dynamics in the discrete setting, see for example \cite{WarrenWindridge,BorodinFerrari,BorodinOlshanski}.} of dynamics we will consider here involving Brownian motions, interacting via local time terms, only when they collide, in a triangular interlacing array (or continuous Gelfand-Tsetlin pattern) were first introduced and studied by Warren in \cite{Warren}. These are related, but not identical, to the dynamics of eigenvalues of submatrices of Hermitian Brownian motion, see \cite{Warren,Adler,FerrariFringsPartial}. For squared Bessel diffusions (without the generalized drifts) analogous dynamics in triangular interlacing arrays were studied in \cite{InterlacingDiffusions,Sun}. These are related, but again not identical, to the evolution of eigenvalues of submatrices of the Laguerre diffusion $\mathbf{B}_t^*\mathbf{B}_t$. In the Brownian case one could also add drifts to these interacting Brownian motions in the triangular interlacing array and still have analogous results, see \cite{FerrariFrings,InterlacingDiffusions}. This does not appear to be possible in the squared Bessel case however; in particular none of the natural candidates of how to involve these generalized drifts in a triangular array seem to work. Instead, one needs to consider slightly more complicated but still rather natural dynamics in so-called (interlacing) half arrays\footnote{In the discrete setting, when the entries are integers, these half arrays are closely related to symplectic and orthogonal Gelfand-Tsetlin patterns (these are essentially special cases satisfying some extra conditions).} that we define next. 

We need some notation and terminology. We will say that $x\in \mathbb{W}_N$ interlaces with $y\in \mathbb{W}_{N+1}$, and denote this by $x\prec y$, if the following inequalities hold
\begin{align*}
y_1\le x_1 \le y_2 \le x_2 \le \cdots \le y_N\le x_N \le y_{N+1}.    
\end{align*}
We will say that $x\in \mathbb{W}_N$ interlaces with $y\in \mathbb{W}_{N}$, and abusing notation still denote this by $x\prec y$, if the following holds
\begin{align*}
x_1\le y_1 \le x_2 \le y_2 \le \cdots \le x_N \le y_N.    
\end{align*}
We then define the space of (interlacing) half arrays of length $N$, see for example Figure \ref{Figure1} for an illustration,
\begin{align*}
 \mathfrak{HA}_N=\left\{(x^{(1)},\dots,x^{(N)}):x^{(i)}\in \mathbb{W}_{\left\lfloor\frac{i+1}{2} \right\rfloor}, \ x^{(i)}\prec x^{(i+1)} \right\}.
\end{align*}
We call the $x^{(i)}$'s the rows or levels of the array. 

\begin{figure}
\captionsetup{singlelinecheck = false, justification=justified}
\begin{tikzpicture}
 
     \draw[fill] (0,0) circle [radius=0.1];   
     \draw[fill] (0.75,0.75) circle [radius=0.1];  
      \draw[fill] (0,1.5) circle [radius=0.1]; 
    \draw[fill] (1.5,1.5) circle [radius=0.1]; 
          \draw[fill] (0.75,2.25)circle [radius=0.1]; 
    \draw[fill] (2.25,2.25) circle [radius=0.1]; 
              \draw[fill] (0,3) circle [radius=0.1]; 
    \draw[fill] (1.5,3) circle [radius=0.1]; 
          \draw[fill] (3,3)circle [radius=0.1]; 
          
    \node[below right] at (0,0) {$x_1^{(1)}$};
    \node[below right] at (0.75,0.75) {$x_1^{(2)}$};
     \node[below right] at (0,1.5) {$x_1^{(3)}$};
     \node[below right] at (1.5,1.5) {$x_2^{(3)}$};
    \node[below right] at (0.75,2.25) {$x_1^{(4)}$};
     \node[below right] at (2.25,2.25) {$x_2^{(4)}$};
       \node[below right] at (0,3) {$x_1^{(5)}$};
     \node[below right] at (1.5,3) {$x_2^{(5)}$};
  \node[below right] at (3,3) {$x_3^{(5)}$};

 \end{tikzpicture}
 \centering
 \caption{A cartoon representing an element of $\mathfrak{HA}_5$.}\label{Figure1}
 \end{figure}
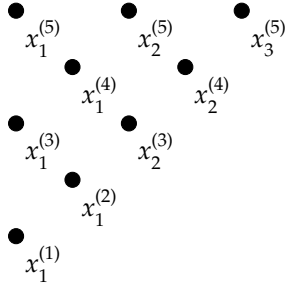

We now consider the following system of reflecting SDEs in $\mathfrak{HA}_{2N-1}$
\begin{align}\label{DynamicsBessel}
d\mathsf{x}_i^{(k)}(t)&=2\sqrt{\mathsf{x}_i^{(k)}(t)}d\mathsf{w}^{(k)}_i(t)+\mathcal{B}^{(k),\nu}\left(\mathsf{x}_i^{(k)}(t)\right)dt+\frac{1}{2}d\mathfrak{l}^{(k),+}_i(t)-\frac{1}{2}d\mathfrak{l}^{(k),-}_i(t),
\end{align}
with $1\le k \le 2N-1$, $1\le i \le \left \lfloor \frac{k+1}{2}\right \rfloor$ and where the $\mathsf{w}_i^{(k)}$ are independent standard Brownian motions, the level dependent drift term $\mathcal{B}^{(k),\nu}$ is given by
\begin{align*}
\mathcal{B}^{(k),\nu}(x)=\begin{cases}
2(\nu+1)+2\sqrt{\mu_n x}\frac{I_{\nu+1}\left(\sqrt{\mu_n x}\right)}{I_\nu\left(\sqrt{\mu_n x}\right)}, &\textnormal{ if }k=2n-1,\\
-2\nu-2\sqrt{\mu_{n+1} x}\frac{I_{\nu+1}\left(\sqrt{\mu_{n+1} x}\right)}{I_\nu\left(\sqrt{\mu_{n+1} x}\right)},  &\textnormal{ if }k=2n,
\end{cases}
\end{align*}
while the finite variation terms $\mathfrak{l}_i^{(k),\pm }$, increasing only when particles from consecutive levels collide in order to keep them ordered and satisfy the interlacing, can be identified with the semimartingale local times:
\begin{align*}
\mathfrak{l}_i^{(k),+}=\begin{cases}
\textnormal{ sem. loc. time of } \mathsf{x}_i^{(k)}-\mathsf{x}^{(k-1)}_{i-1} \textnormal{ at 0 },&\textnormal{ if }k=2n-1,\\
\textnormal{ sem. loc. time of } \mathsf{x}_i^{(k)}-\mathsf{x}^{(k-1)}_{i}\textnormal{ at 0 },&\textnormal{ if } k=2n,
\end{cases}\\
\mathfrak{l}_i^{(k),-}=\begin{cases}
\textnormal{ sem. loc. time of } \mathsf{x}_i^{(k)}-\mathsf{x}^{(k-1)}_{i}\textnormal{ at 0 },&\textnormal{ if }k=2n-1,\\
\textnormal{ sem. loc. time of } \mathsf{x}_i^{(k)}-\mathsf{x}^{(k-1)}_{i+1}\textnormal{ at 0 },&\textnormal{ if } k=2n.
\end{cases}
\end{align*}
The terms for which the indices underflow or overflow above are identically zero. See Section \ref{SectionInteractingDiffusions}, also Figure \ref{Figure2} therein, for more details on these dynamics.

The reader will notice that if for a moment we forget the local time interactions when particles collide then the SDEs on odd levels $2n-1$ are simply those of $n$ independent $\textnormal{BESQ}_{\frac{\mu_n}{2}}(\delta)$ processes. The SDEs on even levels $2n$ on the other hand might seem a bit mysterious, however as we see in Section \ref{SectionInteractingDiffusions} they are rather natural as they describe the so-called dual diffusions to the $\textnormal{BESQ}_{\frac{\mu_{n+1}}{2}}(\delta)$ process.

We can now state the following theorem. This is a special case of an analogous result, presented in Section \ref{SectionProofInteracting}, for a general class of one-dimensional diffusions for which the $\textnormal{BESQ}(\delta)$ process is the prototypical example.

\begin{thm}\label{InteractingDiffusionsThm}
Let $N\ge 1$ and $\mu_1<\cdots<\mu_N$. Let $\delta\ge 2$ with $\nu=\frac{\delta}{2}-1$. Suppose $\mathcal{M}$ is a probability measure on $\mathbb{W}_N$ and let $\mathfrak{GM}^{\mathcal{M}}_{2N-1}$ be the probability measure on $\mathfrak{HA}_{2N-1}$ given in Definition \ref{DefinitionGibbs}. Assume that the SDEs (\ref{DynamicsBessel}) are initialized according to $\mathfrak{GM}^{\mathcal{M}}_{2N-1}$. Then, the projection on the top level $\left(\mathsf{x}^{(2N-1)};t\ge 0\right)$ is distributed as a Markov process, initialized according to $\mathcal{M}$, having transition density $q_t^{(\nu),N,\mu}$ from (\ref{CondBesselDriftsTransDens}).
\end{thm}

\begin{rmk}
In fact, information is available about the evolution of each level and also about the distribution of the whole array at a fixed time $T\ge 0$, see Section \ref{SectionProofInteracting} for more details.
\end{rmk}

\begin{rmk}\label{RemarkDegenerate}
It is also possible to consider, in the general setting of one-dimensional diffusions, a degenerate case when all the parameters are zero, see Section \ref{SectionDegenerate} for more details. A discrete analogue for birth and death chains appeared in \cite{Randomgrowth}.
\end{rmk}

\begin{rmk}\label{RemarkEntrance}
It is also possible to start the dynamics from the origin, when all the coordinates coincide, using an entrance law, see Section \ref{SectionEntrance} for more details.
\end{rmk}

\begin{rmk}\label{RemarkEdge}
Observe that, see Figure \ref{Figure2} in Section \ref{SectionProofInteracting} for an illustration, the rightmost coordinates $\left(\mathsf{x}^{(1)}_1,\mathsf{x}_1^{(2)},\dots,\mathsf{x}^{(2N-2)}_{N-1},\mathsf{x}^{(2N-1)}_{N}\right)$ in the SDEs form an autonomous interacting particle system. This is the analogue in this setting of the well-known model of Brownian motions with one-sided collisions, equivalently Brownian last passage percolation. The result above then implies the analogues of classical results due to Baryshnikov, Gravner-Tracy-Widom, O'Connell-Yor and Bougerol-Jeulin for the Brownian model, see for example \cite{Baryshnikov,GTW,OConnellYor,BougerolJeulin,OConnellTams,BBO,Warren} and the references therein. We discuss this in more detail in Section 
\ref{SectionEdge}.
\end{rmk}

\paragraph{Organisation of the paper} In Section \ref{SectionMatrix} we prove Theorem \ref{EigenvalueTheorem}. In Section \ref{SectionConditioned} we prove Theorem \ref{ConditionedBesselProcess} in the more general setting of a class of one-dimensional diffusions. Then, in Section \ref{SectionInteractingDiffusions} we prove Theorem \ref{InteractingDiffusionsThm} in a more general setting of a class of interacting one-dimensional diffusions. In Section \ref{FurtherInteractingSection} we elaborate on the results mentioned in Remarks \ref{RemarkDegenerate}-\ref{RemarkEdge}. The arguments in Sections \ref{SectionMatrix}, \ref{SectionConditioned} and \ref{SectionInteractingDiffusions}/\ref{FurtherInteractingSection} (together) are essentially independent of each other and can be read separately. 

\paragraph{Acknowledgements} I am very grateful to Jon Warren and Neil O'Connell for useful discussions. I am also very grateful to two anonymous referees for a careful reading of the paper and many useful comments and suggestions which have improved the paper.

\section{The matrix process and its eigenvalues proofs}\label{SectionMatrix}

We now prove Theorem \ref{EigenvalueTheorem}. The idea is the following. We want to relate the eigenvalues of $\left(\mathbf{X}_t^{\mathbf{M}};t\ge 0\right)$ to the eigenvalues of $\left(\mathbf{X}_t^{\mathbf{0}};t\ge 0\right)$. Towards this end we apply Girsanov's theorem on the matrix level. By invariance of matrix Brownian motion, and since we are looking at the eigenvalues, the formula we obtain remains true even if we multiply $\mathbf{B}_t$ to the left and to the right by arbitrary unitary matrices and integrate over the corresponding groups with respect to Haar measure. Remarkably this matrix integral is known to have an explicit evaluation which depends on $\mathbf{B}_t$ only through the eigenvalues of $\mathbf{B}_t^*\mathbf{B}_t=\mathbf{X}^{\mathbf{0}}_t$ and the desired conclusion follows.

\begin{proof}[Proof of Theorem \ref{EigenvalueTheorem}] We start by applying Girsanov's theorem \cite{RevuzYor} with a non-negative functional $\mathcal{F}$:
\begin{align*}
 \mathbb{E}\left[\mathcal{F}\left(\mathbf{X}_s^{\mathbf{M}};s\le t\right)\right]=  \mathbb{E}\left[\mathcal{F}\left(\mathbf{X}_s^{\mathbf{0}};s\le t\right)\exp\left(\frac{1}{2}\textnormal{Tr}\left(\mathbf{M}\mathbf{B}^*_t+\mathbf{B}_t\mathbf{M}^*\right)\right)\exp\left(-\frac{1}{2}\textnormal{Tr}\left(\mathbf{M}^*\mathbf{M}\right)t\right)\right],
\end{align*}
since we have $\textnormal{Tr}\left(\mathbf{M}^*\mathbf{M}\right)=\sum_{i=1}^K\sum_{j=1}^N\left(\Re\mathbf{M}_{ij}\right)^2+\left(\Im\mathbf{M}_{ij}\right)^2$ and also
\begin{align*}
 \frac{1}{2}\textnormal{Tr}\left(\mathbf{M}\mathbf{B}^*_t+\mathbf{B}_t\mathbf{M}^*\right)=\sum_{i=1}^K\sum_{j=1}^N \Re \mathbf{M}_{ij}\mathsf{w}_{ij}(t)+\Im \mathbf{M}_{ij}\tilde{\mathsf{w}}_{ij}(t).
\end{align*}
We now make use of unitary invariance of matrix Brownian motion, recall that $\left(\mathbf{B}_t;t\ge 0\right)$ is starting from the zero matrix, to get that, where $\mathbb{U}(n)$ is the group of $n\times n$ unitary matrices
\begin{align*}
\left(\mathbf{U}\mathbf{B}_t;t\ge 0\right)&\overset{\textnormal{d}}{=}\left(\mathbf{B}_t;t\ge 0\right), \ \ \forall \  \mathbf{U}\in \mathbb{U}(K),\\
\left(\mathbf{B}_t\mathbf{V};t\ge 0\right)&\overset{\textnormal{d}}{=}\left(\mathbf{B}_t;t\ge 0\right), \ \ \forall \ \mathbf{V}\in \mathbb{U}(N).
\end{align*}
Thus, for any $\mathbf{U}\in \mathbb{U}(K), \mathbf{V}\in \mathbb{U}(N)$ we have
\begin{align*}
 \mathbb{E}\left[\mathcal{F}\left(\mathbf{X}_s^{\mathbf{M}};s\le t\right)\right]=e^{-\frac{1}{2}\textnormal{Tr}\left(\mathbf{M}^*\mathbf{M}\right)t}\mathbb{E}\left[\mathcal{F}\left(\mathbf{V}^*\mathbf{X}_s^{\mathbf{0}}\mathbf{V};s\le t\right)\exp\left(\frac{1}{2}\textnormal{Tr}\left(\mathbf{M}\mathbf{V}^*\mathbf{B}^*_t\mathbf{U}^*+\mathbf{U}\mathbf{B}_t\mathbf{V}\mathbf{M}^*\right)\right)\right].
\end{align*}

Moreover, since we are interested in the evolution of the eigenvalues, suppose that $\mathcal{F}$ is any non-negative functional that only depends on the matrix through its spectrum. In particular, for such an $\mathcal{F}$ we have, for any $\mathbf{U}\in \mathbb{U}(K), \mathbf{V}\in \mathbb{U}(N)$, since $\mathsf{eval}\left(\mathbf{V}^*\mathbf{X}_s^{\mathbf{0}}\mathbf{V}\right)=\mathsf{eval}\left(\mathbf{X}_s^{\mathbf{0}}\right)$:
\begin{align*}
  \mathbb{E}\left[\mathcal{F}\left(\mathsf{eval}\left(\mathbf{X}_s^{\mathbf{M}}\right);s\le t\right)\right]=e^{-\frac{1}{2}\textnormal{Tr}\left(\mathbf{M}^*\mathbf{M}\right)t}\mathbb{E}\bigg[&\mathcal{F}\left(\mathsf{eval}\left(\mathbf{X}_s^{\mathbf{0}}\right);s\le t\right)\\&\exp\left(\frac{1}{2}\textnormal{Tr}\left(\mathbf{M}\mathbf{V}^*\mathbf{B}^*_t\mathbf{U}^*+\mathbf{U}\mathbf{B}_t\mathbf{V}\mathbf{M}^*\right)\right)\bigg].
\end{align*}
Since this formula holds for arbitrary $\mathbf{U}\in \mathbb{U}(K), \mathbf{V}\in \mathbb{U}(N)$, it remains true even if we integrate over these unitary groups with respect to the corresponding Haar probability measures. In probabilistic terms we take $\mathbf{U}, \mathbf{V}$ independent of each other and of the Brownian matrix $\mathbf{B}_t$ and distributed according to the Haar measure on $\mathbb{U}(K)$ and $\mathbb{U}(N)$ respectively and take expectations. We then use Tonelli's theorem to bring these expectations inside and the resulting double matrix integral can in fact be computed explicitly.

Namely, let $\mathbf{A},\mathbf{C}\in \textnormal{Mat}_{\mathbb{C}}\left(K,N\right)$ with $K\ge N$ and write $a=\mathsf{eval}\left(\mathbf{A}^*\mathbf{A}\right)$ and $c=\mathsf{eval}\left(\mathbf{C}^*\mathbf{C}\right)$. Then, we have the following explicit evaluation of the matrix integral, see \cite{HCRectangular1,HCRectangular2,HCRectangular3},
\begin{align*}
\int_{\mathbb{U}(K)}^{}\int_{\mathbb{U}(N)}^{}\exp\big(\textnormal{Tr}\left(\mathbf{A}\mathbf{V}^*\mathbf{C}^*\mathbf{U}^*+\mathbf{U}\mathbf{C}\mathbf{V}\mathbf{A}^*\right)\big)d\mathbf{V}d\mathbf{U}=\frac{\prod_{p=1}^{N-1}p! \prod_{q=1}^{K-1}q!\det\big(I_{K-N}\big(2\sqrt{a_ic_j}\big)\big)_{i,j=1}^N}{\prod_{r=1}^{K-N-1}r!\mathsf{\Delta}_N(a)\mathsf{\Delta}_N(c)\prod_{i=1}^{N}(a_ic_i)^{\frac{K-N}{2}}},
\end{align*}
where $d\mathbf{V},d\mathbf{U}$ denote the Haar probability measure on $\mathbb{U}(N)$ and $\mathbb{U}(K)$ respectively. 

Hence, we obtain the following formula
\begin{align}
  \mathbb{E}\left[\mathcal{F}\left(\mathsf{eval}\left(\mathbf{X}_s^{\mathbf{M}}\right);s\le t\right)\right]=e^{-\frac{1}{2}\sum_{i=1}^N\mu_i t}\mathbb{E}\left[\mathcal{F}\left(\mathsf{eval}\left(\mathbf{X}_s^{\mathbf{0}}\right);s\le t\right)\mathsf{const}(K;N;\mu)\mathsf{\Phi}_{\mu}\left(\mathsf{eval}\left(\mathbf{X}_t^{\mathbf{0}}\right)\right)\right],
\end{align}
where $\mu=\mathsf{eval}\left(\mathbf{M}^*\mathbf{M}\right)$ and the function $\mathsf{\Phi}_\mu(x)$ and $\mathsf{const}(K;N;\mu)$ are given by, with $\nu=K-N$:
\begin{align*}
\mathsf{\Phi}_\mu(x)&=\frac{\det\left(\phi^{(\nu)}_{\frac{\mu_i}{2}}\left(x_j\right)\right)_{i,j=1}^N}{\mathsf{\Delta}_N\left(x\right)},\\
 \mathsf{const}(K;N;\mu)&= \frac{\prod_{p=1}^{N-1}p! \prod_{q=1}^{K-1}q!2^{N(K-1)}}{\prod_{r=1}^{K-N-1}r!\mathsf{\Delta}_N(\mu)}.
\end{align*}
Thus, $\left(\mathsf{eval}\left(\mathbf{X}_t^{\mathbf{M}}\right);t\ge 0\right)$ can be obtained as a Doob $h$-transform \cite{Doob} from $\left(\mathsf{eval}\left(\mathbf{X}_t^{\mathbf{0}}\right);t\ge 0\right)$. In particular, the evolution of $\left(\mathsf{eval}\left(\mathbf{X}_s^{\mathbf{M}}\right);s\le t\right)$ is Markovian with transition density with respect to the Lebesgue measure in $\mathbb{W}_{N}$ given by the Doob $h$-transformed density
\begin{align*}
 q_t^{(\nu),N,\mathbf{M}}(x,y)&=e^{-\frac{1}{2}\sum_{i=1}^N\mu_i t} \frac{\det\left(\phi^{(\nu)}_{\frac{\mu_i}{2}}(y_j)\right)_{i,j=1}^N\mathsf{\Delta}_N(x)}{\det\left(\phi^{(\nu)}_{\frac{\mu_i}{2}}(x_j)\right)_{i,j=1}^N\mathsf{\Delta}_N(y)}  q_t^{(\nu),N}(x,y)\\
 &=e^{-\frac{1}{2}\sum_{i=1}^N\mu_i t}\frac{\det\left(\phi^{(\nu)}_{\frac{\mu_i}{2}}(y_j)\right)_{i,j=1}^N}{\det\left(\phi^{(\nu)}_{\frac{\mu_i}{2}}(x_j)\right)_{i,j=1}^N}\det\left(q_t^{(\nu)}(x_i,y_j)\right)_{i,j=1}^N.
\end{align*}
\end{proof}

We now go on to prove a result for the matrix process $\left(\mathbf{X}_t^{\mathbf{M}};t\ge 0\right)$ itself. Towards this end denote by $\mathsf{P}_t^{K,N}(\mathbf{X},d\mathbf{Y})$ the transition kernel of $\left(\mathbf{X}_t^{\mathbf{0}};t\ge 0\right)$. This has been computed explicitly in \cite{Demni} and takes the following form (we will not need to use the explicit formula here)
\begin{align*}
\mathsf{P}_t^{K,N}(\mathbf{X},d\mathbf{Y})=\frac{1}{(2\pi)^{NK}\Gamma_N(K)}\exp \left(-\frac{1}{2t}\textnormal{Tr}\left(\mathbf{X}+\mathbf{Y}\right)\right) \left(\det \mathbf{Y}\right)^{K-N} {}_0\textnormal{F}^{(1)}_1\left(K;\frac{\mathbf{X}\mathbf{Y}}{4t^2}\right) \mathbf{1}_{\left(\mathbf{Y}\in \textnormal{Her}^+_\mathbb{C}(N)\right)}d\mathbf{Y},
\end{align*}
where ${}_0\textnormal{F}^{(1)}_1$ is the complex matrix hypergeometric function and $\Gamma_N(K)$ is the multivariate Gamma function, see \cite{Demni,GrossRichards} for more details, while $d\mathbf{Y}$ is the Lebesgue measure on Hermitian matrices. We also define, for a fixed $\mathbf{M}\in \textnormal{Mat}_{\mathbb{C}}(K,N)$, the following function $\mathsf{H}_{\mathbf{M}}$ on $\textnormal{Her}^+_\mathbb{C}(N)$,
\begin{align*}
 \mathsf{H}_{\mathbf{M}}\left(\mathbf{X}\right)=\frac{\det\left(\left[\mathsf{eval}_i\left(\mathbf{M}\mathbf{X}\mathbf{M}^*\right)\right]^{\frac{j-1}{2}}I_{j-1}\left(\sqrt{\mathsf{eval}_i\left(\mathbf{M}\mathbf{X}\mathbf{M}^*\right)}\right)\right)_{i,j=1}^K}{\mathsf{\Delta}_K\left(\mathsf{eval}\left(\mathbf{M}\mathbf{X}\mathbf{M}^*\right)\right)}.
\end{align*}
Then, we have the following proposition.

\begin{prop}\label{MatrixTheorem}
Let $K\ge N$ and suppose $\mathbf{M}\in \textnormal{Mat}_{\mathbb{C}}(K,N)$. Then, the matrix process $\left(\mathbf{X}_t^{\mathbf{M}};t\ge 0\right)=\left(\left(\mathbf{B}_t+t\mathbf{M}\right)^*\left(\mathbf{B}_t+t\mathbf{M}\right);t\ge 0\right)$, if started from the zero matrix, is a Markov process with transition kernel given by
\begin{align}\label{MatrixTransDens}
\mathsf{P}_t^{K,N,\mathbf{M}}(\mathbf{X},d\mathbf{Y})=\exp{\left(-\frac{1}{2}\textnormal{Tr}\left(\mathbf{M}^*\mathbf{M}\right)t\right)}\frac{\mathsf{H}_{\mathbf{M}}(\mathbf{Y})}{\mathsf{H}_{\mathbf{M}}(\mathbf{X})}\mathsf{P}_t^{K,N}(\mathbf{X},d\mathbf{Y}).
\end{align}
\end{prop}

\begin{rmk}
It might have been possible to obtain Theorem \ref{EigenvalueTheorem} directly from Proposition \ref{MatrixTheorem} by integrating out the eigenvectors, but this seems to involve some more complicated computations compared to the proof we gave above. For the special case $N=1$ the two formulae (\ref{MatrixTransDens}) and (\ref{CondBesselDriftsTransDens}) are, as they should be, exactly equal. This can be seen as follows. First, note that from linear algebra
\begin{align*}
  \left(\mathsf{eval}_1\left(\mathbf{M}\mathbf{X}\mathbf{M}^*\right),\dots,\mathsf{eval}_K\left(\mathbf{M}\mathbf{X}\mathbf{M}^*\right)\right)=\left(0,\dots,0,\mathsf{eval}_1\left(\mathbf{M}^*\mathbf{M}\mathbf{X}\right),\dots,\mathsf{eval}_N\left(\mathbf{M}^*\mathbf{M}\mathbf{X}\right)\right).
\end{align*}
Then, using the explicit formula $x^{\frac{j-1}{2}}I_{j-1}\left(\sqrt{x}\right)=\sum_{m=0}^\infty \frac{1}{m!(m+j-1)!2^{2m+j}}x^{m+j-1}$ and L'H\^{o}pital's rule, after some computations, we obtain that $\mathsf{H}_{\mathbf{M}}\left(\mathbf{X}\right)$ is equal, up to a multiplicative constant independent of $\mathbf{X}$, to
\begin{align*}
 \tilde{\mathsf{H}}_{\mathbf{M}}\left(\mathbf{X}\right)=\frac{\det\left(\left[\mathsf{eval}_i\left(\mathbf{M}^*\mathbf{M}\mathbf{X}\right)\right]^{\frac{N-K+j-1}{2}}I_{K-N+j-1}\left(\sqrt{\mathsf{eval}_i\left(\mathbf{M}^*\mathbf{M}\mathbf{X}\right)}\right)\right)_{i,j=1}^N}{\mathsf{\Delta}_N\left(\mathsf{eval}\left(\mathbf{M}^*\mathbf{M}\mathbf{X}\right)\right)},
\end{align*}
which gives the equality of the two formulae for $N=1$.
\end{rmk}
\begin{proof}[Proof of Proposition \ref{MatrixTheorem}]
It would be possible to prove this result by adapting the arguments from the proof of Theorem \ref{EigenvalueTheorem} above; the main difference is that one multiplies $\mathbf{B}_t$ only by a single Haar distributed matrix $\mathbf{U}\in \mathbb{U}(K)$. We will instead establish Proposition \ref{MatrixTheorem} by adapting the proof of Rogers and Pitman \cite{RogersPitman} for multidimensional Brownian motion, as each step in that proof has a matrix analogue which might be of independent interest.

Consider the function $\Phi:\textnormal{Mat}_{\mathbb{C}}(K,N) \to \textnormal{Her}^+_\mathbb{C}(N)$ given by $\Phi:\mathbf{A}\mapsto \mathbf{A}^*\mathbf{A}$. Moreover, define the following set, for a fixed $\mathbf{Y}\in \textnormal{Her}^+_\mathbb{C}(N)$:
\begin{align*}
 \mathfrak{R}_K\left(\mathbf{Y}\right)=\left\{\mathbf{A}\in \textnormal{Mat}_{\mathbb{C}}(K,N):\Phi\left(\mathbf{A}\right)=\mathbf{A}^*\mathbf{A}=\mathbf{Y}\right\}.
\end{align*}
Consider the following Markov kernel $\mathsf{\Lambda}$ from $\textnormal{Her}^+_\mathbb{C}(N)$ to $\mathfrak{R}_K\left(\mathbf{Y}\right)$, the uniform measure on $\mathfrak{R}_K\left(\mathbf{Y}\right)$, given by
\begin{align*}
    \mathsf{\Lambda}\left(\mathbf{Y},d\mathbf{A}\right)=\mathsf{Law}\left[\mathbf{U}\mathbf{A}_{\star}\right](d\mathbf{A}), \  \textnormal{ where } \mathbf{A}_{\star} \textnormal{ is any representative of } \mathfrak{R}_K\left(\mathbf{Y}\right),
\end{align*}
and $\mathbf{U}$ is Haar distributed on $\mathbb{U}(K)$. Note that this is well-defined, namely independent of the choice of representative $\mathbf{A}_\star$, since for any $\mathbf{A}_1,\mathbf{A}_2\in \mathfrak{R}_K\left(\mathbf{Y}\right)$ there exists a $\mathbf{Q}\in \mathbb{U}(K)$ so that $\mathbf{A}_1=\mathbf{Q}\mathbf{A}_2$. 

Observe that $\mathsf{\Lambda}\circ \Phi=\mathsf{Id}$, the identity kernel on $\textnormal{Her}^+_\mathbb{C}(N)$. Also, by unitary invariance of matrix Brownian motion we have\footnote{Note that by unitary invariance we have:
\begin{align*}
  \mathsf{Law}\left(\mathbf{B}_t|\Phi(\mathbf{B}_t)=\mathbf{Y}\right)(d\mathbf{H})=\mathsf{\Lambda}(\mathbf{Y},d\mathbf{H}), \ \ \forall t \ge 0.
\end{align*}
See Lemma 1 in \cite{RogersPitman} on how this also essentially implies the intertwining.}, where $\mathsf{Q}_t^{K,N}$ is the semigroup of $\mathbf{B}_t$ on $\textnormal{Mat}_{\mathbb{C}}(K,N)$:
\begin{align*}
\mathsf{P}_t^{K,N}\mathsf{\Lambda}=\mathsf{\Lambda}\mathsf{Q}_t^{K,N}, \ \ \forall t \ge 0.
\end{align*}
Now, for a fixed $\mathbf{M}\in \textnormal{Mat}_{\mathbb{C}}(K,N)$ define the following function $\mathsf{h}_{\mathbf{M}}$ on $\textnormal{Mat}_{\mathbb{C}}(K,N)$
\begin{align*}
\mathsf{h}_{\mathbf{M}}(\mathbf{X})=\exp\left(\frac{1}{2}\textnormal{Tr}\left(\mathbf{M}\mathbf{X}^*+\mathbf{X}\mathbf{M}^*\right)\right).
\end{align*}
We consider the Doob $h$-transform $\mathsf{Q}_t^{K,N,\mathbf{M}}$ of $\mathsf{Q}_t^{K,N}$ by $\mathsf{h}_{\mathbf{M}}$, which is simply the semigroup of $\mathbf{B}_t+t\mathbf{M}$ on $\textnormal{Mat}_{\mathbb{C}}(K,N)$,
\begin{align*}
\mathsf{Q}_t^{K,N,\mathbf{M}}\left(\mathbf{X},d\mathbf{Y}\right)=\exp{\left(-\frac{1}{2}\textnormal{Tr}\left(\mathbf{M}^*\mathbf{M}\right)t\right)}\frac{\mathsf{h}_{\mathbf{M}}(\mathbf{Y})}{\mathsf{h}_{\mathbf{M}}(\mathbf{X})}\mathsf{Q}_t^{K,N}\left(\mathbf{X},d\mathbf{Y}\right).
\end{align*}
We now need the explicit evaluation of the following matrix integral, the partition function of the so-called Brezin-Gross-Witten model, see \cite{BrezinGross,GrossWitten}, with $\mathbf{C}\in \textnormal{Mat}_{\mathbb{C}}(K,K)$:
\begin{align*}
\int_{\mathbb{U}(K)}^{}\exp\left(\frac{1}{2}\textnormal{Tr}\left(\mathbf{C}^*\mathbf{U}^*+\mathbf{U}\mathbf{C}\right)\right)d\mathbf{U}=2^{\frac{K(K-1)}{2}}\left[\prod_{j=1}^{K-1}j!\right]\frac{\det\left(c_i^{\frac{j-1}{2}}I_{j-1}\left(\sqrt{c_i}\right)\right)_{i,j=1}^K}{\mathsf{\Delta}_K(c)},
\end{align*}
where $c=\mathsf{eval}\left(\mathbf{C}^*\mathbf{C}\right)$. 
Thus, we have
\begin{align*}
\left[\mathsf{\Lambda}\mathsf{h}_{\mathbf{M}}\right](\mathbf{Y}) &=\int_{\mathbb{U}(K)}^{}\exp\left(\frac{1}{2}\textnormal{Tr}\left(\mathbf{M}\mathbf{A}_{\star}^*\mathbf{U}^*+\mathbf{U}\mathbf{A}_{\star}\mathbf{M}^*\right)\right)d\mathbf{U}=2^{\frac{K(K-1)}{2}}\left[\prod_{j=1}^{K-1}j!\right]\mathsf{H}_{\mathbf{M}}\left(\mathbf{Y}\right).
\end{align*}
Hence if we define the following Markov kernel $\mathsf{\Lambda}$ from $\textnormal{Her}^+_\mathbb{C}(N)$ to $\mathfrak{R}_K\left(\mathbf{Y}\right)$ by
\begin{align*}
    \mathsf{\Lambda}^{\mathbf{M}}\left(\mathbf{Y},d\mathbf{A}\right)=\frac{\mathsf{h}_{\mathbf{M}}(\mathbf{A})}{\left[\mathsf{\Lambda}\mathsf{h}_{\mathbf{M}}\right](\mathbf{Y})}\mathsf{\Lambda}(\mathbf{Y},d\mathbf{A}),
\end{align*}
we have, see Lemma 2 in \cite{RogersPitman}:
\begin{align*}
    \mathsf{P}_t^{K,N,\mathbf{M}}\mathsf{\Lambda}^{\mathbf{M}}=\mathsf{\Lambda}^{\mathbf{M}}\mathsf{Q}_t^{K,N,\mathbf{M}}, \ \ \forall t \ge 0.
\end{align*}
Then Proposition \ref{MatrixTheorem} follows from Theorem 2 in \cite{RogersPitman}.

\end{proof}

\begin{rmk} Both proofs presented in this section can easily be adapted to study the eigenvalues of Hermitian Brownian motion with a drift; essentially the second argument is the one followed in \cite{Chin} where this result was first proven. Namely, consider the Hermitian Brownian motion $\mathbf{H}_t=\frac{1}{2}\left(\mathbf{B}_t+\mathbf{B}_t^*\right)$ on $\textnormal{Her}_\mathbb{C}(N)$, where $\mathbf{B}_t$ is the complex Brownian motion on $\textnormal{Mat}_\mathbb{C}\left(N,N\right)$ starting from the zero matrix, and let $\mathbf{M}\in \textnormal{Her}_\mathbb{C}(N)$ with $\mu=\mathsf{eval}\left(\mathbf{M}\right)$. The two key ingredients one needs are the unitary invariance of Hermitian Brownian motion
\begin{align*}
\left(\mathbf{U}\mathbf{H}_t\mathbf{U}^*;t\ge 0\right)&\overset{\textnormal{d}}{=}\left(\mathbf{H}_t;t\ge 0\right), \ \ \forall \  \mathbf{U}\in \mathbb{U}(N)
\end{align*}
and the Harish-Chandra (also called Itzykson-Zuber) matrix integral \cite{HarishChandra}
\begin{align*}
    \int_{\mathbb{U}(N)}^{}\exp\left(\textnormal{Tr}\left(\mathbf{A}\mathbf{U}\mathbf{C}\mathbf{U}^*\right)\right)d\mathbf{U}=\left[\prod_{p=1}^{N-1}p!\right]\frac{ \det\left(\exp(a_ic_j)\right)_{i,j=1}^N}{\mathsf{\Delta}_N(a)\mathsf{\Delta}_N(c)},
\end{align*}
where $\mathbf{A},\mathbf{C}\in \textnormal{Her}_\mathbb{C}(N)$ with $a=\mathsf{eval}\left(\mathbf{A}\right), c=\mathsf{eval}\left(\mathbf{C}\right)$. 
Then, we get that the evolution of the eigenvalues $\left(\mathsf{eval}\left(\mathbf{H}_t+t\mathbf{M}\right);t\ge 0\right)$ is Markovian with transition density given by
\begin{align*}
e^{-\frac{1}{2}\sum_{i=1}^{N}\mu_i^2t}\frac{\det\big(\exp(\mu_jy_i)\big)_{i,j=1}^N}{\det\big(\exp(\mu_jx_i)\big)_{i,j=1}^N}\det\big(s_t(y_j-x_i)\big)_{i,j=1}^N,
\end{align*}
where $s_t(z)=\frac{1}{\sqrt{2\pi t}}\exp\left(-\frac{z^2}{2t}\right)$ is the standard heat kernel. When $\mu_1<\cdots<\mu_N$ it is well-known, see for example \cite{BBO,JonesOConnell}, and also falls within the general framework of Section \ref{SectionConditioned} that this is the transition density of $N$ independent Brownian motions with drifts $\mu_1,\dots,\mu_N$, starting in $\mathbb{W}_N$ and conditioned to never intersect.
\end{rmk}

\section{Diffusions conditioned to never intersect proofs}\label{SectionConditioned}
In this section we prove Theorem \ref{ConditionedBesselProcess} as a corollary of a result on a more general class of one-dimensional diffusions under some rather natural assumptions. The novelty here is not in the form of the argument, which is based on the classical case of Brownian motion, but rather in finding the right setting and level of generality for which explicit formulae for non-intersection probabilities exist. We note that exit probabilities also in more general cones for multidimensional Brownian motion, with or without drifts, have been studied in many papers by a variety of techniques, see for example \cite{DeBlassie,Banuelos,Grabiner,BBO,Doumerc,JonesOConnell,Puchala,Garbit, KatoriBook} and the references therein. We begin with some preliminaries.

We consider a one-dimensional diffusion process in an interval $(l,r)$ with infinitesimal generator $\mathsf{L}$ given by
\begin{align}
\mathsf{L}=\mathsf{a}(x)\frac{d^2}{dx^2}+\mathsf{b}(x)\frac{d}{dx},
\end{align}
where we assume that the boundary points $l$ and $r$ are inaccessible (and in particular they can be removed from the state space, \cite{ItoMckean,BorodinSalminen,StochasticBook,EthierKurtz}). Namely, they are either natural or entrance boundaries, see \cite{ItoMckean,BorodinSalminen,StochasticBook,EthierKurtz} for the details on this terminology. An integral criterion involving the coefficients $\mathsf{a}$ and $\mathsf{b}$ for this to hold  exists due to Feller, see for example \cite{ItoMckean,BorodinSalminen,StochasticBook,EthierKurtz}. Also, in order to avoid unnecessary technicalities, we assume throughout this paper that
\begin{align}\label{smoothness}
 \mathsf{a}(\cdot), \mathsf{b}(\cdot)\in \mathcal{C}^{\infty} \ \textnormal{ and } \ \mathsf{a}(\cdot)>0 \ \textnormal{ on } (l,r).
\end{align}
We call the diffusion with generator $\mathsf{L}$ the $\mathsf{L}$-diffusion and denote its transition density with respect to the Lebesgue measure in $(l,r)$ by $p_t(x,y)$.

We shall denote by $\psi_{\lambda}$ the unique, up to multiplicative constant (the choice of which is unimportant in what follows and thus fix in an arbitrary way henceforth), strictly positive increasing\footnote{There is also a unique up to multiplicative constant strictly positive decreasing eigenfunction $u_\lambda$ of $\mathsf{L}$ with eigenvalue $\lambda$: $\mathsf{L}u_\lambda(x)=\lambda u_\lambda(x)$. A corresponding theory (with some adaptations) to the one we present below on computing probabilities of non-intersection can be built using these eigenfunctions as well.} eigenfunction of $\mathsf{L}$ (with eigenvalue $\lambda$) such that for $\lambda>0$
\begin{align}\label{EigenfunctionDef}
\mathsf{L}\psi_\lambda(x)=\lambda \psi_{\lambda}(x), \ \ x\in (l,r)
\end{align}
and subject to the appropriate boundary conditions at $l$ and $r$, see  \cite{ItoMckean,BorodinSalminen,StochasticBook}. 
We will denote by $\mathsf{L}^{\psi_{\lambda}}$ the generator of the Doob $h$-transformed diffusion, see \cite{RevuzYor,PitmanYor,Pinsky}, obtained from $\mathsf{L}$ and $\psi_{\lambda}$ (we also use the convention $\mathsf{L}^{\psi_0}\equiv \mathsf{L}$):
\begin{align*}
\mathsf{L}^{\psi_{\lambda}}=\psi_{\lambda}^{-1}\circ \mathsf{L} \circ \psi_{\lambda}-\lambda=\mathsf{a}(x)\frac{d^2}{dx^2}+\left(\mathsf{b}(x)+2\mathsf{a}(x)\frac{\psi'_\lambda(x)}{\psi_\lambda(x)}\right)\frac{d}{dx}.
\end{align*}
We will call the corresponding process the $\mathsf{L}^{\psi_{\lambda}}$-diffusion. We make the following standing assumption.
\begin{itemize}
    \item For all $\lambda\ge0$ the boundary points $l$ and $r$ are inaccessible for the $\mathsf{L}^{\psi_\lambda}$-diffusion and we denote this condition by (\textbf{BC}).
\end{itemize}
In particular, its transition density, denoted by $p_t^{\psi_{\lambda}}(x,y)$, with respect to the Lebesgue measure
\begin{align*}
 p_t^{\psi_{\lambda}}(x,y)=e^{-\lambda t}   \frac{\psi_\lambda(y)}{\psi_\lambda(x)}p_t(x,y),
\end{align*}
is a bona fide Markov (integrates to $1$) transition density in $(l,r)$.

The $\mathsf{L}^{\psi_\lambda}$-diffusion has the following interesting probabilistic interpretation (that we will not make use of here though), which is also closely related to Williams' path decomposition of one-dimensional diffusions \cite{Williams}, see \cite{PitmanYor} for more details. Namely, the $\mathsf{L}^{\psi_{\lambda}}$-diffusion is obtained by first killing the original $\mathsf{L}$-diffusion at a constant rate $\lambda$ and then conditioning this killed process to exit $(l,r)$ through $r$ (not necessarily in finite time, see \cite{PitmanYor} for the details).

It is instructive for the reader to keep in mind the following simple examples:
\begin{itemize}
    \item $\mathsf{L}=\frac{1}{2}\frac{d^2}{dx^2}$ is the generator of standard Brownian motion $\mathsf{w}(t)$ on $\mathbb{R}$ and $\psi_{\lambda}(x)=e^{\sqrt{2\lambda}x}$, where both $\pm \infty$ are natural boundaries. Then, the $\mathsf{L}^{\psi_{\lambda}}$-diffusion is simply Brownian motion with drift $\sqrt{2\lambda}$: $\mathsf{x}^{\lambda}(t)=\mathsf{w}(t)+\sqrt{2\lambda}t$.
    \item $\mathsf{L}=2x\frac{d^2}{dx^2}+\delta \frac{d}{dx}$ for $\delta\ge 2$ is the generator of the $\textnormal{BESQ}(\delta)$ process in $(0,\infty)$, with $\psi_{\lambda}(x)=\phi^{(\nu)}_{\lambda}(x)=(2\lambda x)^{-\frac{\nu}{2}}I_{\nu}\left(\left(2\lambda x\right)^{\frac{1}{2}}\right)$, where as usual $\delta=2(\nu+1)$. It is well-known, see \cite{RevuzYor,SurveyBessel}, that $0$ is an entrance boundary while $\infty$ is natural. Then, the $\mathsf{L}^{\psi_\lambda}$-diffusion is the $\textnormal{BESQ}_\lambda(\delta)$ process from the introduction.
\end{itemize}

Based on the example of Brownian motion we think of the term $2\mathsf{a}\frac{\psi_{\lambda}'}{\psi_\lambda}$ as a kind of generalized drift added to the diffusion $\mathsf{L}$ and that the parameter $\lambda$ governs the strength of this drift. We require a final assumption that we call asymptotic ordering, and denote by (\textbf{AO}).

\begin{itemize}
    \item Let $\lambda_1<\lambda_2$ and suppose that $\mathsf{x}^{\lambda_1}_1$ and $\mathsf{x}^{\lambda_2}_2$ are independent $\mathsf{L}^{\psi_{\lambda_1}}$ and $\mathsf{L}^{\psi_{\lambda_2}}$ diffusions starting from $x_1,x_2\in (l,r)$, not necessarily ordered. Then, (\textbf{AO}) is the following condition:
\begin{align}\label{AsymptoticOrdering}
 \mathbb{P}_{(x_1,x_2)}\left(\mathsf{x}_1^{\lambda_1}(t)<\mathsf{x}_2^{\lambda_2}(t)\right) \overset{t\to \infty}{ \longrightarrow} 1.
\end{align}
\end{itemize} 
In words, the two independent diffusions become ordered (irrespective of their initial ordering) depending on the parameter of their generalized drift. We believe that this should be true under some rather general assumptions on $\mathsf{L}$, which unfortunately remain elusive for now. Nevertheless, it could be checked on a case by case basis. For explicit examples, including the $\textnormal{BESQ}_\lambda(\delta)$ case \cite{Watanabe}, much stronger asymptotic results are known. For the quintessential example of Brownian motion, which provides the intuition, it is clearly obvious.

We now move on to our first result. Let $\left(\mathsf{x}_i^{\lambda_i}(t);t\ge 0\right)$ be independent $\mathsf{L}^{\psi_{\lambda_i}}$-diffusions. We assume that the process $\left(\left(\mathsf{x}_1^{\lambda_1}(t),\dots,\mathsf{x}_N^{\lambda_N}(t);t\ge 0\right)\right)$ starts in the chamber $\mathbb{W}_N=\mathbb{W}_N(l,r)$:
\begin{align*}
 \mathbb{W}_N=\left\{x=(x_1,\dots,x_N)\in (l,r)^N:x_1<x_2<\cdots<x_N \right\}.
\end{align*}
We are interested in the first collision (or intersection) time of these independent diffusions, which is also the time $\left(\left(\mathsf{x}_1^{\lambda_1}(t),\dots,\mathsf{x}_N^{\lambda_N}(t)\right);t\ge 0\right)$ first exits the chamber $\mathbb{W}_N$,
\begin{align*}
 \tau_{C}=\tau_C^{\mathsf{L},\lambda_1,\dots,\lambda_N}=\inf\left\{t\ge 0:\left(\mathsf{x}^{\lambda_1}_1(t),\dots,\mathsf{x}^{\lambda_N}_N(t)\right)\notin \mathbb{W}_N\right\}.
\end{align*}
We have the following explicit formula for $\tau_C$.

\begin{prop} \label{AsymptoticProbability} Let $\mathsf{L}$ be a one-dimensional diffusion process generator in an interval $(l,r)$, with positive increasing eigenfunctions $\psi_\lambda$ as in (\ref{EigenfunctionDef}), satisfying the assumptions above: (\ref{smoothness}), (\textbf{BC}) and (\textbf{AO}). Let $\mathsf{x}_i^{\lambda_i}$ be independent $\mathsf{L}^{\psi_{\lambda_i}}$-diffusions, where $\lambda_1<\lambda_2<\cdots<\lambda_N$, starting from $\left(\mathsf{x}_1^{\lambda_1}(0),\dots,\mathsf{x}_N^{\lambda_N}(0)\right)=x\in \mathbb{W}_N$. Then, the probability that these diffusions never collide is explicit and is given by
\begin{align}\label{CollisionProba}
\mathbb{P}_x(\tau_C=\infty)=\frac{\det\left(\psi_{\lambda_i}(x_j)\right)^N_{i,j=1}}{\prod_{i=1}^N\psi_{\lambda_i}(x_i)}.
\end{align}
\end{prop}

\begin{proof}
By using the Karlin-McGregor formula, see \cite{KarlinMcGregor,Karlin,ItoMckean}, and a Doob $h$-transform \cite{Doob} we obtain, recall that from (\textbf{BC}) we have no atoms at $l$ or $r$:
\begin{align*}
    \mathbb{P}_x\left(\left(\mathsf{x}^{\lambda_1}_1(t),\dots,\mathsf{x}^{\lambda_N}_N(t)\right)\in dy,\tau_C>t\right)=e^{-t\sum_{i=1}^{N}\lambda_i}\prod_{i=1}^N\frac{\psi_{\lambda_i}(y_i)}{\psi_{\lambda_i}(x_i)}\det\left(p_t(x_i,y_j)\right)^N_{i,j=1}dy.
\end{align*}
Thus, we can compute, using the formula above and the fact that no mass is lost at the boundary points, in particular $p_t^{\psi_\lambda}(x,y)$ integrates to $1$ over $(l,r)$, because of (\textbf{BC}):
\begin{align*}
\mathbb{P}_x(\tau_C>t)&=\int_{\mathbb{W}_N}^{}e^{-t\sum_{i=1}^{N}\lambda_i}\prod_{i=1}^N\frac{\psi_{\lambda_i}(y_i)}{\psi_{\lambda_i}(x_i)}\det\left(p_t(x_i,y_j)\right)^N_{i,j=1}dy_1\cdots dy_N\\
&=\sum_{\sigma \in \mathfrak{S}(N)}\textnormal{sgn}(\sigma)\int_{\mathbb{W}_N}e^{-t\sum_{i=1}^{N}\lambda_i}\prod_{i=1}^N\frac{\psi_{\lambda_i}(y_i)}{\psi_{\lambda_i}(x_i)} \prod_{i=1}^N p_t(x_{\sigma(i)},y_i)dy_1\cdots dy_N\\
&=\sum_{\sigma \in \mathfrak{S}(N)}\textnormal{sgn}(\sigma)\bigg[ \int_{(l,r)^N}e^{-t\sum_{i=1}^{N}\lambda_i}\prod_{i=1}^N\frac{\psi_{\lambda_i}(y_i)}{\psi_{\lambda_i}(x_i)} \prod_{i=1}^N p_t(x_{\sigma(i)},y_i)dy_1\cdots dy_N\\
&- \int_{(l,r)^N\backslash\mathbb{W}_N}e^{-t\sum_{i=1}^{N}\lambda_i}\prod_{i=1}^N\frac{\psi_{\lambda_i}(y_i)}{\psi_{\lambda_i}(x_i)} \prod_{i=1}^N p_t(x_{\sigma(i)},y_i)dy_1\cdots dy_N\bigg]\\
&=\frac{\det\left(\psi_{\lambda_i}(x_j)\right)^N_{i,j=1}}{\prod_{i=1}^N\psi_{\lambda_i}(x_i)}-\sum_{\sigma \in \mathfrak{S}(N)}\textnormal{sgn}(\sigma)\bigg[\prod_{i=1}^N\frac{\psi_{\lambda_i}(x_{\sigma(i)})}{\psi_{\lambda_i}(x_i)} \\ 
& \ \ \ \ \times \mathbb{P}_{(x_{\sigma(1)},\dots,x_{\sigma(N)})}\left(\left(\mathsf{x}^{\lambda_1}_1(t),\dots,\mathsf{x}^{\lambda_N}_N(t)\right)\notin \mathbb{W}_N\right)\bigg].
\end{align*}
Note that, for the last equality we made use of the following manipulation, 
\begin{align*}
\int_{\mathcal{A}}e^{-t\sum_{i=1}^{N}\lambda_i}\prod_{i=1}^N\frac{\psi_{\lambda_i}(y_i)}{\psi_{\lambda_i}(x_i)} \prod_{i=1}^N p_t(x_{\sigma(i)},y_i)dy_1\cdots dy_N\\ 
=\prod_{i=1}^N\frac{\psi_{\lambda_i}(x_{\sigma(i)})}{\psi_{\lambda_i}(x_i)}\int_{\mathcal{A}}\prod_{i=1}^N p_t^{\psi_{\lambda_i}}(x_{\sigma(i)},y_i)dy_1\cdots dy_N,
\end{align*}
with $\mathcal{A}=(l,r)^N$ and $\mathcal{A}=(l,r)^N\backslash\mathbb{W}_N$ respectively, and recall again that $p_t^{\psi_\lambda}(x,y)$ integrates to $1$ over $(l,r)$. Hence, by the asymptotic ordering assumption (\textbf{AO}) we have that for any permutation $\sigma$:
\begin{align*}
 \mathbb{P}_{(x_{\sigma(1)},\dots,x_{\sigma(N)})}\left(\left(\mathsf{x}^{\lambda_1}_1(t),\dots,\mathsf{x}^{\lambda_N}_N(t)\right)\notin \mathbb{W}_N\right)\longrightarrow 0, \ \ \textnormal{ as } t \to \infty,
\end{align*}
and this completes the proof.
\end{proof}

\begin{rmk}
It is easy to show that the right hand side of (\ref{CollisionProba}) is a bounded 
solution to the PDE in $\mathbb{W}_N$
\begin{align*}
 \left(\sum_{i=1}^N \left[\psi_{\lambda_i}^{-1}\circ \mathsf{L}_{x_i}\circ \psi_{\lambda_i}-\lambda_i\right]\right)u(x_1,\dots,x_N)=0,
\end{align*}
 with Dirichlet boundary conditions $ u|_{x_i=x_{i-1}}\equiv 0, \textnormal{ for } i=2,\dots, N,$ when two coordinates coincide. Restricting for simplicity to the case $r=\infty$, being a natural boundary point, we expect (in some generality) that the right hand side of (\ref{CollisionProba}) should also satisfy: $ u(x_1,\dots,x_N)\longrightarrow 1$  if $(x_i-x_{i-1})\to \infty$ for all $i=2,\dots,N$. This turns out to be rather tricky to prove in a general setting but for explicit examples such as the $\textnormal{BESQ}(\delta)$ case it can be shown using asymptotics for special functions. If moreover the corresponding $\mathsf{L}^{\psi_\lambda}$-diffusions satisfy a strengthened asymptotic ordering condition\footnote{Let $\lambda_1<\lambda_2$ and suppose that $\mathsf{x}^{\lambda_1}_1$ and $\mathsf{x}^{\lambda_2}_2$ are independent $\mathsf{L}^{\psi_{\lambda_1}}$ and $\mathsf{L}^{\psi_{\lambda_2}}$ diffusions starting from $x_1<x_2$. Then, this strengthened asymptotic ordering condition is the following:
\begin{align*}
\mathbb{P}\left(\lim_{t\to \infty}\left(\mathsf{x}_2^{\lambda_2}(t)-\mathsf{x}_1^{\lambda_1}(t)\right)=\infty\big|\mathsf{x}_2^{\lambda_2}(0)=x_2,\mathsf{x}_1^{\lambda_1}(0)=x_1\right)=1.
\end{align*} In the $\textnormal{BESQ}(\delta)$ case this is a consequence of a result of Watanabe \cite{Watanabe} recalled in (\ref{WatanabeAsymptotic}) below.} then it can be shown, using It\^{o}'s formula and the optional stopping theorem, that the non-collision probability $\mathbb{P}_x(\tau_C=\infty)$ is the unique solution to the PDE (along with the conditions) above. This gives an alternative route to establishing the statement of Proposition \ref{AsymptoticProbability}, under stronger assumptions, which can nevertheless still be checked for explicit examples such as the $\textnormal{BESQ}(\delta)$ case.
\end{rmk}

While $\mathbb{P}_x\left(\tau_C=\infty\right)$ is clearly non-negative, we will need that it is in fact strictly positive and this is what we prove in the next lemma. Strict positivity is also a consequence 
  for a different class of diffusions of our results in Section \ref{SectionInteractingDiffusions}.

\begin{lem}\label{StrictPosLem} Let $\mathsf{L}$ be a diffusion process generator in an interval $(l,r)$ with positive increasing eigenfunctions $\psi_\lambda$ as in (\ref{EigenfunctionDef}), satisfying (\ref{smoothness}), (\textbf{BC}) and (\textbf{AO}). If $\lambda_1<\lambda_2<\cdots<\lambda_N$ we have
\begin{align*}
 \det\left(\psi_{\lambda_i}(x_j)\right)^N_{i,j=1}> 0, \ \ \forall x\in \mathbb{W}_N.   
\end{align*}
\end{lem}

\begin{proof}
Let $x\in \mathbb{W}_N$ be arbitrary. First, by using the Andr\'{e}ief identity \cite{Andreief} and the fact that, due to (\textbf{BC}), $p_t^{\psi_{\lambda}}(x,y)$ integrates to $1$ over $(l,r)$, we obtain
\begin{align}
\int_{\mathbb{W}_N}\det\left(p_t(x_i,y_j)\right)_{i,j=1}^N\det\left(\psi_{\lambda_i}(y_j)\right)_{i,j=1}^N dy&=\det\left(\int_l^r p_t(x_j,z)\psi_{\lambda_i}(z)dz\right)_{i,j=1}^N\nonumber\\
&=e^{\sum_{i=1}^N \lambda_i t}\det\left(\psi_{\lambda_i}(x_j)\right)_{i,j=1}^N.\label{AndreiefApplication}
\end{align}
Now, from (\ref{smoothness}) we have that $p_t(x,y)$ is strictly positive and also smooth (continuity would suffice for what follows) in $y\in (l,r)$ for any $(t,x)\in (0,\infty)\times (l,r)$. Hence, from Theorem 4 of \cite{KarlinMcGregor} we get that 
\begin{align*}
  \det\left(p_t(x_i,y_j)\right)_{i,j=1}^N >0 , \ \ \forall (t,x,y)\in (0,\infty) \times \mathbb{W}_N \times \mathbb{W}_N,
\end{align*}
and moreover this is a smooth function in $y\in \mathbb{W}_N$. In addition, from Proposition \ref{AsymptoticProbability}, whose conditions are satisfied, we get that $\det\left(\psi_{\lambda_i}(y_j)\right)_{i,j=1}^N\ge 0$ and from (\ref{smoothness}) this function is also smooth in $\mathbb{W}_N$. We then see from (\ref{AndreiefApplication}) that we must have $\det\left(\psi_{\lambda_i}(x_j)\right)_{i,j=1}^N>0$, since the integral on the left hand side is strictly positive unless the following holds
\begin{equation}\label{Contradiction}
 \det\left(\psi_{\lambda_i}(y_j)\right)_{i,j=1}^N\equiv 0, \ \ \forall y \in \mathbb{W}_N,   
\end{equation}
which we now show leads to a contradiction. So assume (\ref{Contradiction}). Then, by dividing by the Vandermonde determinant $\mathsf{\Delta}_N(y)$ and taking $(y_1,\dots,y_N)\to (z,\dots,z)$ we obtain that the Wronskian of $\psi_{\lambda_1}, \dots, \psi_{\lambda_N}$ is identically zero on $(l,r)$, namely
\begin{equation*}
 \det \left(\partial_z^{j-1}\psi_{\lambda_i}(z)\right)_{i,j=1}^N \equiv 0  , \ \ \forall z \in (l,r).
\end{equation*}
By a classical result, see for example the Corollary on page 48 of \cite{Hurewicz}, we obtain that there exists a non-empty subinterval $(\tilde{l},\tilde{r})\subseteq	(l,r)$ on which the functions $\psi_{\lambda_i}$ are linearly dependent. Namely, there exist constants $\alpha_1, \dots, \alpha_N \in \mathbb{R}$, not all of them zero, such that
\begin{equation*}
   \sum_{j=1}^N \alpha_j \psi_{\lambda_j}(z) =0 , \ \ \forall z \in (\tilde{l},\tilde{r}).
\end{equation*}
Let $i_*=\max\{i: \alpha_i \neq 0 \}$. Then, we can write
\begin{align*}
 \psi_{\lambda_{i_*}}(z) =-\sum_{j=1}^{i_*-1}\frac{\alpha_j}{\alpha_{i_*}}\psi_{\lambda_j}(z), \ \ \forall z \in (\tilde{l},\tilde{r}).
\end{align*}
Applying the diffusion operator $\mathsf{L}$ to both sides of this equality, and using the eigenfunction relation (\ref{EigenfunctionDef}), $n$ times  we obtain
\begin{align*}
    \psi_{\lambda_{i_*}}(z) =-\sum_{j=1}^{i_*-1}\left(\frac{\lambda_j}{\lambda_{i_*}}\right)^n\frac{\alpha_j}{\alpha_{i_*}}\psi_{\lambda_j}(z), \ \ \forall z \in (\tilde{l},\tilde{r}).
\end{align*}
Since for all $j=1,\dots,i_*-1$, $\left(\frac{\lambda_j}{\lambda_{i_*}}\right)<1$, by sending $n \to \infty$ we get $\psi_{\lambda_{i_*}}(z)=0$ in $(\tilde{l},\tilde{r})$ which gives the desired contradiction (since $\psi_\lambda$ is strictly positive) and completes the proof.
\end{proof}

An immediate consequence of Proposition \ref{AsymptoticProbability} and Lemma \ref{StrictPosLem} is the following.

\begin{cor}\label{StrictPositivity} In the setting of Proposition \ref{AsymptoticProbability} we have
\begin{align*}
 \mathbb{P}_x(\tau_C=\infty)=\frac{\det\left(\psi_{\lambda_i}(x_j)\right)^N_{i,j=1}}{\prod_{i=1}^N\psi_{\lambda_i}(x_i)}>0, \ \ \forall x\in \mathbb{W}_N.   
\end{align*}
\end{cor}

We can now prove the following result.

\begin{prop} In the setting of Proposition \ref{AsymptoticProbability},
for any $x \in \mathbb{W}_N$ and Borel set $\mathcal{A}\subset \mathbb{W}_N$ we have
\begin{align}
\mathbb{P}_x\left(\left(\mathsf{x}^{\lambda_1}_1(t),\dots,\mathsf{x}^{\lambda_N}_N(t)\right)\in \mathcal{A}|\tau_C=\infty\right)=\int_{\mathcal{A}}^{}\mathbb{P}_x\left(\left(\mathsf{x}^{\lambda_1}_1(t),\dots,\mathsf{x}^{\lambda_N}_N(t)\right)\in dy,\tau_C>t\right)\frac{\mathbb{P}_y\left(\tau_C=\infty\right)}{\mathbb{P}_x\left(\tau_C=\infty\right)}.
\end{align}
In particular, the transition kernel of the conditioned process is given by the explicit formula, for any $t>0$, $x \in \mathbb{W}_N$ and $y \in \overline{\mathbb{W}}_N$:
\begin{align}\label{ConditionedSemigroup}
\mathfrak{P}_t^{(N),(\lambda_1,\dots,\lambda_N)}(x,dy)=e^{-t\sum_{i=1}^{N}\lambda_i}\frac{\det\left(\psi_{\lambda_i}(y_j)\right)^N_{i,j=1}}{\det\left(\psi_{\lambda_i}(x_j)\right)^N_{i,j=1}}\det\left(p_t(x_i,y_j)\right)^N_{i,j=1}dy_1\dots dy_N.
\end{align}
\end{prop}
\begin{proof} Since $\mathbb{P}_x\left(\tau_C=\infty\right)>0$ from Corollary \ref{StrictPositivity}, we then have, for any $t>0$
\begin{align*}
\frac{\mathbb{P}_y\left(\tau_C>s\right)}{\mathbb{P}_x\left(\tau_C>t+s\right)}\overset{s\to \infty}{\longrightarrow}\frac{\mathbb{P}_y\left(\tau_C=\infty\right)}{\mathbb{P}_x\left(\tau_C=\infty\right)}
\end{align*}
and we can use bounded convergence to obtain
\begin{align*}
&\mathbb{P}_x\left(\left(\mathsf{x}^{\lambda_1}_1(t),\dots,\mathsf{x}^{\lambda_N}_N(t)\right)\in \mathcal{A}|\tau_C=\infty\right)=\lim_{s\to \infty}\mathbb{P}_x\left(\left(\mathsf{x}^{\lambda_1}_1(t),\dots,\mathsf{x}^{\lambda_N}_N(t)\right)\in \mathcal{A}|\tau_C>t+s\right)\\
&=\lim_{s\to \infty}\int_{\mathcal{A}}^{}\mathbb{P}_x\left(\left(\mathsf{x}^{\lambda_1}_1(t),\dots,\mathsf{x}^{\lambda_N}_N(t)\right)\in dy,\tau_C>t\right)\frac{\mathbb{P}_y\left(\tau_C>s\right)}{\mathbb{P}_x\left(\tau_C>t+s\right)}\\
&=\int_{\mathcal{A}}^{}\mathbb{P}_x\left(\left(\mathsf{x}^{\lambda_1}_1(t),\dots,\mathsf{x}^{\lambda_N}_N(t)\right)\in dy,\tau_C>t\right)\frac{\mathbb{P}_y\left(\tau_C=\infty\right)}{\mathbb{P}_x\left(\tau_C=\infty\right)}.
\end{align*}
The form of the transition kernel then follows from Proposition \ref{AsymptoticProbability}.
\end{proof}

We now give the proof of Theorem \ref{ConditionedBesselProcess}.

\begin{proof}[Proof of Theorem \ref{ConditionedBesselProcess}]
We simply need to check the assumptions of Proposition \ref{AsymptoticProbability}. It is clear that (\ref{smoothness}) holds. Moreover, the asymptotic ordering (\textbf{AO}) is a consequence of the following result of Watanabe, see \cite{Watanabe}, where $\mathsf{x}^{\lambda}$ is a $\textnormal{BESQ}_{\lambda}(\delta)$ process with $\delta>0$ (note that the limit does not depend on $\delta$):
\begin{align}\label{WatanabeAsymptotic}
    \mathbb{P}_x\left(\lim_{t\to \infty}\frac{\mathsf{x}^\lambda(t)}{t^2}=2\lambda\right)=1, \ \  \forall x \in [0,\infty).
\end{align}
Finally, (\textbf{BC}) is a consequence of Lemma \ref{LemBoundaryBehaviour} in the sequel.  The condition (\textbf{YW}), given in Definition \ref{YWDef}, that is required to apply Lemma \ref{LemBoundaryBehaviour} will be checked explicitly for $\textnormal{BESQ}_\lambda(\delta)$ in the proof of Theorem \ref{InteractingDiffusionsThm} in Section \ref{SectionProofInteracting}.
\end{proof}

\section{Interacting diffusions proofs}\label{SectionInteractingDiffusions}
In this section we prove Theorem \ref{InteractingDiffusionsThm}. We do this in the more general setting of one-dimensional diffusions in $(0,\infty)$ with $0$ an entrance and $\infty$ a natural boundary point. The computations that follow become more transparent if performed in this general setting rather than using the explicit formulae in the $\textnormal{BESQ}(\delta)$ case (which is how we discovered the result in the first place). It would also be possible to have natural analogues of the constructions that follow with more general boundary conditions (with some modifications in the statements and assumptions) but keeping track of everything becomes very cumbersome and we chose to restrict to the present setting. 

We need a few preliminaries but first we give a little roadmap to this section. The main result of the section is Proposition \ref{GeneralInteracting} whose conditions we check to prove Theorem \ref{InteractingDiffusionsThm}. We prove Proposition \ref{GeneralInteracting} by induction using two key results, Propositions \ref{TwoLevel1} and \ref{TwoLevel2}. These propositions are proven by combining some results from \cite{InterlacingDiffusions} along with Lemmas \ref{LemBoundaryBehaviour} and \ref{SemigroupObs}, the construction in Definition \ref{MarkovKernelDef} and the discussion between equations (\ref{Inter2}) and (\ref{DoobKMsemi1}). Finally, the basic data required to apply Proposition \ref{GeneralInteracting} is recalled before its formal statement.

\subsection{Background on one-dimensional diffusions}\label{SectionBackground}
We assume throughout this section that we are given a one-dimensional diffusion process generator $\mathsf{L}$ in $(0,\infty)$:
\begin{align*}
 \mathsf{L}=\mathsf{a}(x)\frac{d^2}{dx^2}+\mathsf{b}(x)\frac{d}{dx},
\end{align*}
satisfying (\ref{smoothness}), so that $0$ is an entrance boundary point while $\infty$ is a natural boundary point, see \cite{ItoMckean,BorodinSalminen,StochasticBook,EthierKurtz}.  As mentioned already, it is well-known, see for example \cite{RevuzYor,SurveyBessel}, that the $\textnormal{BESQ}(\delta)$ process for $\delta\ge 2$, has this boundary behaviour. 

We denote by $\mathfrak{s}'$ the derivative of its scale function $\mathfrak{s}$, which is defined up to a multiplicative constant (encoded by the constant $c \in (0,\infty)$ below which is arbitrary but fixed throughout this section) and given by the explicit formula, see \cite{ItoMckean,BorodinSalminen,StochasticBook,EthierKurtz},
\begin{align}
 \mathfrak{s}'(x)=\exp\left(-\int_{c}^{x}\frac{\mathsf{b}(y)}{\mathsf{a}(y)}dy\right).
\end{align}
We denote by $\mathfrak{m}$ the density of its speed measure, with respect to Lebesgue measure, given by the formula (with the same $c$ as above), see \cite{ItoMckean,BorodinSalminen,StochasticBook,EthierKurtz},
\begin{align}\label{speedmeasure}
 \mathfrak{m}(x)=\frac{1}{\mathsf{a}(x)\mathfrak{s}'(x)}=\frac{1}{\mathsf{a}(x)}\exp\left(\int_{c}^{x}\frac{\mathsf{b}(y)}{\mathsf{a}(y)}dy\right). 
\end{align}
We denote by $p_t(x,y)$ its transition density with respect to Lebesgue measure in $(0,\infty)$. We denote by $\mathfrak{m}^{\psi_\lambda}$ and $\mathfrak{s}^{\psi_\lambda}$ the corresponding quantities for the $\mathsf{L}^{\psi_\lambda}$-diffusion. For example a little computation gives $\mathfrak{m}^{\psi_{\lambda}}(x)=\frac{1}{\psi^2_{\lambda}(c)}\psi_{\lambda}^2(x)\mathfrak{m}(x)$.

In the sequel we will also need the following well-known condition due to Yamada and Watanabe that gives the $\mathsf{L}$-diffusion as the unique strong solution to a non-exploding SDE, see \cite{IkedaWatanabe,RevuzYor}.

\begin{defn}\label{YWDef} We will say that an ordered pair of functions $(f,g)$ satisfy the Yamada-Watanabe condition and denote this by (\textbf{YW}) if, for some increasing Borel function $\rho:(0,\infty)\to (0,\infty)$ with $\int_{0^+}\frac{1}{\rho(u)}du=\infty$, we have
\begin{align*}
    |f(x)-f(y)|^2\le \rho\left(|x-y|\right), \ \ |f(x)|^2 \le C_1\left(1+x^2\right),  \ \ |g(x)-g(y)|\le C_2|x-y|,
\end{align*}
for some constants $C_1, C_2$. We will say that a one-dimensional diffusion process generator $Q=a(x)\frac{d^2}{dx^2}+b(x)\frac{d}{dx}$ satisfies (\textbf{YW}) if the pair $(\sqrt{a},b)$ satisfies (\textbf{YW}).
\end{defn}
We will assume throughout that all the diffusions we encounter satisfy (\textbf{YW}).

We now define the following involution operation on one-dimensional diffusions. Let $Q=a(x)\frac{d^2}{dx^2}+b(x)\frac{d}{dx}$ be a diffusion process generator in $[0,\infty)$ so that $\infty$ is natural and $0$ is either an entrance or an exit boundary point, see \cite{ItoMckean,BorodinSalminen,StochasticBook,EthierKurtz} for more on this terminology. We then define its dual diffusion process generator $\widehat{Q}$ by
\begin{align}\label{DualDiffusion}
\widehat{Q}=a(x)\frac{d^2}{dx^2}+\left[a'(x)-b(x)\right]\frac{d}{dx}.
\end{align}
We note that this is well-defined since the form of the generator $\widehat{Q}$ uniquely determines its boundary classification at $0$ and $\infty$ using the integral criterion of Feller, see for example the Appendix in \cite{InterlacingDiffusions}. In particular, $\infty$ stays a natural boundary while if $0$ is an entrance boundary for the $Q$-diffusion it becomes exit for the $\widehat{Q}$-diffusion and vice-versa. The example to keep in mind is again the $\textnormal{BESQ}(\delta)$ process for $\delta\ge 2$. Its dual diffusion is the $\textnormal{BESQ}(2-\delta)$ process (see \cite{SurveyBessel} where this was studied in detail) absorbed at the origin and vice-versa.

We denote by $\hat{\mathfrak{m}}$ and $\hat{\mathfrak{s}}'$ the density of the speed measure and the derivative of the scale function respectively of the $\widehat{\mathsf{L}}$-diffusion. We write $\hat{p}_t(x,y)$ for its transition density with respect to Lebesgue measure in $(0,\infty)$. Observe that, the transition kernel of the $\widehat{\mathsf{L}}$-diffusion has an atom at the origin (since the $\widehat{\mathsf{L}}$-diffusion is absorbed there as $0$ is an exit boundary) and thus the density $\hat{p}_t(x,y)$ can also be viewed as the transition kernel of the $\widehat{\mathsf{L}}$-diffusion killed when it hits $0$ (instead of absorbed, since we disregard the atom). We also denote by $\widehat{\mathfrak{m}^{\psi_\lambda}}$ and $\widehat{\mathfrak{s}^{\psi_\lambda}}$ the corresponding quantities for the $\widehat{\mathsf{L}^{\psi_\lambda}}$-diffusion. A small computation for example gives $\widehat{\mathfrak{m}^{\psi_{\lambda}}}(x)=\psi^2_{\lambda}(c)\psi_{\lambda}^{-2}(x)\hat{\mathfrak{m}}(x)$.

Finally, it will be convenient to introduce the following notation, for a (smooth enough) function $g$,
\begin{align*}
    \mathcal{D}_g=\frac{1}{g'(x)}\frac{d}{dx}, \ \  \mathsf{D}_g=\frac{1}{g(x)}\frac{d}{dx}.
\end{align*}
The exact computations that come up in the next section become very clean if written in terms of these operators because of the following fact. Namely, a little calculation gives $\hat{\mathfrak{s}}'(x)=\mathsf{a}(c)\mathfrak{m}(x)$ and $\hat{\mathfrak{m}}(x)=\frac{1}{\mathsf{a}(c)}\mathfrak{s}'(x)$ and also
\begin{align}\label{GenRep}
 \mathsf{L}=\mathsf{D}_{\mathfrak{m}}\mathcal{D}_{\mathfrak{s}}\ \  \textnormal{ and } \ \  \widehat{\mathsf{L}}=\mathsf{D}_{\hat{\mathfrak{m}}}\mathcal{D}_{\hat{\mathfrak{s}}}=\mathcal{D}_{\mathfrak{s}}\mathsf{D}_{\mathfrak{m}},
\end{align}
and similarly for $\mathsf{L}^{\psi_\lambda}$.

\subsection{Two-level dynamics}

In this section we investigate certain two-level dynamics on interlacing configurations which are the basic building blocks for our construction. An essential input to our study are certain two-level couplings from \cite{InterlacingDiffusions} for Karlin-McGregor semigroups associated to a diffusion $Q$ and its dual $\widehat{Q}$ (under certain boundary conditions) so that the corresponding diffusions interlace. These were first discovered in the case of Brownian motion (which is self-dual) in \cite{Warren} and extended to the general setting in \cite{InterlacingDiffusions}. These two-level results from \cite{Warren,InterlacingDiffusions} take as input a strictly positive eigenfunction of a Karlin-McGregor semigroup. Then, one needs to put these two-level dynamics inductively together, in a consistent way, to obtain the desired result in the interlacing array. This is however a non-trivial task and at present there is no systematic way of doing this in general\footnote{In particular, this is the reason why although general two-level couplings were considered in \cite{InterlacingDiffusions} only special explicit examples of dynamics in whole interlacing arrays were presented there.}. The main contribution of the present paper is the construction of these eigenfunctions in such a consistent way for the general setting considered in this section. We note that this does not appear to be possible for a triangular interlacing array as alluded to in the introduction. 

We need some notation and terminology. We define
\begin{align*}
 \overline{\mathbb{W}}_N=\{x=(x_1,\dots,x_N)\in (0,\infty)^N : x_1\le \cdots \le x_N \},
\end{align*}
and similarly consider the following spaces of two-level interlacing configurations
\begin{align*}
\mathbb{W}_{N,N+1}&=\left\{(x,y)\in \mathbb{W}_N\times \mathbb{W}_{N+1}:x\prec y \right\}, & \tilde{\mathbb{W}}_{N,N+1}&=\left\{(x,y)\in \mathbb{W}_N\times \overline{\mathbb{W}}_{N+1}:x\prec y \right\}, \\
\mathbb{W}_{N,N}&=\left\{(x,y)\in \mathbb{W}_N\times \mathbb{W}_{N}:x\prec y \right\},  & \tilde{\mathbb{W}}_{N,N}&=\left\{(x,y)\in \mathbb{W}_N\times \overline{\mathbb{W}}_{N}:x\prec y \right\}.
\end{align*}
Moreover, suppose that we are given a sequence of non-negative numbers $\lambda_i$ satisfying $\lambda_1<\lambda_2<\lambda_3<\cdots$. 

We begin by computing certain multiple integrals under an interlacing constraint. We will rephrase these integrals in Proposition \ref{PropIntegration} in the sequel which will be important in the subsequent developments.

\begin{lem}\label{IntProp1}Let $\mathsf{L}$ be a one-dimensional diffusion process generator satisfying (\ref{smoothness}) in $(0,\infty)$ and with positive increasing eigenfunctions $\psi_\lambda$ as in (\ref{EigenfunctionDef}). Then, we have for $y\in \overline{\mathbb{W}}_{n+1}$:
\begin{align}\label{Integral1}
\int_{x\prec y}^{}\prod_{i=1}^{n}\widehat{\mathfrak{m}^{\psi_{\lambda_{n+1}}}}(x_i)(-1)^n\det\left(\mathsf{D}_{\widehat{\mathfrak{m}^{\psi_{\lambda_{n+1}}}}}\left(\frac{\psi_{\lambda_i}}{\psi_{\lambda_{n+1}}}\right)(x_j)\right)^n_{i,j=1}dx=\det \left(\left(\frac{\psi_{\lambda_i}}{\psi_{\lambda_{n+1}}}\right)(y_j)\right)_{i,j=1}^{n+1}.
\end{align}
\end{lem}

\begin{proof}
By multilinearity of the determinant the left hand side of (\ref{Integral1}) is equal to
\begin{align*}
    \det\left(-\int_{y_j}^{y_{j+1}}\widehat{\mathfrak{m}^{\psi_{\lambda_{n+1}}}}(z)\mathsf{D}_{\widehat{\mathfrak{m}^{\psi_{\lambda_{n+1}}}}}\left(\frac{\psi_{\lambda_i}}{\psi_{\lambda_{n+1}}}\right)(z)dz\right)_{i,j=1}^n=\det\left(\left(\frac{\psi_{\lambda_i}}{\psi_{\lambda_{n+1}}}\right)(y_j)-\left(\frac{\psi_{\lambda_i}}{\psi_{\lambda_{n+1}}}\right)(y_{j+1})\right)_{i,j=1}^n.
\end{align*}
The statement then follows by column operations on the right hand side of (\ref{Integral1}).
\end{proof}

\begin{lem}\label{IntProp2} Let $\mathsf{L}$ be a one-dimensional diffusion process generator satisfying (\ref{smoothness}) in $(0,\infty)$ so that $0$ is an entrance boundary point and with positive increasing eigenfunctions $\psi_\lambda$ as in (\ref{EigenfunctionDef}). Then, we have for $y\in \overline{\mathbb{W}}_n$
\begin{align}\label{Integral2}
\int_{x\prec y}^{}\prod_{i=1}^{n}\mathfrak{m}^{\psi_{\lambda_{n+1}}}(x_i)\det \left(\left(\frac{\psi_{\lambda_i}}{\psi_{\lambda_{n+1}}}\right)(x_j)\right)_{i,j=1}^{n}dx=\mathsf{c}_n \times (-1)^n \det \left(\mathsf{D}_{\widehat{\mathfrak{m}^{\psi_{\lambda_{n+1}}}}}\left(\frac{\psi_{\lambda_i}}{\psi_{\lambda_{n+1}}}\right)(y_j)\right)_{i,j=1}^{n},
\end{align}
where the multiplicative constant $\mathsf{c}_n$ is given by
\begin{align}\label{constant}
\mathsf{c}_n=\mathsf{c}_n(c;\lambda_1,\dots,\lambda_{n+1})=\prod_{i=1}^n\frac{1}{\left(\lambda_{n+1}-\lambda_{i}\right)\mathsf{a}(c)}.
\end{align}
\end{lem}

\begin{proof}
By multilinearity of the determinant the left hand side of (\ref{Integral2}) is equal to, with the convention $y_0=0$,
\begin{align*}
    &\det\left(   \int_{y_{j-1}}^{y_j}\frac{\psi_{\lambda_i}(z)}{\psi_{\lambda_{n+1}}(z)} \mathfrak{m}^{\psi_{\lambda_{n+1}}}(z)dz\right)_{i,j=1}^n=\\
    &\det\left(\frac{1}{\left(\lambda_i-\lambda_{n+1}\right)\mathsf{a}(c)}\mathsf{D}_{\widehat{\mathfrak{m}^{\psi_{\lambda_{n+1}}}}}\left(\frac{\psi_{\lambda_i}}{\psi_{\lambda_{n+1}}}\right)(y_j)-\frac{1}{\left(\lambda_i-\lambda_{n+1}\right)\mathsf{a}(c)}\mathsf{D}_{\widehat{\mathfrak{m}^{\psi_{\lambda_{n+1}}}}}\left(\frac{\psi_{\lambda_i}}{\psi_{\lambda_{n+1}}}\right)(y_{j-1})\right)_{i,j=1}^n,
\end{align*}
where we have used Lemma \ref{EigenRel} below. The statement then follows by column operations.
\end{proof}

\begin{lem}\label{EigenRel}  Let $\mathsf{L}$ be a one-dimensional diffusion process generator satisfying (\ref{smoothness}) in $(0,\infty)$ so that $0$ is an entrance boundary point and with positive increasing eigenfunctions $\psi_\lambda$ as in (\ref{EigenfunctionDef}). Then, we have for $y\ge 0$,
\begin{align}\label{EigenRelDisplay}
   \int_0^y\frac{\psi_{\lambda_1}(x)}{\psi_{\lambda_2}(x)} \mathfrak{m}^{\psi_{\lambda_2}}(x)dx=\frac{1}{\left(\lambda_1-\lambda_2\right)\mathsf{a}(c)}\mathsf{D}_{\widehat{\mathfrak{m}^{\psi_{\lambda_2}}}}\left(\frac{\psi_{\lambda_1}}{\psi_{\lambda_2}}\right)(y).
\end{align}
\end{lem}

\begin{rmk}
We note that the equality (\ref{EigenRelDisplay}) is, as it should be, independent of the choice of the constant $c\in (0,\infty)$; recall that $\mathfrak{m}^{\psi_{\lambda_2}}$ depends on $c$ through a multiplicative constant. It can be checked directly that changing $c$ to $\tilde{c}$ simply multiplies each side by a corresponding constant so that (\ref{EigenRelDisplay}) holds with $c$ replaced by $\tilde{c}$.
\end{rmk}

We first check that the integral in (\ref{EigenRelDisplay}) is finite.

\begin{lem}\label{LemIntegrable}
For any $\lambda,\mu \in (0,\infty)$, $y\in (0,\infty)$:
\begin{align*}
  \int_0^y\frac{\psi_{\lambda}(x)}{\psi_{\mu}(x)} \mathfrak{m}^{\psi_{\mu}}(x)dx= \frac{1}{\psi^2_{\mu}(c)} \int_0^y\psi_{\lambda}(x)\psi_{\mu}(x) \mathfrak{m}(x)dx<\infty. 
\end{align*}
\end{lem}

\begin{proof}
Since $\psi_{\mu}(y)$ is increasing and $\psi_{\mu}(y) \ge 0$ we have
\begin{align*}
\int_{0}^{y}\psi_\mu(x)\psi_{\lambda}(x)\mathfrak{m}(x)dx\le \sup_{0 \le x \le y}\psi_{\mu}(x)\int_{0}^{y}\psi_{\lambda}(x)\mathfrak{m}(x)dx<\infty,
\end{align*}
where the integral $\int_{0}^{y}\psi_{\lambda}(x)\mathfrak{m}(x)dx$ is seen to be finite from Table 1 on page 130 of \cite{ItoMckean}.

\end{proof}

\begin{proof}[Proof of Lemma \ref{EigenRel}]
By Lemma \ref{LemIntegrable} the integral in (\ref{EigenRelDisplay}) is finite. Now, apply the differential operator $\mathsf{D}_{\mathfrak{m}^{\psi_{\lambda_2}}}=\mathsf{a}(c)\mathcal{D}_{\widehat{\mathfrak{s}^{\psi_{\lambda_2}}}}$ to both sides of (\ref{EigenRelDisplay}). The left hand side is clearly given by $\frac{\psi_{\lambda_1}(y)}{\psi_{\lambda_2}(y)}$ while the right hand side is, using (\ref{GenRep}) and (\ref{EigenfunctionDef}),
\begin{align*}
(\lambda_1-\lambda_2)^{-1}\mathcal{D}_{\widehat{\mathfrak{s}^{\psi_{\lambda_2}}}}\mathsf{D}_{\widehat{\mathfrak{m}^{\psi_{\lambda_2}}}}\left(\frac{\psi_{\lambda_1}}{\psi_{\lambda_2}}\right)(y)&=(\lambda_1-\lambda_2)^{-1}\mathsf{D}_{\mathfrak{m}^{\psi_{\lambda_2}}}\mathcal{D}_{\mathfrak{s}^{\psi_{\lambda_2}}}\left(\frac{\psi_{\lambda_1}}{\psi_{\lambda_2}}\right)(y)\\
&=(\lambda_1-\lambda_2)^{-1}\mathsf{L}^{\psi_{\lambda_2}}\left(\frac{\psi_{\lambda_1}}{\psi_{\lambda_2}}\right)(y)
=\frac{\psi_{\lambda_1}(y)}{\psi_{\lambda_2}(y)}.
\end{align*}
Thus it suffices to check that they are equal at $y=0$. The left hand side is clearly $0$ while for the right hand side we can compute, using (\ref{GenRep}),
\begin{align*}
\mathsf{D}_{\widehat{\mathfrak{m}^{\psi_{\lambda_2}}}}\left(\frac{\psi_{\lambda_1}}{\psi_{\lambda_2}}\right)(y)=\frac{1}{\mathsf{a}(c)}\mathcal{D}_{\mathfrak{s}^{\psi_{\lambda_2}}}\left(\frac{\psi_{\lambda_1}}{\psi_{\lambda_2}}\right)(y)&=\frac{1}{\mathsf{a}(c)\mathfrak{s}^{\psi_{\lambda_2}}(y)}\frac{\psi_{\lambda_1}'(y)\psi_{\lambda_2}(y)-\psi_{\lambda_2}'(y)\psi_{\lambda_1}(y)}{\psi^2_{\lambda_2}(y)}\\
&=\frac{\psi_{\lambda}^2(c)}{\mathsf{a}(c)}\left[\psi_{\lambda_2}(y)\left(\mathcal{D}_{\mathfrak{s}}\psi_{\lambda_1}\right)(y)-\psi_{\lambda_1}(y)\left(\mathcal{D}_{\mathfrak{s}}\psi_{\lambda_2}\right)(y)\right].
\end{align*}
The claim now follows since $\lim_{y\to 0}\left(\mathcal{D}_{\mathfrak{s}}\psi_{\lambda}\right)(y)=0$, because $0$ is an entrance boundary point, see for example Table 1 in Section 4.6 page 130 in \cite{ItoMckean}.
\end{proof}

Before we continue to prove Theorem \ref{InteractingDiffusionsThm}, we introduce a lemma which we need for the proof. 

\begin{lem}\label{LemBoundaryBehaviour}
Let $\mathsf{L}$ be a one-dimensional diffusion process generator in $(0,\infty)$ so that $0$ is an entrance boundary point while $\infty$ is a natural boundary point and with positive increasing eigenfunctions $\psi_\lambda$ as in (\ref{EigenfunctionDef}). Moreover, assume that $\mathsf{L}^{\psi_{\lambda}}$ satisfies the (\textbf{YW}) condition. Then, $0$ is an entrance and $\infty$ is a natural boundary point for the $\mathsf{L}^{\psi_{\lambda}}$-diffusion as well.
\end{lem}

\begin{proof}
First note that by the (\textbf{YW}) condition the SDE for $\mathsf{L}^{\psi_\lambda}$-diffusion is non-exploding and thus $\infty$ is inaccessible. Moreover, since we clearly have
\begin{align*}
 \mathsf{b}(x)\le \mathsf{b}(x)+2\mathsf{a}(x)\frac{\psi'_\lambda(x)}{\psi_\lambda(x)},
\end{align*}
the standard comparison theorem for one-dimensional SDEs, see \cite{IkedaWatanabe}, gives us a coupling so that the $\mathsf{L}^{\psi_\lambda}$-diffusion $\mathsf{x}^{\lambda}$ and the $\mathsf{L}$-diffusion $\mathsf{x}$ are ordered: $\mathsf{x}^\lambda\ge \mathsf{x}$ almost surely. Thus, by comparing with the $\mathsf{L}$-diffusion we see that the $\mathsf{L}^{\psi_\lambda}$-diffusion cannot reach $0$ in finite time if started in $(0,\infty)$ while it reaches $(0,\infty)$ in finite time if started from $0$ and finally cannot reach $(0,\infty)$ if started from $\infty$. This gives the conclusion.
\end{proof}

Consider the following positive kernels $\Lambda_{2n,2n-1}$ and $\Lambda_{2n+1,2n}$ from $\mathbb{W}_n$ to $\mathbb{W}_n$ and from $\mathbb{W}_{n+1}$ to $\mathbb{W}_n$ respectively:
\begin{align}
 \Lambda_{2n,2n-1}\left(y,dx\right)&=\prod_{i=1}^n\mathfrak{m}^{\psi_{\lambda_{n+1}}}(x_i)\mathbf{1}_{\left(x\prec y\right)}dx=\frac{1}{\psi^{2n}_{\lambda_{n+1}}(c)}\prod_{i=1}^n\psi_{\lambda_{n+1}}^2(x_{i})\mathfrak{m}(x_i)\mathbf{1}_{\left(x\prec y\right)}dx,\label{PreMarkovKernel1}\\
  \Lambda_{2n+1,2n}\left(y,dx\right)&=\prod_{i=1}^n\widehat{\mathfrak{m}^{\psi_{\lambda_{n+1}}}}(x_i)\mathbf{1}_{\left(x\prec y\right)}dx=\psi^{2n}_{\lambda_{n+1}}(c)\prod_{i=1}^n\psi_{\lambda_{n+1}}^{-2}(x_{i})\hat{\mathfrak{m}}(x_i)\mathbf{1}_{\left(x\prec y\right)}dx.\label{PreMarkovKernel2}
\end{align}
We will transform them shortly so that they become Markov kernels. Towards this end define the following functions on $\overline{\mathbb{W}}_n$ (these clearly depend on the parameters $\lambda_i$ but we suppress it in the notation)
\begin{align}
 \mathsf{\Psi}^{(n)}(x)&=\frac{\det\left(\psi_{\lambda_i}(x_j)\right)_{i,j=1}^n}{\prod_{i=1}^n\psi_{\lambda_{n+1}}(x_i)}, \ \ \overline{\mathsf{\Psi}}^{(n)}(x)=\prod_{i=1}^n\frac{\psi_{\lambda_n}(x_i)}{\psi_{\lambda_{n+1}}(x_i)}\mathsf{\Psi}^{(n)}(x)=\frac{\det\left(\psi_{\lambda_i}(x_j)\right)_{i,j=1}^n}{\prod_{i=1}^n\psi_{\lambda_{n}}(x_i)}, \label{FunctionDisplay1}\\
 \widetilde{\mathsf{\Psi}}^{(n)}(x)&=(-1)^n\det\left(\mathsf{D}_{\widehat{\mathfrak{m}^{\psi_{\lambda_{n+1}}}}}\left(\frac{\psi_{\lambda_i}}{\psi_{\lambda_{n+1}}}\right)(x_j)\right)_{i,j=1}^n.\label{FunctionDisplay2}
\end{align}

With all these notations in place we can rewrite Lemma \ref{IntProp1} and Lemma \ref{IntProp2} as follows:
\begin{prop}\label{PropIntegration}
We have 
\begin{align}
 \left[\Lambda_{2n,2n-1}\mathsf{\Psi}^{(n)}\right](x)&=\mathsf{c}_n \widetilde{\mathsf{\Psi}}^{(n)}(x), \ \ \forall x\in \overline{\mathbb{W}}_n,\label{IntRel1}\\
 \left[\Lambda_{2n+1,2n}\widetilde{\mathsf{\Psi}}^{(n)}\right](x)&=\overline{\mathsf{\Psi}}^{(n+1)}(x),\ \ \forall x\in \overline{\mathbb{W}}_{n+1},\label{IntRel2}
\end{align}
where $\mathsf{c}_n$ is given by (\ref{constant}).
\end{prop}
 Note that by inductively using Proposition \ref{PropIntegration}, since $\mathfrak{m}$ and $\hat{\mathfrak{m}}$ are strictly positive on $(l,r)$, when restricted to $\mathbb{W}_n$ all the functions $\mathsf{\Psi}^{(n)}, \overline{\mathsf{\Psi}}^{(n)},  \widetilde{\mathsf{\Psi}}^{(n)}$ are strictly positive. We can thus make the following definition.

\begin{defn}\label{MarkovKernelDef}
We define the following Markov kernels $\mathfrak{L}_{2n,2n-1}$ and $\mathfrak{L}_{2n+1,2n}$ from $\mathbb{W}_n$ to $\mathbb{W}_n$ and from $\mathbb{W}_{n+1}$ to $\mathbb{W}_n$ respectively:
\begin{align*}
 \mathfrak{L}_{2n,2n-1}(y,dx)&=\frac{1}{\mathsf{c}_n\widetilde{\mathsf{\Psi}}^{(n)}(y)}\Lambda_{2n,2n-1}(y,dx)\mathsf{\Psi}^{(n)}(x),\\
 \mathfrak{L}_{2n+1,2n}(y,dx)&=\frac{1}{\overline{\mathsf{\Psi}}^{(n+1)}(y)}\Lambda_{2n+1,2n}(y,dx)\widetilde{\mathsf{\Psi}}^{(n)}(x).
\end{align*}
\end{defn}
We also define the Karlin-McGregor semigroups, see \cite{KarlinMcGregor,Karlin,ItoMckean}, $\mathcal{P}_t^n, \overline{\mathcal{P}}_t^n, \hat{\mathcal{P}}_t^n$ associated to $n$ independent $\mathsf{L}^{\psi_{\lambda_{n+1}}}$,  $\mathsf{L}^{\psi_{\lambda_{n}}}$ and $\widehat{\mathsf{L}^{\psi_{\lambda_{n+1}}}}$ (killed when they hit $0$) diffusions respectively in $\mathbb{W}_n$ with transition densities with respect to the Lebesgue measure in $\mathbb{W}_n$ given by
\begin{align*}
\mathcal{P}_t^n\left(x,y\right)&=\det\left(p_t^{\psi_{\lambda_{n+1}}}(x_i,y_j)\right)_{i,j=1}^n,\\
\overline{\mathcal{P}}_t^n\left(x,y\right)&=\det\left(p_t^{\psi_{\lambda_{n}}}(x_i,y_j)\right)_{i,j=1}^n,\\
\hat{\mathcal{P}}_t^n\left(x,y\right)&=\det\left(\widehat{p_t^{\psi_{\lambda_{n+1}}}}(x_i,y_j)\right)_{i,j=1}^n.
\end{align*}
From equation (26) in \cite{InterlacingDiffusions} we get the following intertwining relations
\begin{align}
\hat{\mathcal{P}}_t^n\Lambda_{2n,2n-1}&=\Lambda_{2n,2n-1} \mathcal{P}_t^n,\label{Inter1}\\
\overline{\mathcal{P}}_t^{n+1}\Lambda_{2n+1,2n}&=\Lambda_{2n+1,2n} \hat{\mathcal{P}}_t^n.\label{Inter2}
\end{align}
For (\ref{Inter1}) we have picked the $L$-diffusion therein to be our $\widehat{\mathsf{L}^{\psi_{\lambda_{n+1}}}}$-diffusion, with $(n_1,n_2)=(N,N)$ while for (\ref{Inter2}) we have picked the $L$-diffusion therein to be our $\mathsf{L}^{\psi_{\lambda_{n+1}}}$-diffusion, with $(n_1,n_2)=(N,N+1)$. The regularity condition there is satisfied because of (\ref{smoothness}) while the boundary condition there is satisfied because $0$ is entrance and $\infty$ is natural for our $\mathsf{L}$-diffusion and by virtue of Lemma \ref{LemBoundaryBehaviour} they are so for the $\mathsf{L}^{\psi_{\lambda_{n+1}}}$-diffusion as well.

Now, observe that $\mathsf{\Psi}^{(1)}$ is an eigenfunction of $\mathcal{P}_t^1$ with eigenvalue $e^{t(\lambda_1-\lambda_2)}$. The corresponding Doob $h$-transformation transforms a $\mathsf{L}^{\psi_{\lambda_2}}$-diffusion to a $\mathsf{L}^{\psi_{\lambda_1}}$-diffusion. Then, by induction using the intertwining relations above and (\ref{IntRel1}) and (\ref{IntRel2}) we get that $\mathsf{\Psi}^{(n)}$, $\overline{\mathsf{\Psi}}^{(n+1)}$ and $\widetilde{\mathsf{\Psi}}^{(n)}$ are strictly positive eigenfunctions, with eigenvalue $e^{t\sum_{i=1}^n(\lambda_i-\lambda_{n+1})}$, of $\mathcal{P}_t^n$, $\overline{\mathcal{P}}_t^{n+1}$ and $\hat{\mathcal{P}}_t^n$ respectively. Thus, we can correctly define the Doob $h$-transformed Karlin-McGregor semigroups (which are now bona fide Markov semigroups compared to sub-Markov) giving rise to non-intersecting paths in $\mathbb{W}_n$ with transition kernels given by
\begin{align}
\mathcal{P}_t^{n,\mathsf{\Psi}^{(n)}}\left(x,dy\right)&=e^{-\sum_{i=1}^n(\lambda_i-\lambda_{n+1})t}\frac{\mathsf{\Psi}^{(n)}(y)}{\mathsf{\Psi}^{(n)}(x)}\det\left(p_t^{\psi_{\lambda_{n+1}}}(x_i,y_j)\right)_{i,j=1}^ndy_1\cdots dy_n,\label{DoobKMsemi1}\\
\overline{\mathcal{P}}_t^{n,\overline{\mathsf{\Psi}}^{(n)}}\left(x,dy\right)&=e^{-\sum_{i=1}^{n-1}(\lambda_i-\lambda_{n})t}\frac{\overline{\mathsf{\Psi}}^{(n)}(y)}{\overline{\mathsf{\Psi}}^{(n)}(x)}\det\left(p_t^{\psi_{\lambda_{n}}}(x_i,y_j)\right)_{i,j=1}^ndy_1\cdots dy_n,\label{DoobKMsemi2}\\
\hat{\mathcal{P}}_t^{n,\widetilde{\mathsf{\Psi}}^{(n)}}\left(x,dy\right)&=e^{-\sum_{i=1}^n(\lambda_i-\lambda_{n+1})t}\frac{\widetilde{\mathsf{\Psi}}^{(n)}(y)}{\widetilde{\mathsf{\Psi}}^{(n)}(x)}\det\left(\widehat{p_t^{\psi_{\lambda_{n+1}}}}(x_i,y_j)\right)_{i,j=1}^ndy_1\cdots dy_n.\label{DoobKMsemi3}
\end{align}
\begin{rmk}
A simple computation, see for example Lemma 2 in \cite{RogersPitman}, using (\ref{Inter1}) and (\ref{Inter2}) gives the following intertwinings (that we record here although we will not make direct use of)
\begin{align*}
\hat{\mathcal{P}}_t^{n,\widetilde{\mathsf{\Psi}}^{(n)}}\mathfrak{L}_{2n,2n-1}&=\mathfrak{L}_{2n,2n-1} \mathcal{P}_t^{n,\mathsf{\Psi}^{(n)}},\\
\overline{\mathcal{P}}_t^{n+1,\overline{\mathsf{\Psi}}^{(n+1)}}\mathfrak{L}_{2n+1,2n}&=\mathfrak{L}_{2n+1,2n} \hat{\mathcal{P}}_t^{n,\widetilde{\mathsf{\Psi}}^{(n)}}.
\end{align*}
\end{rmk}

Finally, the following simple observation will be key for the constructions in Propositions \ref{TwoLevel1} and \ref{TwoLevel2}.

\begin{lem}\label{SemigroupObs}
With the notations above we have the equality of semigroups
\begin{align*}
 \mathcal{P}_t^{n,\mathsf{\Psi}^{(n)}}=\mathfrak{P}_t^{(n),(\lambda_1,\dots,\lambda_n)}  \ \ \textnormal{ and } \ \  \overline{\mathcal{P}}_t^{n+1,\overline{\mathsf{\Psi}}^{(n+1)}}=\mathfrak{P}_t^{(n+1),(\lambda_1,\dots,\lambda_{n+1})}, \ \ \forall t\ge 0,   
\end{align*}
where recall that $\mathfrak{P}_t^{(n),(\lambda_1,\dots,\lambda_n)}$ was defined in (\ref{ConditionedSemigroup}).
\end{lem}

\begin{proof} For the first equality of semigroups we observe that using (\ref{FunctionDisplay1}), the transition density in (\ref{DoobKMsemi1}) is equal to
\begin{align*}
 e^{-\sum_{i=1}^n \lambda_i t}e^{n \lambda_{n+1}t}\frac{\det\left(\psi_{\lambda_i}(y_j)\right)_{i,j=1}^n}{\det\left(\psi_{\lambda_i}(x_j)\right)_{i,j=1}^n}\prod_{i=1}^n\frac{\psi_{\lambda_{n+1}}(x_i)}{\psi_{\lambda_{n+1}}(y_i)}\times e^{-n\lambda_{n+1}t}\prod_{i=1}^n\frac{\psi_{\lambda_{n+1}}(y_i)}{\psi_{\lambda_{n+1}}(x_i)}\det \left(p_t(x_i,y_j)\right)_{i,j=1}^n,
\end{align*}
which equals (\ref{ConditionedSemigroup}). The second equality follows analogously from (\ref{DoobKMsemi2}), (\ref{FunctionDisplay1}) and (\ref{ConditionedSemigroup}).
\end{proof}

In order to proceed we require some terminology that makes more precise the notion of a diffusion reflected off continuous (non-intersecting) paths, see Section 5.1 in \cite{InterlacingDiffusions}.

Let $Q=a(x)\frac{d}{dx}+b(x)\frac{d}{dx}$ be a diffusion process generator in $(0,\infty)$ with both boundary points $0$ and $\infty$ being inaccessible. Suppose that we are given a process $(\mathsf{X}(t);t\ge 0)=\left(\left(\mathsf{X}_1(t),\dots,\mathsf{X}_N(t)\right);t\ge 0\right)$ of continuous non-intersecting paths in $\mathbb{W}_N$. Then, by $\mathsf{Y}$ is a system of $N+1$ $Q$-diffusions reflected off $\mathsf{X}$ in $\tilde{\mathbb{W}}_{N,N+1}$ we mean continuous processes $\left(\left(\mathsf{Y}_1(t),\dots,\mathsf{Y}_{N+1}(t)\right);t\ge 0\right)$ satisfying $\mathsf{Y}_1(t)\le \mathsf{X}_1(t)\le \mathsf{Y}_2(t)\le \cdots \le \mathsf{X}_N(t)\le \mathsf{Y}_{N+1}(t)$ for all $t\ge 0$ and so that the following SDEs hold:
\begin{align*}
  d\mathsf{Y}_1(t) &=\sqrt{2a\left(\mathsf{Y}_1(t)\right)}d\mathsf{w}_1(t)+b\left(\mathsf{Y}_1(t)\right)dt-d\mathfrak{k}_1^-(t),\\
  &\vdots\\
  d\mathsf{Y}_j(t)&=\sqrt{2a\left(\mathsf{Y}_j(t)\right)}d\mathsf{w}_j(t)+b\left(\mathsf{Y}_j(t)\right)dt+d\mathfrak{k}_j^+(t)-d\mathfrak{k}_j^-(t),\\
  &\vdots \\
 d\mathsf{Y}_{N+1}(t)&=\sqrt{2a\left(\mathsf{Y}_{N+1}(t)\right)}d\mathsf{w}_{N+1}(t)+b\left(\mathsf{Y}_{N+1}(t)\right)dt+d\mathfrak{k}_{N+1}^+(t),
\end{align*}
where the positive finite variation processes $\mathfrak{k}_j^+, \mathfrak{k}_j^-$ ($\mathfrak{k}_1^+$ and $\mathfrak{k}_{N+1}^-$ are identically zero) are such that $\mathfrak{k}_j^-$ increases only when $\mathsf{X}_j=\mathsf{Y}_j$ and $\mathfrak{k}_j^+$ increases only when $\mathsf{X}_{j-1}=\mathsf{Y}_j$ in order for $\left(\mathsf{X}(t),\mathsf{Y}(t)\right) \in \tilde{\mathbb{W}}_{N,N+1}$ forever. Here, $\mathsf{w}_1,\dots,\mathsf{w}_{N+1}$ are independent standard Brownian motions which are moreover independent of $\mathsf{X}$. Under (\textbf{YW}) the SDEs above have a unique strong solution in $\tilde{\mathbb{W}}_{N,N+1}$, see Section 5.1 in \cite{InterlacingDiffusions}. 

We also need the corresponding definition in $\tilde{\mathbb{W}}_{N,N}$. Let $Q=a(x)\frac{d}{dx}+b(x)\frac{d}{dx}$ be a one-dimensional diffusion process generator in $(0,\infty)$ with $\infty$ being inaccessible\footnote{We do not need that $0$ is inaccessible here since the smallest $\mathsf{Y}$ coordinate, namely $\mathsf{Y}_1$, is being kept away from $0$ by the continuous barrier $\mathsf{X}_1$ which by assumption stays in $(0,\infty)$.}. Suppose that we are given a process $(\mathsf{X}(t);t\ge 0)=\left(\left(\mathsf{X}_1(t),\dots,\mathsf{X}_N(t)\right);t\ge 0\right)$ of continuous non-intersecting paths in $\mathbb{W}_N$. Then, by $\mathsf{Y}$ is a system of $N$ $Q$-diffusions reflected off $\mathsf{X}$ in $\tilde{\mathbb{W}}_{N,N}$ we mean continuous processes $\left(\left(\mathsf{Y}_1(t),\dots,\mathsf{Y}_{N}(t)\right);t\ge 0\right)$ satisfying $\mathsf{X}_1(t)\le \mathsf{Y}_1(t)\le \mathsf{X}_2(t)\le \cdots \le \mathsf{X}_N(t)\le \mathsf{Y}_{N}(t)$ for all $t\ge 0$ and so that the following SDEs hold:
\begin{align*}
  d\mathsf{Y}_1(t) &=\sqrt{2a\left(\mathsf{Y}_1(t)\right)}d\mathsf{w}_1(t)+b\left(\mathsf{Y}_1(t)\right)dt+d\mathfrak{k}_1^+(t)-d\mathfrak{k}_1^-(t),\\
  &\vdots\\
  d\mathsf{Y}_j(t)&=\sqrt{2a\left(\mathsf{Y}_j(t)\right)}d\mathsf{w}_j(t)+b\left(\mathsf{Y}_j(t)\right)dt+d\mathfrak{k}_j^+(t)-d\mathfrak{k}_j^-(t),\\
  &\vdots \\
 d\mathsf{Y}_{N}(t)&=\sqrt{2a\left(\mathsf{Y}_{N}(t)\right)}d\mathsf{w}_{N}(t)+b\left(\mathsf{Y}_{N}(t)\right)dt+d\mathfrak{k}_{N}^+(t),
\end{align*}
where the positive finite variation processes $\mathfrak{k}_j^+, \mathfrak{k}_j^-$ ($\mathfrak{k}_N^-$ is identically zero) are such that $\mathfrak{k}_j^-$ increases only when $\mathsf{X}_{j+1}=\mathsf{Y}_j$ and $\mathfrak{k}_j^+$ increases only when $\mathsf{X}_{j}=\mathsf{Y}_j$ in order for $\left(\mathsf{X}(t),\mathsf{Y}(t)\right) \in \tilde{\mathbb{W}}_{N,N}$ forever. Here, $\mathsf{w}_1,\dots,\mathsf{w}_{N}$ are independent standard Brownian motions which are moreover independent of $\mathsf{X}$. Under (\textbf{YW}) the SDEs above have a unique strong solution in $\tilde{\mathbb{W}}_{N,N}$, see Section 5.1 in \cite{InterlacingDiffusions}.

With all these preliminaries in place we can state the following two results which are the basic building blocks of the construction in the interlacing array.

\begin{prop}\label{TwoLevel1} Let $\mathsf{L}$ be a one-dimensional diffusion process generator satisfying (\ref{smoothness}) in $(0,\infty)$ so that $0$ is an entrance boundary point while $\infty$ is a natural boundary point and with positive increasing eigenfunctions $\psi_\lambda$ as in (\ref{EigenfunctionDef}). Assume that $\mathsf{L}^{\psi_{\lambda_{N+1}}}$ and $\widehat{\mathsf{L}^{\psi_{\lambda_{N+1}}}}$ satisfy the (\textbf{YW}) condition.

Consider a two-level process $\left(\left(\mathsf{X}(t),\mathsf{Y}(t)\right);t\ge 0\right)$ in $\tilde{\mathbb{W}}_{N,N+1}$ with the $\mathsf{X}$-process evolving as $N$ non-intersecting paths with transition semigroup $\hat{\mathcal{P}}_t^{N,\widetilde{\mathsf{\Psi}}^{(N)}}$ and $\mathsf{Y}$ as a system of $N+1$ $\mathsf{L}^{\psi_{\lambda_{N+1}}}$-diffusions reflected off $\mathsf{X}$ in $\tilde{\mathbb{W}}_{N,N+1}$. 

Let $\mathcal{M}(dy)$ be a probability measure on $\mathbb{W}_{N+1}$ and assume that $\left(\left(\mathsf{X}(t),\mathsf{Y}(t)\right);t\ge 0\right)$ is initialized according to $\mathcal{M}(dy)\mathfrak{L}_{2N+1,2N}(y,dx)$ in $\mathbb{W}_{N,N+1}$. Then, $\left(\mathsf{Y}(t);t\ge 0\right)$ is distributed as a diffusion process with transition semigroup $\mathfrak{P}_t^{(N+1),(\lambda_1,\dots,\lambda_{N+1})}$ starting from $\mathcal{M}$. In particular, it consists of non-intersecting paths, namely for all $t\ge 0$: $\mathsf{Y}(t)\in \mathbb{W}_{N+1}$ and so $\left(\mathsf{X}(t),\mathsf{Y}(t)\right)\in \mathbb{W}_{N,N+1}$. Moreover, for any $T\ge 0$ the distribution of $\left(\mathsf{X}(T),\mathsf{Y}(T)\right)$  in $\mathbb{W}_{N,N+1}$ is given by $\left[\mathcal{M}\mathfrak{P}_T^{(N+1),(\lambda_1,\dots,\lambda_{N+1})}\right](dy)\mathfrak{L}_{2N+1,2N}\left(y,dx\right)$.
\end{prop}

\begin{proof}
We apply Theorem 2.19, see also Proposition 2.17 and Corollary 2.18, in \cite{InterlacingDiffusions} with the following correspondences. First, note that the role of the $X$ and $Y$ processes is reversed there compared to the current paper, namely $\mathsf{Y}=X$ and $\mathsf{X}=Y$ in our notation. Then, we pick the $L$-diffusion therein to be our $\mathsf{L}^{\psi_{\lambda_{N+1}}}$-diffusion. The regularity and boundary conditions therein are satisfied by our assumptions on the $\mathsf{L}$-diffusion above and by virtue of Lemma \ref{LemBoundaryBehaviour}. Finally, the strictly positive eigenfunction $\hat{h}_N$ for $N$ independent copies of $\hat{L}=\widehat{\mathsf{L}^{\psi_{\lambda_{N+1}}}}$ diffusions killed when they intersect (or hit $0$), namely associated to $\hat{\mathcal{P}}_t^N$, is given by our $\widetilde{\mathsf{\Psi}}^{(N)}$. The conclusion then follows by noting (\ref{IntRel2}) and Lemma \ref{SemigroupObs}.
\end{proof}

\begin{prop}\label{TwoLevel2}
Let $\mathsf{L}$ be a one-dimensional diffusion process generator satisfying (\ref{smoothness}) in $(0,\infty)$ so that $0$ is an entrance boundary point while $\infty$ is a natural boundary point and with positive increasing eigenfunctions $\psi_\lambda$ as in (\ref{EigenfunctionDef}). Assume that $\mathsf{L}^{\psi_{\lambda_{N+1}}}$ and $\widehat{\mathsf{L}^{\psi_{\lambda_{N+1}}}}$ satisfy the (\textbf{YW}) condition.

Consider a two-level process $\left(\left(\mathsf{X}(t),\mathsf{Y}(t)\right);t\ge 0\right)$ in $\tilde{\mathbb{W}}_{N,N}$ with the $\mathsf{X}$-process evolving as $N$ non-intersecting paths with transition semigroup $\mathfrak{P}_t^{(N),(\lambda_1,\dots,\lambda_N)}$ and $\mathsf{Y}$ as a system of $N$ $\widehat{\mathsf{L}^{\psi_{\lambda_{N+1}}}}$-diffusions reflected off $\mathsf{X}$ in $\tilde{\mathbb{W}}_{N,N}$. 

Let $\mathcal{M}(dy)$ be a probability measure on $\mathbb{W}_{N}$ and assume that $\left(\left(\mathsf{X}(t),\mathsf{Y}(t)\right);t\ge 0\right)$ is initialized according to $\mathcal{M}(dy)\mathfrak{L}_{2N,2N-1}(y,dx)$ in $\mathbb{W}_{N,N}$. Then, $\left(\mathsf{Y}(t);t\ge 0\right)$ is distributed as a diffusion process with semigroup $\hat{\mathcal{P}}_t^{N,\widetilde{\mathsf{\Psi}}^{(N)}}$ starting from $\mathcal{M}$. In particular, it consists of non-intersecting paths, namely for all $t\ge 0$: $\mathsf{Y}(t)\in \mathbb{W}_{N}$ and so $\left(\mathsf{X}(t),\mathsf{Y}(t)\right)\in \mathbb{W}_{N,N}$ . Moreover, for any $T\ge 0$ the distribution of $\left(\mathsf{X}(T),\mathsf{Y}(T)\right)$ in $\mathbb{W}_{N,N}$ is given by $\left[\mathcal{M}\hat{\mathcal{P}}_T^{N,\widetilde{\mathsf{\Psi}}^{(N)}}\right](dy)\mathfrak{L}_{2N,2N-1}\left(y,dx\right)$.
\end{prop}

\begin{proof}
We apply the results of Section 2.4 in \cite{InterlacingDiffusions} with the following correspondences. We take as the $L$-diffusion therein our $\widehat{\mathsf{L}^{\psi_{\lambda_{N+1}}}}$-diffusion. The regularity and boundary conditions therein are satisfied by our assumptions on the $\mathsf{L}$-diffusion above and by virtue of Lemma \ref{LemBoundaryBehaviour}. Finally, the strictly positive eigenfunction $\hat{h}_N$ for $N$ independent copies of $\hat{L}=\mathsf{L}^{\psi_{\lambda_{N+1}}}$ diffusions killed when they intersect, namely associated to $\mathcal{P}_t^N$, is given by our $\mathsf{\Psi}^{(N)}$. The conclusion then follows by noting (\ref{IntRel1}) and Lemma \ref{SemigroupObs}.

\end{proof}

\subsection{Multilevel dynamics in interlacing arrays}\label{SectionProofInteracting}

\begin{figure}[hbt!]
\centering
\captionsetup{singlelinecheck = false, justification=justified}
\begin{tikzpicture}
 
     \draw[red,fill] (0,0) circle [radius=0.1];   
     \draw[red,fill] (1,1) circle [radius=0.1];  
      \draw[red,fill] (0,2) circle [radius=0.1]; 
    \draw[red,fill] (2,2) circle [radius=0.1]; 
          \draw[red,fill] (1,3)circle [radius=0.1]; 
    \draw[red,fill] (3,3) circle [radius=0.1]; 
              \draw[red,fill] (0,4) circle [radius=0.1]; 
    \draw[red,fill] (2,4) circle [radius=0.1]; 
          \draw[red,fill] (4,4) circle [radius=0.1];

    \node[below right] at (0,0) {$\mathsf{x}_1^{(1)}$};
   \node[ left] at (0,0) {$\mathsf{L}^{\psi_{\lambda_1}}$};
   
   \draw[blue,ultra thick,->] (0.1,0.1) to (0.8,0.8);
   \node[below right] at (0.5,0.5) {$\textcolor{blue}{\mathfrak{l}_1^{(2),+}}$};
   
       \node[below right] at (1,1) {$\mathsf{x}_1^{(2)}$};
   \node[ left] at (0.9,1) {$\widehat{\mathsf{L}^{\psi_{\lambda_2}}}$};

      \node[below right] at (2,2) {$\mathsf{x}_2^{(3)}$};
   \node[ left] at (2,2) {$\mathsf{L}^{\psi_{\lambda_2}}$};

      \draw[blue,ultra thick,->] (1.1,1.1) to (1.8,1.8);
   \node[below right] at (1.5,1.5) {$\textcolor{blue}{\mathfrak{l}_2^{(3),+}}$};
 
       \node[ right] at (0.1,2) {$\mathsf{x}_1^{(3)}$};
   \node[ left] at (0,2) {$\mathsf{L}^{\psi_{\lambda_2}}$};

      \draw[blue,ultra thick,->] (0.9,1.1) to (0.2,1.8);
   \node[above right] at (0.5,1.3) {$\textcolor{blue}{\mathfrak{l}_1^{(3),-}}$};

    \node[below right] at (3,3) {$\mathsf{x}_2^{(4)}$};
   \node[ left] at (3,3) {$\widehat{\mathsf{L}^{\psi_{\lambda_3}}}$};

      \draw[blue,ultra thick,->] (2.1,2.1) to (2.8,2.8);
   \node[below right] at (2.5,2.5) {$\textcolor{blue}{\mathfrak{l}_2^{(4),+}}$};

       \node[ right] at (1.1,3) {$\mathsf{x}_1^{(4)}$};
   \node[ left] at (1,3) {$\widehat{\mathsf{L}^{\psi_{\lambda_3}}}$};
   
         \draw[blue,ultra thick,->] (1.9,2.1) to (1.2,2.8);
   \node[above right] at (1.5,2.3) {$\textcolor{blue}{\mathfrak{l}_1^{(4),-}}$};

            \draw[blue,ultra thick,->] (0.1,2.1) to (0.8,2.8);
   \node[above left] at (0.6,2.1) {$\textcolor{blue}{\mathfrak{l}_1^{(4),+}}$};

    \node[below right] at (4,4) {$\mathsf{x}_3^{(5)}$};
   \node[ left] at (4,4) {$\mathsf{L}^{\psi_{\lambda_3}}$};
   
    \node[right] at (2.1,4) {$\mathsf{x}_2^{(5)}$};
   \node[ left] at (2,4) {$\mathsf{L}^{\psi_{\lambda_3}}$};
   
    \node[right] at (0.1,4) {$\mathsf{x}_1^{(5)}$};
   \node[ left] at (0,4) {$\mathsf{L}^{\psi_{\lambda_3}}$};
   
               \draw[blue,ultra thick,->] (3.1,3.1) to (3.8,3.8);
   \node[below right] at (3.5,3.5) {$\textcolor{blue}{\mathfrak{l}_3^{(5),+}}$};
   
                 \draw[blue,ultra thick,->] (2.9,3.1) to (2.2,3.8);
   \node[above right] at (2.5,3.3) {$\textcolor{blue}{\mathfrak{l}_2^{(5),-}}$};

            \draw[blue,ultra thick,->] (1.1,3.1) to (1.8,3.8);
   \node[above left] at (1.65,3.2) {$\textcolor{blue}{\mathfrak{l}_2^{(5),+}}$};

          \draw[blue,ultra thick,->] (0.9,3.1) to (0.2,3.8);
   \node[ left] at (0.6,3.3) {$\textcolor{blue}{\mathfrak{l}_1^{(5),-}}$};

 \end{tikzpicture}
 \caption{A cartoon illustrating the dynamics (\ref{Dynamics}) for $N=3$ in $\mathfrak{HA}_{2N-1}$. On odd levels $2n-1$ we have $n$ $\mathsf{L}^{\psi_{\lambda_n}}$-diffusions while on even levels $2n$ we have $n$  $\widehat{\mathsf{L}^{\psi_{\lambda_{n+1}}}}$-diffusions. All of these are independent modulo the following interactions that keep them in $\mathfrak{HA}_{5}$. Coordinate $\mathsf{x}^{(2n)}_i$ gets an infinitesimal push (identified with $\frac{1}{2}d\mathfrak{l}^{(2n),+}_i$) to the right only at times when it collides with $\mathsf{x}_i^{(2n-1)}$ and an infinitesimal push (identified with $\frac{1}{2}d\mathfrak{l}^{(2n),-}_i$) to the left only at times when it collides with $\mathsf{x}_{i+1}^{(2n-1)}$ in order to remain between the barriers 
 $\mathsf{x}_i^{(2n-1)}$ and $\mathsf{x}_{i+1}^{(2n-1)}$ (away from the barriers it simply evolves as an $\widehat{\mathsf{L}^{\psi_{\lambda_{n+1}}}}$-diffusion). Similarly, coordinate $\mathsf{x}^{(2n-1)}_i$ gets an infinitesimal push (identified with $\frac{1}{2}d\mathfrak{l}^{(2n-1),+}_i$) to the right only at times when it collides with $\mathsf{x}_{i-1}^{(2n-2)}$ and an infinitesimal push (identified with $\frac{1}{2}d\mathfrak{l}^{(2n-1),-}_i$) to the left only at times when it collides with $\mathsf{x}_{i}^{(2n-2)}$ in order to remain between the barriers 
 $\mathsf{x}_{i-1}^{(2n-2)}$ and $\mathsf{x}_{i}^{(2n-2)}$ (away from the barriers it simply evolves as an $\mathsf{L}^{\psi_{\lambda_{n}}}$-diffusion). These interactions are denoted by the blue arrows in the figure.}\label{Figure2} 
 \end{figure}
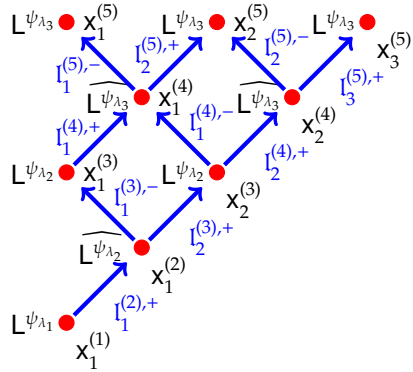

In this section we put everything together to obtain consistent dynamics in the interlacing half array. We consider the following system of SDEs in $\mathfrak{HA}_{2N-1}$
\begin{align}\label{Dynamics}
d\mathsf{x}_i^{(k)}(t)&=\sqrt{2\mathsf{a}\left(\mathsf{x}_i^{(k)}(t)\right)}d\mathsf{w}^{(k)}_i(t)+\mathsf{B}^{(k)}\left(\mathsf{x}_i^{(k)}(t)\right)dt+\frac{1}{2}d\mathfrak{l}^{(k),+}_i(t)-\frac{1}{2}d\mathfrak{l}^{(k),-}_i(t),
\end{align}
with $1\le k \le 2N-1$, $1\le i \le \left \lfloor \frac{k+1}{2}\right \rfloor$ and where the $\mathsf{w}_i^{(k)}$ are independent standard Brownian motions, the level dependent drift term $\mathsf{B}^{(k)}$ is given by
\begin{align}
\mathsf{B}^{(k)}(x)=\begin{cases}
\mathsf{b}\left(x\right)+2\mathsf{a}\left(x\right)\frac{\psi'_{\lambda_n}\left(x\right)}{\psi_{\lambda_n}\left(x\right)}, &\textnormal{ if }k=2n-1,\\
\mathsf{a}'\left(x\right)-\mathsf{b}\left(x\right)-2\mathsf{a}\left(x\right)\frac{\psi'_{\lambda_{n+1}}\left(x\right)}{\psi_{\lambda_{n+1}}\left(x\right)},  &\textnormal{ if }k=2n,\label{FunctionB}
\end{cases}
\end{align}
while the finite variation terms $\mathfrak{l}_i^{(k),\pm }$ can be identified\footnote{The fact that the finite variation terms $\mathfrak{l}_i^{(k),\pm}$ can be identified with the semimartingale local times at 0, see Section 5 of \cite{InterlacingDiffusions}, also Exercise 1.16 ($\textnormal{3}^\circ$) of Chapter VI of \cite{RevuzYor}, is of no significant importance here or in any of the works \cite{Warren,InterlacingDiffusions,Cerenzia} where analogous models have been studied. The statement simply becomes nicer. The only thing that really matters is that the finite variation terms $\mathfrak{k}_i^{(k),\pm}=\frac{1}{2}\mathfrak{l}_i^{(k),\pm }$ (to connect to the notation $\mathfrak{k}_i^{\pm}$ used when discussing diffusions reflected off general continuous non-intersecting paths) increase only when the corresponding coordinates collide so that they remain ordered.} with the semimartingale local times:
\begin{align}
\mathfrak{l}_i^{(k),+}=\begin{cases}
\textnormal{ sem. loc. time of } \mathsf{x}_i^{(k)}-\mathsf{x}^{(k-1)}_{i-1} \textnormal{ at 0 },&\textnormal{ if }k=2n-1,\\
\textnormal{ sem. loc. time of } \mathsf{x}_i^{(k)}-\mathsf{x}^{(k-1)}_{i}\textnormal{ at 0 },&\textnormal{ if } k=2n,
\end{cases}\label{SemLocTime1}\\
\mathfrak{l}_i^{(k),-}=\begin{cases}
\textnormal{ sem. loc. time of } \mathsf{x}_i^{(k)}-\mathsf{x}^{(k-1)}_{i}\textnormal{ at 0 },&\textnormal{ if }k=2n-1,\\
\textnormal{ sem. loc. time of } \mathsf{x}_i^{(k)}-\mathsf{x}^{(k-1)}_{i+1}\textnormal{ at 0 },&\textnormal{ if } k=2n,
\end{cases}\label{SemLocTime2}
\end{align}
and the terms for which the indices underflow or overflow in (\ref{SemLocTime1}) and (\ref{SemLocTime2}) are identically zero. In words the dynamics can be described as follows: levels $2n$ and $2n-1$ of the array consist of $n$ independent $\widehat{\mathsf{L}^{\psi_{\lambda_{n+1}}}}$ and $\mathsf{L}^{\psi_{\lambda_n}}$ diffusions reflected off (in the sense made precise in the previous section) the paths of level $2n-1$ and $2n-2$ respectively. See Figure \ref{Figure2} for an illustration of the interactions. By the results of Section 5 of \cite{InterlacingDiffusions} this SDE has a unique strong solution in $\mathfrak{HA}_{2N-1}$, under the assumption that all the $\mathsf{L}^{\psi_{\lambda_n}}$ and $\widehat{\mathsf{L}^{\psi_{\lambda_n}}}$-diffusions, equivalently $\left(\sqrt{a},\mathsf{B}^{(k)}\right)$ satisfy (\textbf{YW}), until the stopping time
\begin{align}\label{StoppingTime}
\tau_{\mathfrak{HA}_{2N-1}}=\inf\left\{t\ge 0:\exists \ 1\le k \le 2N-1 \textnormal { and } 1 \le i < j \le \left \lfloor \frac{k+1}{2} \right \rfloor \textnormal{ such that } \mathsf{x}^{(k)}_i(t)=\mathsf{x}^{(k)}_j(t)\right\}.
\end{align}
This stopping time corresponds to the problematic situation when two coordinates on the same level come together, in which case the coordinate on the next level that is between them becomes ``trapped" and pushed in opposing directions in an infinitesimally small interval. Under certain special initial conditions, that we define next, this issue does not arise.

\begin{defn}\label{DefinitionGibbs}
With the notations of this section, let $\mathcal{M}$ be a probability measure on $\mathbb{W}_{\left \lfloor \frac{N+1}{2} \right \rfloor}$. Then, by the Gibbs measure on $\mathfrak{HA}_N$ associated to $\mathcal{M}$, denoted by $\mathfrak{GM}_N^{\mathcal{M}}$, we mean the following probability measure
\begin{align}
 \mathfrak{GM}_N^{\mathcal{M}}\left(dx^{(1)},\dots,dx^{(N)}\right)=\mathcal{M}\left(dx^{(N)}\right) \prod_{i=1}^{N-1}\mathfrak{L}_{i+1,i}\left(x^{(i+1)},dx^{(i)}\right).
\end{align}
\end{defn}

\begin{rmk}
Clearly, although not indicated in the notation, $\mathfrak{GM}_N^{\mathcal{M}}$ is depended on the diffusion $\mathsf{L}$ and eigenvalues $\lambda_1,\dots,\lambda_N$ as well.
\end{rmk}


For the convenience of the reader we write out 
$\mathfrak{GM}_{2N-1}^{\mathcal{M}}$ explicitly.
\begin{prop}\label{ExplicitGibbs}
The explicit expression for $\mathfrak{GM}_{2N-1}^{\mathcal{M}}$ is the following
\begin{align*}
 \mathfrak{GM}_{2N-1}^{\mathcal{M}}\left(dx^{(1)},\dots,dx^{(2N-1)}\right)=\mathcal{M}\left(dx^{(2N-1)}\right)\mathsf{a}(c)^{\frac{N(N-1)}{2}}\frac{\prod_{1\le i < j \le N}(\lambda_j-\lambda_i)}{\det\left(\psi_{\lambda_i}\left(x_j^{(2N-1)}\right)\right)_{i,j=1}^N} \prod_{i=1}^N \psi_{\lambda_N}\left(x_i^{(2N-1)}\right)\\
 \times\prod_{n=1}^{N-1}\prod_{i=1}^n \psi_{\lambda_{n+1}}\left(x_i^{(2n-1)}\right)\psi_{\lambda_{n}}\left(x_i^{(2n-1)}\right)\mathfrak{m}\left(x_i^{(2n-1)}\right)\psi^{-2}_{\lambda_{n+1}}\left(x_i^{(2n)}\right)\hat{\mathfrak{m}}\left(x_i^{(2n)}\right)\\
 \times \mathbf{1}_{\left(\left(x^{(1)},\dots,x^{(2N-1)}\right)\in \mathfrak{HA}_{2N-1}\right)} dx^{(1)}\cdots dx^{(2N-2)},
\end{align*}
which is independent of $c\in (l,r)$.
\end{prop}
\begin{proof}
Observe that 
\begin{equation*}
   \mathfrak{GM}_{2N-1}^{\mathcal{M}}\left(dx^{(1)},\dots,dx^{(2N-1)}\right) =\mathcal{M}\left(dx^{(2N-1)}\right) \prod_{n=1}^{N-1}\mathfrak{L}_{2n+1,2n}\left(x^{(2n+1)},dx^{(2n)}\right) \mathfrak{L}_{2n,2n-1}\left(x^{(2n)},dx^{(2n-1)}\right)
\end{equation*}
where from Definition \ref{MarkovKernelDef}, by virtue of equations (\ref{PreMarkovKernel1}) and (\ref{PreMarkovKernel2}), we have
\begin{align*}
\mathfrak{L}_{2n+1,2n}\left(x^{(2n+1)},dx^{(2n)}\right) \mathfrak{L}_{2n,2n-1}\left(x^{(2n)},dx^{(2n-1)}\right) = \frac{1}{\mathsf{c}_n}\frac{\mathsf{\Psi}^{(n)}\left(x^{(2n-1)}\right)}{\overline{\mathsf{\Psi}}^{(n+1)}\left(x^{(2n+1)}\right)}\prod_{i=1}^n\hat{\mathfrak{m}}\left(x_i^{(2n)}\right)\mathfrak{m}\left(x_i^{(2n-1)}\right)\\
\prod_{i=1}^n \psi_{\lambda_{n+1}}^{-2}\left(x_i^{(2n)}\right)\psi_{\lambda_{n+1}}^{2}\left(x_i^{(2n-1)}\right) \mathbf{1}_{\left(x^{(2n-1)}\prec x^{(2n)} \prec x^{(2n+1)}\right)} dx^{(2n)}dx^{(2n-1)}.
\end{align*}
Moreover, by using equation (\ref{FunctionDisplay1}), we obtain
\begin{align*}
 \prod_{n=1}^{N-1}\frac{\mathsf{\Psi}^{(n)}\left(x^{(2n-1)}\right)}{\overline{\mathsf{\Psi}}^{(n+1)}\left(x^{(2n+1)}\right)}=\frac{\psi_{\lambda_1}\left(x_1^{(1)}\right)}{\det \left(\psi_{\lambda_i}\left(x_j^{(2N-1)}\right)\right)_{i,j=1}^N}\prod_{n=1}^{N-1}\prod_{i=1}^{n+1}\psi_{\lambda_{n+1}}\left(x_i^{(2n+1)}\right)\prod_{i=1}^n\frac{1}{\psi_{\lambda_{n+1}}\left(x_i^{(2n-1)}\right)}.
\end{align*}
By putting everything together, after some relabelling of the indices and recalling the definition of $\mathsf{c}_n$ from (\ref{constant}), we obtain the desired expression for $\mathfrak{GM}_{2N-1}^{\mathcal{M}}$. Finally, the fact that the expression is independent, as it should be, of the choice of $c\in (l,r)$ can be seen as follows. An easy computation gives $\hat{\mathfrak{m}}(z)=\frac{1}{\mathsf{a}(c)}\exp \left(-\int_{c}^z\frac{\mathsf{b}(y)}{\mathsf{a}(y)}dy\right)$. Then, we see that for any $z_1,z_2$ the expression
\begin{equation*}
 \mathsf{a}(c)\mathfrak{m}(z_1) \hat{\mathfrak{m}}(z_2)=\exp \left(\int_{z_2}^{z_1}\frac{\mathsf{b}(y)}{\mathsf{a}(y)}dy\right) 
\end{equation*}
is independent of the choice of $c\in (l,r)$, which by collecting terms as $\mathsf{a}(c)\mathfrak{m}\left(x_i^{(2n-1)}\right)\hat{\mathfrak{m}}\left(x_i^{(2n)}\right)$ gives the conclusion.
\end{proof}

We then have the following general result but before stating it let us recall the required ingredients. The basic data is a diffusion generator $\mathsf{L}$ (subject to certain conditions in the formal statement below) and a sequence of eigenvalues $\lambda_1<\dots<\lambda_N$. Associated to this data are strictly positive eigenfunctions $\psi_{\lambda_1},\dots, \psi_{\lambda_N}$ as in (\ref{EigenfunctionDef}), the density $\mathfrak{m}$  of the speed measure from (\ref{speedmeasure}) and the density $\hat{\mathfrak{m}}$ of the speed measure  of the dual diffusion $\widehat{\mathsf{L}}$ defined in (\ref{DualDiffusion}). Then, given a probability measure $\mathcal{M}$ on $\mathbb{W}_N$ we can construct the Gibbs measure $\mathfrak{GM}_{2N-1}^{\mathcal{M}}$ on $\mathfrak{HA}_{2N-1}$ through the explicit expression in Proposition \ref{ExplicitGibbs}. Moreover, from $\mathsf{L}$ and the eigenfunctions $\psi_{\lambda_i}$ we can construct interacting SDEs in $\mathfrak{HA}_{2N-1}$ as in (\ref{Dynamics}). Finally, recall the semigroup of diffusions conditioned to never intersect in (\ref{ConditionedSemigroup}).

\begin{prop}\label{GeneralInteracting}
Let $\mathsf{L}$ be a one-dimensional diffusion process generator satisfying (\ref{smoothness}) in $(0,\infty)$ so that $0$ is an entrance boundary point while $\infty$ is a natural boundary point and with positive increasing eigenfunctions $\psi_\lambda$ as in (\ref{EigenfunctionDef}).

Let $N\ge 1$ and $\lambda_1<\cdots <\lambda_N$ and assume that $\mathsf{L}^{\psi_{\lambda_n}}$ and $\widehat{\mathsf{L}^{\psi_{\lambda_{n}}}}$ satisfy the (\textbf{YW}) condition. Consider a probability measure $\mathcal{M}$ on $\mathbb{W}_N$ and suppose that the dynamics (\ref{Dynamics}) of the process $\left(\left(\mathsf{x}^{(1)}(t),\dots,\mathsf{x}^{(2N-1)}(t)\right);t\ge 0\right)$ in $\mathfrak{HA}_{2N-1}$ are initialized according to the Gibbs measure $\mathfrak{GM}_{2N-1}^{\mathcal{M}}$. 

Then, the projection on an odd level of the array $\left(\mathsf{x}^{(2n-1)}(t);t\ge 0\right)$ evolves according to $\mathfrak{P}^{(n),(\lambda_1,\dots,\lambda_n)}_t$ while on an even level $\left(\mathsf{x}^{(2n)}(t);t\ge 0\right)$ according to $\hat{\mathcal{P}}_t^{n,\widetilde{\mathsf{\Psi}}^{(n)}}$, in particular $\tau_{\mathfrak{HA}_{2N-1}}=\infty$. Finally, for a fixed time $T\ge 0$ the distribution of $\left(\mathsf{x}^{(1)}(T),\dots,\mathsf{x}^{(2N-1)}(T)\right)$ is given by the evolved Gibbs measure $\mathfrak{GM}_{2N-1}^{\mathcal{M}\mathfrak{P}_T^{(N),(\lambda_1,\dots,\lambda_N)}}$.
\end{prop}

\begin{proof}
The argument is standard and has been employed in a number of different settings \cite{Warren,WarrenWindridge,Toda,InterlacingDiffusions,Randomgrowth,Sun,Cerenzia}. Namely, the proof is by induction on the length of the interlacing array making alternating use of Proposition \ref{TwoLevel1} and Proposition \ref{TwoLevel2}. First, note that the base case is given by Proposition \ref{TwoLevel2} with $N=1$. Now, for the inductive step consider the maps $\Pi_{n-1}^n:\mathfrak{HA}_n\to \mathfrak{HA}_{n-1}$ given by $\Pi_{n-1}^n\left[\left(x^{(i)}\right)_{i=1}^{n}\right]=\left(x^{(i)}\right)_{i=1}^{n-1}$. Observe that
\begin{align*}
\left(\Pi_{2N-2}^{2N-1}\right)_*\mathfrak{GM}^\mathcal{M}_{2N-1}=\mathfrak{GM}_{2N-2}^{\mathcal{M}\mathfrak{L}_{2N-1,2N-2}}.
\end{align*}
We can thus use the inductive hypothesis on the process $\left(\mathsf{x}^{(1)},\dots,\mathsf{x}^{(2N-2)}\right)$. We then apply Proposition \ref{TwoLevel1} to the pair $\left(\mathsf{x}^{(2N-2)},\mathsf{x}^{(2N-1)}\right)$ from which the conclusion follows, noting that by construction $\mathsf{x}^{(2N-1)}$ is conditionally independent of $\left(\mathsf{x}^{(1)},\dots,\mathsf{x}^{(2N-3)}\right)$ given $\mathsf{x}^{(2N-2)}$.
\end{proof}


We finally prove Theorem \ref{InteractingDiffusionsThm} as a corollary of Proposition \ref{GeneralInteracting} above.

\begin{proof}[Proof of Theorem \ref{InteractingDiffusionsThm}]
We apply Proposition \ref{GeneralInteracting} with the $\mathsf{L}$-diffusion being a $\textnormal{BESQ}(\delta)$ process with $\delta\ge 2$. It is well-known, see \cite{RevuzYor,IkedaWatanabe}, that $0$ is an entrance boundary point for this parameter range while $\infty$ is always natural. Moreover, the (\textbf{YW}) condition for the $\mathsf{L}^{\psi_{\lambda_n}}$ and $\widehat{\mathsf{L}^{\psi_{\lambda_{n}}}}$ diffusions (for all $\delta>0$ in fact) readily follows from the one for $\textnormal{BESQ}(\delta)$ which is well-known \cite{RevuzYor,IkedaWatanabe}, except for the fact that the drift term coming from the Doob transform is Lipschitz which we check below. 

This boils down to showing that $x\mapsto \sqrt{x}\frac{I_{\nu+1}\left(\sqrt{x}\right)}{I_\nu\left(\sqrt{x}\right)}$ is Lipschitz on $[0,\infty)$. Now, some elementary computations using the standard identity $\frac{d}{dz}\left(z^\nu I_\nu(z)\right)=z^\nu I_{\nu-1}(z)$ give
\begin{align*}
    \frac{d}{dx}\left[\sqrt{x}\frac{I_{\nu+1}\left(\sqrt{x}\right)}{I_\nu\left(\sqrt{x}\right)}\right]=\frac{1}{2}\left[1-\frac{I_{\nu-1}(\sqrt{x})I_{\nu+1}(\sqrt{x})}{I_{\nu}(\sqrt{x})^2}\right].
\end{align*}
We note that, for $\nu=0$ we have $I_{-1}(z)=I_{1}(z)$ as a special case of the well-known symmetry of the modified Bessel function $I_{-n}(z)=I_n(z)$, with $n \in \mathbb{Z}$. The conclusion then follows from the fact that
\begin{align*}
  \sup_{x\in [0,\infty)}\frac{I_{\nu-1}(\sqrt{x})I_{\nu+1}(\sqrt{x})}{I_{\nu}(\sqrt{x})^2}<\infty,
\end{align*}
which is a consequence of the well-known asymptotics for Bessel functions at zero and infinity:
\begin{align*}
I_{\nu}(z)\sim \frac{e^z}{\sqrt{2\pi z}} \left(1+\mathcal{O}_\nu\left(\frac{1}{z}\right)\right) , \  \textnormal{ as } z \to \infty \ \ \textnormal{ and } \ \
I_{\nu}(z)\sim \frac{1}{\Gamma(\nu+1)2^\nu} z^\nu, \  \textnormal{ as } z \to 0, \ \textnormal{ for } \nu>-1.
\end{align*}
Note that, the asymptotics at zero above also cover $I_{-1}(z)$ since $I_{-1}(z)=I_{1}(z)$.
\end{proof}

\begin{rmk} Proposition \ref{GeneralInteracting} also applies to a more general explicit example. Consider the so-called squared radial Ornstein-Uhlenbeck process, also known as the Cox-Ingersoll-Ross (CIR) process in mathematical finance, see \cite{BorodinSalminen,SurveyBessel}, given by the strong solution to the SDE in $(0,\infty)$
\begin{align*}
  d\mathsf{x}(t)=2\sqrt{\mathsf{x}(t)}d\mathsf{w}(t)+\left(\delta-2\gamma \mathsf{x}(t)\right)dt,
\end{align*}
with $\mathsf{w}$ a standard Brownian motion, where we assume $\delta\ge 2$ and $\gamma\in \mathbb{R}$. The $\textnormal{BESQ}(\delta)$ process is the special case $\gamma=0$. This process is of particular interest when $\gamma>0$, in which case it is positive recurrent \cite{BorodinSalminen}. The eigenfunctions $\psi_{\lambda}$ of the corresponding generator $\mathsf{L}$ are explicit and given in terms of the confluent hypergeometric function, see \cite{BorodinSalminen,SurveyBessel}. Moreover, the transition density is also explicit, see \cite{BorodinSalminen}. It is known, see \cite{BorodinSalminen,SurveyBessel}, that for $\delta \ge 2$ and $\gamma\in \mathbb{R}$, $0$ is an entrance boundary point while $\infty$ is a natural boundary point. Also, (\ref{smoothness}) is clearly true. The only remaining condition of Proposition \ref{GeneralInteracting} to check is (\textbf{YW}) for the diffusions $\mathsf{L}^{\psi_{\lambda}}$ and $\widehat{\mathsf{L}^{\psi_{\lambda}}}$, which again boils down to showing Lipschitz continuity for the drift term coming from the Doob transform. As above, this can be shown using standard properties of the confluent hypergeometric function. We omit the details.
\end{rmk}

\section{Further results on interacting diffusions}\label{FurtherInteractingSection}

\subsection{The degenerate case: alternating construction}\label{SectionDegenerate}
In this section, expanding on Remark \ref{RemarkDegenerate}, we briefly consider the degenerate case when the parameters $\lambda_i$ are identically zero. We call this the alternating construction since we alternate between using an $\mathsf{L}$ or $\widehat{\mathsf{L}}$-diffusion on each level of the array, see Figure \ref{Figure3} for an illustration. In principle it should be possible to obtain the results presented next by taking the limit of parameters $\lambda_1,\dots,\lambda_N \to 0$ in the previous section's results. This turns out to be a singular limit for a number of the formulae involved and we do not try to justify\footnote{One needs to understand $\psi_\lambda(x)$ as a function of $\lambda$ (after the normalizing multiplicative constant is fixed), in particular as $\lambda \to 0$. This seems to be trickier in general than understanding its behaviour as a function of $x$. In the case of $\textnormal{BESQ}(\delta)$ we of course have the symmetry $x\leftrightarrow \lambda$: $\phi_\lambda^{(\nu)}(x)=\phi_{x}^{(\nu)}(\lambda)$, but this seems to be very special.} it rigorously here, see Remark \ref{LambdaParameterLimit} as well. Instead we redo the whole analysis from the beginning, partly because the situation is much more straightforward\footnote{It is worth noting that even after knowing the results in the degenerate case it is still not readily clear what the right way to introduce the parameters is (as we did in the previous sections).} compared to the setting of non-zero parameters. 

A discrete analogue of this alternating construction involving general birth and death chains interacting through the so-called push-block dynamics, which includes previous models related to representation theory \cite{BorodinKuan,Cerenzia} and classical orthogonal polynomials \cite{CerenziaKuan}, appeared in \cite{Randomgrowth}. One might expect that in the same scaling limit of a general birth and death chain to a general $\mathsf{L}$-diffusion the construction of \cite{Randomgrowth} should converge as a process to the one below. Such a statement has been proven in the case of the Brownian model of \cite{Warren} in \cite{GorinShkolnikov} and presumably by generalizing the results of \cite{GorinShkolnikov} it should be possible 
to confirm this prediction.

The dynamics in the degenerate case are given by the system of SDEs in $\mathfrak{HA}_{2N-1}$
\begin{align}\label{DegenerateDynamics}
d\mathsf{x}_i^{(k)}(t)&=\sqrt{2\mathsf{a}\left(\mathsf{x}_i^{(k)}(t)\right)}d\mathsf{w}^{(k)}_i(t)+\mathfrak{B}^{(k)}\left(\mathsf{x}_i^{(k)}(t)\right)dt+\frac{1}{2}d\mathfrak{l}^{(k),+}_i(t)-\frac{1}{2}d\mathfrak{l}^{(k),-}_i(t),
\end{align}
with $1\le k \le 2N-1$, $1\le i \le \left \lfloor \frac{k+1}{2}\right \rfloor$ and where the $\mathsf{w}_i^{(k)}$ are independent standard Brownian motions, the level dependent drift term $\mathfrak{B}^{(k)}$ is given by:
\begin{align}
\mathfrak{B}^{(k)}(x)=\begin{cases}
\mathsf{b}\left(x\right), &\textnormal{ if }k=2n-1,\\
\mathsf{a}'\left(x\right)-\mathsf{b}\left(x\right),  &\textnormal{ if }k=2n,\label{FunctionBfrak}
\end{cases}
\end{align}
and the finite variation terms $\mathfrak{l}^{(k),\pm}_i$ can be identified with the corresponding semimartingale local times (\ref{SemLocTime1}), (\ref{SemLocTime2}) as before. See Figure \ref{Figure3} for an illustration of the interactions. Again, see Section 5 in \cite{InterlacingDiffusions}, these SDEs have a unique strong solution, under the (\textbf{YW}) condition for $\left(\sqrt{\mathsf{a}},\mathsf{b}\right)$ and $\left(\sqrt{\mathsf{a}},\mathsf{a}'-\mathsf{b}\right)$, until the stopping time $\tau_{\mathfrak{HA}_{2N-1}}$ from (\ref{StoppingTime}) when two coordinates on the same level collide. As before, as we see below for Gibbs initial conditions this issue does not arise.

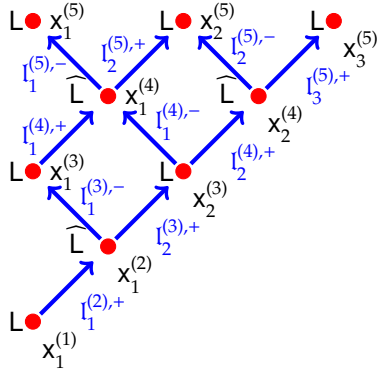
\begin{figure}[hbt!]
\centering
\captionsetup{singlelinecheck = false, justification=justified}
\begin{tikzpicture}
 
     \draw[red,fill] (0,0) circle [radius=0.1];   
     \draw[red,fill] (1,1) circle [radius=0.1];  
      \draw[red,fill] (0,2) circle [radius=0.1]; 
    \draw[red,fill] (2,2) circle [radius=0.1]; 
          \draw[red,fill] (1,3)circle [radius=0.1]; 
    \draw[red,fill] (3,3) circle [radius=0.1]; 
              \draw[red,fill] (0,4) circle [radius=0.1]; 
    \draw[red,fill] (2,4) circle [radius=0.1]; 
          \draw[red,fill] (4,4) circle [radius=0.1];

    \node[below right] at (0,0) {$\mathsf{x}_1^{(1)}$};
   \node[ left] at (0,0) {$\mathsf{L}$};
   
   \draw[blue,ultra thick,->] (0.1,0.1) to (0.8,0.8);
   \node[below right] at (0.5,0.5) {$\textcolor{blue}{\mathfrak{l}_1^{(2),+}}$};
   
       \node[below right] at (1,1) {$\mathsf{x}_1^{(2)}$};
   \node[ left] at (0.8,1.1) {$\widehat{\mathsf{L}}$};

      \node[below right] at (2,2) {$\mathsf{x}_2^{(3)}$};
   \node[ left] at (2,2) {$\mathsf{L}$};

      \draw[blue,ultra thick,->] (1.1,1.1) to (1.8,1.8);
   \node[below right] at (1.5,1.5) {$\textcolor{blue}{\mathfrak{l}_2^{(3),+}}$};
 
       \node[ right] at (0.1,2) {$\mathsf{x}_1^{(3)}$};
   \node[ left] at (0,2) {$\mathsf{L}$};

      \draw[blue,ultra thick,->] (0.9,1.1) to (0.2,1.8);
   \node[above right] at (0.5,1.3) {$\textcolor{blue}{\mathfrak{l}_1^{(3),-}}$};

    \node[below right] at (3,3) {$\mathsf{x}_2^{(4)}$};
   \node[ left] at (2.8,3.1) {$\widehat{\mathsf{L}}$};

      \draw[blue,ultra thick,->] (2.1,2.1) to (2.8,2.8);
   \node[below right] at (2.5,2.5) {$\textcolor{blue}{\mathfrak{l}_2^{(4),+}}$};

       \node[ right] at (1.1,3) {$\mathsf{x}_1^{(4)}$};
   \node[ left] at (0.8,3.1) {$\widehat{\mathsf{L}}$};
   
         \draw[blue,ultra thick,->] (1.9,2.1) to (1.2,2.8);
   \node[above right] at (1.5,2.3) {$\textcolor{blue}{\mathfrak{l}_1^{(4),-}}$};

            \draw[blue,ultra thick,->] (0.1,2.1) to (0.8,2.8);
   \node[above left] at (0.6,2.1) {$\textcolor{blue}{\mathfrak{l}_1^{(4),+}}$};

    \node[below right] at (4,4) {$\mathsf{x}_3^{(5)}$};
   \node[ left] at (4,4) {$\mathsf{L}$};
   
    \node[right] at (2.1,4) {$\mathsf{x}_2^{(5)}$};
   \node[ left] at (2,4) {$\mathsf{L}$};
   
    \node[right] at (0.1,4) {$\mathsf{x}_1^{(5)}$};
   \node[ left] at (0,4) {$\mathsf{L}$};
   
               \draw[blue,ultra thick,->] (3.1,3.1) to (3.8,3.8);
   \node[below right] at (3.5,3.5) {$\textcolor{blue}{\mathfrak{l}_3^{(5),+}}$};
   
                 \draw[blue,ultra thick,->] (2.9,3.1) to (2.2,3.8);
   \node[above right] at (2.5,3.3) {$\textcolor{blue}{\mathfrak{l}_2^{(5),-}}$};

            \draw[blue,ultra thick,->] (1.1,3.1) to (1.8,3.8);
   \node[above left] at (1.65,3.2) {$\textcolor{blue}{\mathfrak{l}_2^{(5),+}}$};

          \draw[blue,ultra thick,->] (0.9,3.1) to (0.2,3.8);
   \node[ left] at (0.6,3.3) {$\textcolor{blue}{\mathfrak{l}_1^{(5),-}}$};

 \end{tikzpicture}
 \caption{A cartoon illustrating the dynamics (\ref{DegenerateDynamics}) for $N=3$ in $\mathfrak{HA}_{2N-1}$. On odd levels $2n-1$ we have $n$ $\mathsf{L}$-diffusions while on even levels $2n$ we have $n$  $\widehat{\mathsf{L}}$-diffusions. All of these are independent modulo the interactions, as described in the completely analogous setting of Figure \ref{Figure3}, when they collide, that keep them in $\mathfrak{HA}_{5}$.}\label{Figure3}
 \end{figure}

In what follows we  will slightly abuse notation and for the most part use the same notations as in the preceding sections. Consider the following positive kernels $\Lambda_{2n,2n-1}$ and $\Lambda_{2n+1,2n}$ from $\mathbb{W}_n$ to $\mathbb{W}_n$ and from $\mathbb{W}_{n+1}$ to $\mathbb{W}_n$ respectively:
\begin{align*}
 \Lambda_{2n,2n-1}\left(y,dx\right)=\prod_{i=1}^n\mathfrak{m}(x_i)\mathbf{1}_{\left(x\prec y\right)}dx,\ \ \
  \Lambda_{2n+1,2n}\left(y,dx\right)=\prod_{i=1}^n\hat{\mathfrak{m}}(x_i)\mathbf{1}_{\left(x\prec y\right)}dx.
\end{align*}
Define the following functions, strictly positive on $\mathbb{W}_n$, with the convention $\mathfrak{h}^{(1)}(x)\equiv 1$:
\begin{align*}
 \mathfrak{h}^{(n)}(x)&=\left[\Lambda_{2n-1,2n-2}\Lambda_{2n-2,2n-3}\cdots \Lambda_{3,2}\Lambda_{2,1}1\right](x), \ \ x\in \overline{\mathbb{W}}_n,\\  
 \hat{\mathfrak{h}}^{(n)}(x)&=\left[\Lambda_{2n,2n-1}\Lambda_{2n-1,2n-2}\cdots \Lambda_{3,2}\Lambda_{2,1}1\right](x), \ \ x\in \overline{\mathbb{W}}_n.
\end{align*}

\begin{rmk}\label{RmkChebyshev}
By induction it is easy to see that $\mathfrak{h}^{(n)}$ and $\hat{\mathfrak{h}}^{(n)}$ have representations as determinants. More precisely, there exist sequences of functions $\{h_i\}_{i=1}^\infty$ and $\{ \hat{h}_i\}_{i=1}^\infty$ obtained recursively, which have a multiple integral representation\footnote{This is reminiscent to the representation of extended complete Chebyshev systems from approximation theory, see \cite{Karlin}.} (that we do not present here) so that, for any $n\ge 1$ and $x\in \mathbb{W}_n$:
\begin{align*}
  \mathfrak{h}^{(n)}(x)=\det\left(h_i(x_j)\right)_{i,j=1}^n, \ \ 
  \hat{\mathfrak{h}}^{(n)}(x)=\det\left(\hat{h}_i(x_j)\right)_{i,j=1}^n.
\end{align*}
\end{rmk}

\begin{rmk}\label{LambdaParameterLimit}
With the notations of Remark \ref{RmkChebyshev} above, we should have
\begin{align*}
 \lim_{(\lambda_1,\dots,\lambda_n)\to (0,\dots,0)}\frac{\det\left(\psi_{\lambda_i}(y_j)\right)_{i,j=1}^n}{\det\left(\psi_{\lambda_i}(x_j)\right)_{i,j=1}^n}=\frac{\det\left(h_i(y_j)\right)_{i,j=1}^n}{\det\left(h_i(x_j)\right)_{i,j=1}^n}.
\end{align*}
We do not attempt to justify this limit in the general setting of this paper. Nevertheless, for the $\textnormal{BESQ}(\delta)$ case it can be checked directly using the explicit formula for $\phi_{\frac{\mu}{2}}^{(\nu)}(x)$, in which case we get (as expected) the ratio of Vandermonde determinants $\frac{\mathsf{\Delta}_n(y)}{\mathsf{\Delta}_n(x)}$.
\end{rmk}

We now define the following Markov kernels $\mathfrak{L}_{2n,2n-1}$ and $\mathfrak{L}_{2n+1,2n}$ from $\mathbb{W}_n$ to $\mathbb{W}_n$ and from $\mathbb{W}_{n+1}$ to $\mathbb{W}_n$ respectively:
\begin{align}
 \mathfrak{L}_{2n,2n-1}(y,dx)&=\frac{1}{\hat{\mathfrak{h}}^{(n)}(y)}\Lambda_{2n,2n-1}(y,dx)\mathfrak{h}^{(n)}(x)\label{MarkovDegenerate1},\\
 \mathfrak{L}_{2n+1,2n}(y,dx)&=\frac{1}{\mathfrak{h}^{(n+1)}(y)}\Lambda_{2n+1,2n}(y,dx)\hat{\mathfrak{h}}^{(n)}(x).\label{MarkovDegenerate2}
\end{align}
We also define the Karlin-McGregor semigroups, see \cite{KarlinMcGregor,Karlin,ItoMckean}, $\mathcal{P}_t^n, \hat{\mathcal{P}}_t^n$ associated to $n$ independent $\mathsf{L}$ and $\widehat{\mathsf{L}}$ (killed when they hit $0$) diffusions in $\mathbb{W}_n$  respectively with transition densities given by
\begin{align*}
\mathcal{P}_t^n\left(x,y\right)=\det\left(p_t(x_i,y_j)\right)_{i,j=1}^n, \ \ \ \ 
\hat{\mathcal{P}}_t^n\left(x,y\right)=\det\left(\hat{p}_t(x_i,y_j)\right)_{i,j=1}^n.
\end{align*}
As in the previous section, from equation (26) in \cite{InterlacingDiffusions} or simply by taking the, non-singular in this case, limit $\lambda_{n+1}\to 0$ in (\ref{Inter1}) and (\ref{Inter2}) we get the following intertwining relations
\begin{align*}
\hat{\mathcal{P}}_t^n\Lambda_{2n,2n-1}&=\Lambda_{2n,2n-1} \mathcal{P}_t^n,\\
\mathcal{P}_t^{n+1}\Lambda_{2n+1,2n}&=\Lambda_{2n+1,2n} \hat{\mathcal{P}}_t^n.
\end{align*}
Finally, by induction using the intertwining relations above we get that $\mathfrak{h}^{(n)}$ and $\hat{\mathfrak{h}}^{(n)}$ are strictly positive in $\mathbb{W}_n$ harmonic (or invariant) functions of $\mathcal{P}_t^n$ and $\hat{\mathcal{P}}_t^n$ respectively. Thus, we can correctly define the Doob $h$-transformed Karlin-McGregor semigroups in $\mathbb{W}_n$ with transition kernels given by
\begin{align*}
\mathcal{P}_t^{n,\mathfrak{h}^{(n)}}\left(x,dy\right)&=\frac{\mathfrak{h}^{(n)}(y)}{\mathfrak{h}^{(n)}(x)}\det\left(p_t(x_i,y_j)\right)_{i,j=1}^ndy_1\cdots dy_n,\\
\hat{\mathcal{P}}_t^{n,\hat{\mathfrak{h}}^{(n)}}\left(x,dy\right)&=\frac{\hat{\mathfrak{h}}^{(n)}(y)}{\hat{\mathfrak{h}}^{(n)}(x)}\det\left(\hat{p}_t(x_i,y_j)\right)_{i,j=1}^ndy_1\cdots dy_n.
\end{align*}

\begin{rmk}
One might expect that $\mathfrak{h}^{(N)}(x)=\det\left(h_i(x_j)\right)_{i,j=1}^N$ appears in the large time $t$ asymptotics of $\mathbb{P}_x\left(\tau_C>t\right)$ for the first collision time $\tau_C$ of $N$ independent $\mathsf{L}$-diffusions starting from $x\in \mathbb{W}_N$. Then, the semigroup $\mathcal{P}_t^{N,\mathfrak{h}^{(N)}}$ would correspond to $N$ independent $\mathsf{L}$-diffusions conditioned to never intersect. It is well-known, see \cite{OConnell}, that for $\textnormal{BESQ}(\delta)$ this is indeed the case. It would be interesting to investigate if and under which conditions this is true in the general setting.
\end{rmk} 

We then have the following analogues of Propositions \ref{TwoLevel1} and \ref{TwoLevel2}.

\begin{prop}\label{TwolevelDegen1} Let $\mathsf{L}$ be a one-dimensional diffusion process generator satisfying (\ref{smoothness}) in $(0,\infty)$ so that $0$ is an entrance boundary point while $\infty$ is a natural boundary point. Assume that $\mathsf{L}$ and $\widehat{\mathsf{L}}$ satisfy the (\textbf{YW}) condition.

Consider a two-level process $\left(\left(\mathsf{X}(t),\mathsf{Y}(t)\right);t\ge 0\right)$ in $\tilde{\mathbb{W}}_{N,N+1}$ with the $\mathsf{X}$-process evolving as $N$ non-intersecting paths with transition semigroup $\hat{\mathcal{P}}_t^{N,\hat{\mathfrak{h}}^{(N)}}$ and $\mathsf{Y}$ as a system of $N+1$ $\mathsf{L}$-diffusions reflected off $\mathsf{X}$ in $\tilde{\mathbb{W}}_{N,N+1}$. 

Let $\mathcal{M}(dy)$ be a probability measure on $\mathbb{W}_{N+1}$ and assume that $\left(\left(\mathsf{X}(t),\mathsf{Y}(t)\right);t\ge 0\right)$ is initialized according to $\mathcal{M}(dy)\mathfrak{L}_{2N+1,2N}(y,dx)$ in $\mathbb{W}_{N,N+1}$. Then, $\left(\mathsf{Y}(t);t\ge 0\right)$ is distributed as a diffusion process with transition semigroup $\mathcal{P}_t^{N+1,\mathfrak{h}^{(N+1)}}$ starting from $\mathcal{M}$. In particular, it consists of non-intersecting paths, namely for all $t\ge 0$: $\mathsf{Y}(t)\in \mathbb{W}_{N+1}$ and so $\left(\mathsf{X}(t),\mathsf{Y}(t)\right)\in \mathbb{W}_{N,N+1}$ . Moreover, for any $T\ge 0$ the distribution of $\left(\mathsf{X}(T),\mathsf{Y}(T)\right)$  in $\mathbb{W}_{N,N+1}$ is given by $\left[\mathcal{M}\mathcal{P}_T^{N+1,\mathfrak{h}^{(N+1)}}\right](dy)\mathfrak{L}_{2N+1,2N}\left(y,dx\right)$.
\end{prop}

\begin{proof}
We apply the results of Section 2.4 of \cite{InterlacingDiffusions} with the $L$-diffusion therein being our $\mathsf{L}$-diffusion and the strictly positive eigenfunction (in this case harmonic) $\hat{h}_N$ therein given by our $\hat{\mathfrak{h}}^{(N)}$.
\end{proof}

\begin{prop}\label{TwolevelDegen2} Let $\mathsf{L}$ be a one-dimensional diffusion process generator satisfying (\ref{smoothness}) in $(0,\infty)$ so that $0$ is an entrance boundary point while $\infty$ is a natural boundary point. Assume that $\mathsf{L}$ and $\widehat{\mathsf{L}}$ satisfy the (\textbf{YW}) condition.

Consider a two-level process $\left(\left(\mathsf{X}(t),\mathsf{Y}(t)\right);t\ge 0\right)$ in $\tilde{\mathbb{W}}_{N,N}$ with the $\mathsf{X}$-process evolving as $N$ non-intersecting paths with transition semigroup $\mathcal{P}_t^{N,\mathfrak{h}^{(N)}}$ and $\mathsf{Y}$ as a system of $N$ $\widehat{\mathsf{L}}$-diffusions reflected off $\mathsf{X}$ in $\tilde{\mathbb{W}}_{N,N}$. 

Let $\mathcal{M}(dy)$ be a probability measure  on $\mathbb{W}_{N}$ and assume that $\left(\left(\mathsf{X}(t),\mathsf{Y}(t)\right);t\ge 0\right)$ is initialized according to $\mathcal{M}(dy)\mathfrak{L}_{2N,2N-1}(y,dx)$ in $\mathbb{W}_{N,N}$. Then, $\left(\mathsf{Y}(t);t\ge 0\right)$ is distributed as a diffusion process with semigroup $\hat{\mathcal{P}}_t^{N,\hat{\mathfrak{h}}^{(N)}}$ starting from $\mathcal{M}$. In particular, it consists of non-intersecting paths, namely for all $t\ge 0$: $\mathsf{Y}(t)\in \mathbb{W}_{N}$ and so $\left(\mathsf{X}(t),\mathsf{Y}(t)\right)\in \mathbb{W}_{N,N}$ . Moreover, for any $T\ge 0$ the distribution of $\left(\mathsf{X}(T),\mathsf{Y}(T)\right)$  in $\mathbb{W}_{N,N}$ is given by $\left[\mathcal{M}\hat{\mathcal{P}}_T^{N,\hat{\mathfrak{h}}^{(N)}}\right](dy)\mathfrak{L}_{2N,2N-1}\left(y,dx\right)$

\end{prop}

\begin{proof}
We apply the results of Section 2.4 of \cite{InterlacingDiffusions} with the $L$-diffusion therein being our $\widehat{\mathsf{L}}$-diffusion and the strictly positive eigenfunction (in this case harmonic) $\hat{h}_N$ therein given by our $\mathfrak{h}^{(N)}$.
\end{proof}

Finally, we define a Gibbs measure $\mathfrak{GM}^{\mathcal{M}}_N$ as in Definition \ref{DefinitionGibbs} but with the Markov kernels $\mathfrak{L}_{i+1,i}$ given by (\ref{MarkovDegenerate1}), (\ref{MarkovDegenerate2}). We have the following analogue of Proposition \ref{GeneralInteracting}.

\begin{prop}
Let $\mathsf{L}$ be a one-dimensional diffusion process generator satisfying (\ref{smoothness}) in $(0,\infty)$ so that $0$ is an entrance boundary point while $\infty$ is a natural boundary point. Assume that $\mathsf{L}$ and $\widehat{\mathsf{L}}$ satisfy the (\textbf{YW}) condition.

Let $N\ge 1$ and consider a probability measure $\mathcal{M}$ on $\mathbb{W}_N$. Suppose that the dynamics (\ref{DegenerateDynamics}) of the process $\left(\left(\mathsf{x}^{(1)}(t),\dots,\mathsf{x}^{(2N-1)}(t)\right);t\ge 0\right)$ in $\mathfrak{HA}_{2N-1}$ are initialized according to the Gibbs measure $\mathfrak{GM}_{2N-1}^{\mathcal{M}}$. 

Then, the projection on an odd level of the array $\left(\mathsf{x}^{(2n-1)}(t);t\ge 0\right)$ evolves according to $\mathcal{P}_t^{n,\mathfrak{h}^{(n)}}$ while on an even level $\left(\mathsf{x}^{(2n)}(t);t\ge 0\right)$ according to $\hat{\mathcal{P}}_t^{n,\hat{\mathfrak{h}}^{(n)}}$, in particular $\tau_{\mathfrak{HA}_{2N-1}}=\infty$. Finally, for a fixed time $T\ge 0$ the distribution of $\left(\mathsf{x}^{(1)}(T),\dots,\mathsf{x}^{(2N-1)}(T)\right)$ is given by the evolved Gibbs measure $\mathfrak{GM}_{2N-1}^{\mathcal{M}\mathcal{P}_T^{N,\mathfrak{h}^{(N)}}}$.

\end{prop}

\begin{proof}
Same as Proposition \ref{GeneralInteracting} by recursive use of Proposition \ref{TwolevelDegen1} and Proposition \ref{TwolevelDegen2}.
\end{proof}

\subsection{On entrance laws}\label{SectionEntrance}
As alluded to in Remark \ref{RemarkEntrance}, we would also like to start the dynamics (\ref{Dynamics}) from a singular point, when some or all of the coordinates of the top level coincide, using an entrance law, see Chapter XII in \cite{RevuzYor} for more details regarding this notion. For simplicity we will only consider the case of all coordinates coinciding. Our main interest is to start the top process from the origin, in which case, due to the interlacing, all entries of the array $\mathfrak{HA}_{2N-1}$ are equal to zero.

Towards this end we first let $x_*\in [0,\infty)$. Our main assumption in this section is the following:
\begin{itemize}
    \item We assume, and denote this condition by (\textbf{E}), that it is possible to start the diffusion process associated to $\mathfrak{P}_t^{(N),(\lambda_1,\dots,\lambda_N)}$ from the singular point $(x_*,\dots,x_*)$, namely that the following probability measures are well-defined
\begin{align*}
    \mu_t^{x_*}(dy)=\mu_t^{x_*,(N),(\lambda_1,\dots,\lambda_N)}(dy)\overset{\textnormal{def}}{=}\lim_{(x_1,\dots,x_N)\to(x_*,\dots,x_*)} \mathfrak{P}_t^{(N),(\lambda_1,\dots,\lambda_N)}(x,dy), \ \ t>0,
\end{align*}
and\footnote{Observe that modulo some domination this requirement is simply a consequence of the semigroup property $\mathfrak{P}_t^{(N),(\lambda_1,\dots,\lambda_N)}\mathfrak{P}_s^{(N),(\lambda_1,\dots,\lambda_N)}=\mathfrak{P}_{t+s}^{(N),(\lambda_1,\dots,\lambda_N)}$.} that they form an entrance law: for any $t>0$ and $s\ge 0$ we have $ \mu_t^{x_*}\mathfrak{P}_s^{(N),(\lambda_1,\dots,\lambda_N)}= \mu_{t+s}^{x_*}$. 
\end{itemize}

We do not attempt to prove this here in complete generality but we will nevertheless justify it in Lemma \ref{AssumptionE} below for the $\textnormal{BESQ}(\delta)$ case and $x_*=0$ which is our main interest. Under the assumption (\textbf{E}) above, we note that by L'H\^{o}pital's rule, $\mu_t^{x_*}$ is given by 
the explicit formula, for $t>0$:
\begin{align}\label{EntranceDensity}
    \mu_t^{x_*}(dy)=e^{-\sum_{i=1}^N\lambda_i t}\frac{\det\left(\partial^{i-1}_{x}p_t(x,y_j)|_{x=x_*}\right)_{i,j=1}^N}{\det\left(\partial^{i-1}_{x}\psi_{\lambda_j}(x)|_{x=x_*}\right)_{i,j=1}^N}\det\left(\psi_{\lambda_i}(y_j)\right)_{i,j=1}^Ndy.
\end{align}

Finally observe that, for any $t>0$, $\mu_t^{x_*}$ does not charge the boundary of $\mathbb{W}_N$ and we can thus define the Gibbs measure $\mathfrak{CM}_{2N-1}^{\mu_t^{x_*}}$ on $\mathfrak{HA}_{2N-1}$. We have the following result.

\begin{prop}\label{EntranceLawProp}
In the setting of Proposition \ref{GeneralInteracting}, under assumption  (\textbf{E}) above,  $\left(\mathfrak{CM}_{2N-1}^{\mu_t^{x_*}}\right)_{t>0}$ forms an entrance law for the dynamics (\ref{Dynamics}).
\end{prop}

\begin{proof}
Let $t>0$. We need to show that if the dynamics (\ref{Dynamics}) are initialized according to $\mathfrak{CM}_{2N-1}^{\mu_t^{x_*}}$ and then run for any time $s$ the resulting distribution of the array is given by $\mathfrak{CM}_{2N-1}^{\mu_{t+s}^{x_*}}$. Since, for any $t>0$, $\mu_t^{x_*}$ does not charge the boundary we can thus apply Proposition \ref{GeneralInteracting} from which the desired conclusion follows since $\mathfrak{CM}_{2N-1}^{\mu_{t}^{x_*}\mathfrak{P}_s^{(N),(\lambda_1,\dots,\lambda_N)}}=\mathfrak{CM}_{2N-1}^{\mu_{t+s}^{x_*}}$.
\end{proof}

\begin{rmk}
Analogous results can be obtained for the degenerate case of Section \ref{SectionDegenerate} by word for word adaptation. We omit the statement.
\end{rmk}

We briefly justify assumption (\textbf{E}) for our main case of interest.

\begin{lem}\label{AssumptionE}
Assumption (\textbf{E}) holds, with $x_*=0$, in the $\textnormal{BESQ}(\delta)$ case.
\end{lem}

\begin{proof}
We first prove that $\mu_t^{0}$ is well-defined by an application of Scheff\'{e}'s lemma. Note that, the partial derivatives $x\mapsto \partial_x^iq^{(\nu)}_t(x,y)$ extend continuously to $x=0$; this can be seen from the exact expression (\ref{BesselIdentity}) below. Moreover, making use of the explicit formula
\begin{align*}
\phi_{\frac{\mu}{2}}^{(\nu)}(x)=(\mu x)^{-\frac{\nu}{2}}I_{\nu}\left(\sqrt{\mu x}\right)=\sum_{m=0}^\infty \frac{\mu^{m}}{m!\Gamma(m+\nu+1)2^{2m}}x^m,  
\end{align*}
we also obtain that the denominator in (\ref{EntranceDensity}) is non-vanishing
\begin{align*}
 \det \left(\partial^{i-1}_{x}\phi^{(\nu)}_{\frac{\mu_j}{2}}(x)\big|_{x=0}\right)_{i,j=1}^N=\det\left(\frac{\mu_j^{i-1}}{\Gamma(i+\nu)2^{2(i-1)+\nu}}\right)_{i,j=1}^N=\prod_{i=1}^N\frac{1}{\Gamma(i+\nu)2^{2(i-1)+\nu}} \mathsf{\Delta}_N(\mu)>0,
\end{align*}
and hence we have the convergence of the densities. While to show that $\mu_t^{x_*}$ is a probability measure we have the following generic computation using Andr\'{e}ief's identity \cite{Andreief} in the first equality:
\begin{align*}
 \int_{\mathbb{W}_N} \det\left(\partial^{i-1}_{x}p_t(x,y_j)\right)_{i,j=1}^N\det\left(\psi_{\lambda_i}(y_j)\right)_{i,j=1}^Ndy&=\det\left(\int_0^{\infty}\partial_x^{i-1}p_t(x,z)\psi_{\lambda_j}(z)dz\right)_{i,j=1}^N\\
 &=\det\left(\partial^{i-1}_x e^{\lambda_j t}\psi_{\lambda_j}(x)\right)_{i,j=1}^N,
\end{align*}
where in the last equality we have implicitly assumed some domination in order to bring the partial derivative outside the integral\footnote{Alternatively, one can check this one-dimensional identity directly in the $\textnormal{BESQ}(\delta)$ case using the explicit formulae.}. This interchange of derivative and integral can be justified in the $\textnormal{BESQ}(\delta)$ case by iterative use of the exact formula, obtained by direct computation,
\begin{align}\label{BesselIdentity}
\partial_x q_t^{(\nu)}(x,y)=\frac{1}{2t}\left[q_t^{(\nu+1)}(x,y)-q_t^{(\nu)}(x,y)\right],
\end{align}
along with the bound $I_\nu(x)\le \frac{x^\nu}{2^\nu \Gamma(\nu+1)}e^x$, valid for $x>0$ and $\nu>-\frac{1}{2}$, and hence $\mu_t^0$ is well-defined.

Now, for the entrance law property it is most convenient to relate this to the eigenvalue evolution of the Laguerre matrix diffusion with semigroup $\left(q_t^{(\nu),N};t\ge 0\right)$ from (\ref{LaguerreEvalSemi}) which is known to be Feller-Markov in $\{x=(x_1,\dots,x_N) \in [0,\infty)^N:x_1\le \dots \le x_N\}$, see Proposition 1.3 in \cite{Gateway} (and hence we have weak continuity in the initial condition, see for example Theorem 2.5 in Chapter 4 of \cite{EthierKurtz}). Towards this end, observe that using the semigroup property and by noting the explicit form (\ref{LaguerreEvalSemi}), (\ref{CondBesselDriftsTransDens}) of the transition kernels, we can write
\begin{align*}
q_{t+s}^{(\nu), N, \mu}(x,dy)=e^{-\frac{1}{2}\sum_{i=1}^N\mu_i(t+s)}\frac{\mathsf{\Delta}_N(x)\det\left(\phi^{(\nu)}_{\frac{\mu_i}{2}}(y_j)\right)_{i,j=1}^N}{\det\left(\phi^{(\nu)}_{\frac{\mu_i}{2}}(x_j)\right)_{i,j=1}^N\mathsf{\Delta}_N(y)}\int_{z\in \mathbb{W}_N}q_t^{(\nu),N}(x,z) q_s^{(\nu),N}(z,y)dzdy.
\end{align*}
Taking the limit $(x_1,\dots,x_N)\to (0,\dots,0)$ and after some rearranging of the right hand side gives the required equation.
\end{proof}

\subsection{Interacting particle system at the right edge}\label{SectionEdge}
As briefly discussed in Remark \ref{RemarkEdge}, we observe that the rightmost coordinates in the dynamics (\ref{Dynamics}) on the interlacing array form an autonomous interacting particle system. It is the analogue, in this setting, of Brownian motions with one-sided collisions (also called Brownian TASEP) which is equivalent to Brownian last passage percolation. The Brownian model has been extensively studied in the literature see for example \cite{Baryshnikov,OConnellYor,BougerolJeulin,BBO,Warren,ReflectedBrownianKPZ,NQR} and the references therein.

By simply relabelling the SDE (\ref{Dynamics}) for the rightmost coordinates we see that this interacting particle system with one-sided collisions is described by the SDE
\begin{align}\label{Edge1}
d\mathsf{x}_i(t)&=\sqrt{2\mathsf{a}\left(\mathsf{x}_i(t)\right)}d\mathsf{w}_i(t)+\mathsf{B}^{(i)}\left(\mathsf{x}_i(t)\right)dt+\frac{1}{2}d\mathfrak{l}_i(t),
\end{align}
where the $\mathsf{w}_i$ are independent standard Brownian motions, $\mathsf{B}^{(i)}$ is given by (\ref{FunctionB}) and as before $\mathfrak{l}_0\equiv 0$ and for $i>1$, $\mathfrak{l}_i$ is the semimartingale local time of $\mathsf{x}_i-\mathsf{x}_{i-1}$ at $0$. Similarly, in the degenerate case of Section \ref{SectionDegenerate} we have the SDE
\begin{align}\label{Edge2}
    d\mathsf{x}_i(t)&=\sqrt{2\mathsf{a}\left(\mathsf{x}_i(t)\right)}d\mathsf{w}_i(t)+\mathfrak{B}^{(i)}\left(\mathsf{x}_i(t)\right)dt+\frac{1}{2}d\mathfrak{l}_i(t),
\end{align}
where $\mathfrak{B}^{(i)}$ is given by (\ref{FunctionBfrak}).

Consider the map $\Pi^{\textnormal{edge}}_N:\mathfrak{HA}_N\to \overline{\mathbb{W}}_N$ \footnote{Observe that if we denote by $\mathfrak{HA}_N^{\circ}$ the subset of $\mathfrak{HA}_N$ with all (not just on a single level) coordinates in the array being distinct then $\Pi^{\textnormal{edge}}_N:\mathfrak{HA}^{\circ}_N\to \mathbb{W}_N$.} given by $\Pi^{\textnormal{edge}}_N\left[\left(x^{(i)}\right)_{i=1}^N\right]=\left(x^{(i)}_{\left \lfloor \frac{i+1}{2}\right \rfloor}\right)_{i=1}^N$. We then have the following result which is the analogue in our setting of a result of Baryshnikov \cite{Baryshnikov} and Gravner-Tracy-Widom \cite{GTW} for a fixed time and at the process level of results of O'Connell-Yor \cite{OConnellYor} and Bougerol-Jeulin \cite{BougerolJeulin}, see also \cite{BBO,Warren}, for the Brownian model.

\begin{prop} Under the assumptions of Proposition \ref{GeneralInteracting}, consider a probability measure $\mathcal{M}$ on $\mathbb{W}_N$ and let $\mathcal{M}^{\textnormal{edge}}=\left(\Pi^{\textnormal{edge}}_{2N-1}\right)_*\mathfrak{GM}_{2N-1}^{\mathcal{M}}$. Suppose that (\ref{Edge1}) is initialized according to $\mathcal{M}^{\textnormal{edge}}$, while the diffusion process $\left(\left(\mathsf{z}_1(t),\dots,\mathsf{z}_N(t)\right);t\ge 0\right)$ associated to the semigroup $\left(\mathfrak{P}_t^{(N),(\lambda_1,\dots,\lambda_N)};t\ge 0\right)$ is initialized according to $\mathcal{M}$. Then, the rightmost coordinates of both processes are equal in distribution, namely
\begin{align*}
 \left(\mathsf{x}_{2N-1}(t);t\ge 0\right)\overset{\textnormal{d}}{=}\left(\mathsf{z}_N(t);t\ge 0\right).
\end{align*}

\end{prop}
\begin{proof}
This is a consequence of Proposition \ref{GeneralInteracting} by looking at the top right coordinate in $\mathfrak{HA}_{2N-1}$.
\end{proof}

\begin{rmk}
By standard results the diffusion process $\left(\left(\mathsf{z}_1(t),\dots,\mathsf{z}_N(t)\right);t\ge 0\right)$ associated to the semigroup $\left(\mathfrak{P}_t^{(N),(\lambda_1,\dots,\lambda_N)};t\ge 0\right)$ has a description in terms of a SDE with a singular drift, in analogy to Dyson's Brownian motion \cite{Dyson,AGW}. We say a bit more about this in Appendix \ref{SectionSDE}.
\end{rmk}

\begin{rmk}
A completely analogous result holds in the degenerate case for the rightmost coordinates of (\ref{Edge2}) and of the diffusion corresponding to $\mathcal{P}_t^{N,\mathfrak{h}^{(N)}}$. We omit the statement.
\end{rmk}

\begin{rmk}
The statement also extends under assumption (\textbf{E}) of Section \ref{SectionEntrance} to starting the dynamics using an entrance law (in the particular case of $\textnormal{BESQ}(\delta)$ for example starting from the origin) see Proposition \ref{EntranceLawProp}.
\end{rmk}

\begin{appendices}

\section{A remark on the SDE description for conditioned diffusions}\label{SectionSDE}
We briefly discuss the singular SDE description, akin to the Dyson SDE \cite{Dyson,AGW}, for the diffusion associated to the semigroup $\left(\mathfrak{P}_t^{(N),(\lambda_1,\dots,\lambda_N)};t\ge0\right)$.

\begin{prop} Under the running assumptions of Section \ref{SectionInteractingDiffusions} on the $\mathsf{L}$-diffusion, consider the diffusion process $\left(\left(\mathsf{z}_1(t),\dots,\mathsf{z}_N(t)\right);t\ge 0\right)$ associated to $\left(\mathfrak{P}_t^{(N),(\lambda_1,\dots,\lambda_N)};t\ge0\right)$, with $\lambda_1<\cdots<\lambda_N$, starting from $\mathbb{W}_N$. Then, it is the unique strong solution to the SDE with almost surely non-colliding coordinates
\begin{align}\label{Top1}
d\mathsf{z}_i(t)=\sqrt{2\mathsf{a}\left(\mathsf{z}_{i}(t)\right)}d\mathsf{w}_{i}(t)+\left(\mathsf{b}\left(\mathsf{z}_i(t)\right)+2\mathsf{a}\left(\mathsf{z}_i(t)\right)\frac{\partial_{i}\det \left(\psi_{\lambda_k}(\mathsf{z}_j(t))\right)^N_{k,j=1}}{\det\left(\psi_{\lambda_k}(\mathsf{z}_j(t))\right)^N_{k,j=1}}\right)dt, \ \ i=1,\dots,N,
\end{align}
where $\mathsf{w}_1,\dots,\mathsf{w}_N$ are independent standard Brownian motions and the notation $\partial_i$ denotes the derivative in the $x_i$ variable.
\end{prop}

\begin{proof}
First, by standard results, see for example \cite{RevuzYor,Pinsky}, on how a diffusion process generator transforms under a Doob $h$-transform we get that the generator $\mathbf{L}^{(N),(\lambda_1,\dots,\lambda_N)}$ of $\left(\mathfrak{P}_t^{(N),(\lambda_1,\dots,\lambda_N)};t\ge 0\right)$ is given by
\begin{align*}
 \mathbf{L}^{(N),(\lambda_1,\dots,\lambda_N)}&=\left(\frac{\det\left(\psi_{\lambda_i}(x_j)\right)_{i,j=1}^N}{\prod_{i=1}^N\psi_{\lambda_i}(x_i)}\right)^{-1}\circ\left[\sum_{i=1}^N\mathsf{L}_{x_i}^{\psi_{\lambda_i}}\right]\circ\left(\frac{\det\left(\psi_{\lambda_i}(x_j)\right)_{i,j=1}^N}{\prod_{i=1}^N\psi_{\lambda_i}(x_i)}\right)\\  
 &=\sum_{i=1}^N\mathsf{a}(x_i)\partial_{x_i}^2+\sum_{i=1}^N\left(\mathsf{b}(x_i)+2\mathsf{a}(x_i)\frac{\partial_{x_i}\det\left(\psi_{\lambda_k}(x_j)\right)_{k,j=1}^N}{\det\left(\psi_{\lambda_k}(x_j)\right)_{k,j=1}^N}\right)\partial_{x_i}.
\end{align*}
Then, the argument to show that the associated SDE (\ref{Top1}) has a unique strong solution with almost surely no collisions if started in $\mathbb{W}_N$ \footnote{Presumably, possibly under some further assumptions, one could start the SDEs from a singular point in $\overline{\mathbb{W}}_N$ with coinciding coordinates, see for example \cite{AGW} and \cite{Noncolliding} for a theory covering a general class of examples which unfortunately does not seem to apply directly to our setting. We do not attempt to justify such a statement here but it would be interesting to explore it further.} is standard and goes as follows, see for example \cite{BruMult,Bru,Demni,AGW}. Consider the stopping time
\begin{align*}
\tau=\inf\left\{t\ge 0: \exists \ i<j \ \textnormal{ such that } \mathsf{z}_i(t)=\mathsf{z}_j(t)\right\}.
\end{align*}
Due to the regularity of the coefficients away from the singularity, when two coordinates coincide, a unique strong solution exists up to time $\tau$, see \cite{IkedaWatanabe}. It then suffices to show that $\tau=\infty$ almost surely. Observe that, the process $\left(\mathfrak{U}(t);t\ge 0\right)$ given by
\begin{align*}
 \mathfrak{U}(t)=\frac{\prod_{i=1}^N\psi_{\lambda_i}(\mathsf{z}_i(t))}{\det\left(\psi_{\lambda_i}(\mathsf{z}_j(t))\right)^N_{i,j=1}}, \ \ t\ge 0,
\end{align*}
is a non-negative continuous local martingale and thus a non-negative continuous supermartingale which moreover satisfies $\lim_{t\to \tau}\mathfrak{U}(t)=\infty$. The conclusion follows.
\end{proof}

\section{A remark on the determinantal point process property}\label{SectionDet}

Observe that, both the dynamics corresponding to the diffusion process with semigroup $\left(\mathfrak{P}_t^{(N),(\lambda_1,\dots,\lambda_N)};t\ge 0\right)$ viewed at different times $t_1\le \cdots \le t_m$ and the random interlacing array obtained by running the dynamics (\ref{Dynamics}) for a fixed time $T$ can be viewed as random point processes. In both cases, 
by the results of this paper, if one writes down explicitly the corresponding probability distributions (for certain initial conditions) it is seen that they are given in terms of products of determinants having a special structure. 

It is then an immediate consequence of the Eynard-Mehta theorem \cite{EynardMehta,BorodinRains}, that these point processes are determinantal. Namely, their correlation functions are determined by (and given as determinants of) a so-called correlation kernel, see for example \cite{AGW,BorodinRains}.

However, the Eynard-Mehta theorem, see \cite{BorodinRains}, gives the correlation kernel in an implicit form and to obtain an explicit formula for the kernel is a highly non-trivial problem that we do not pursue here. One needs to invert a certain matrix or equivalently perform a biorthogonalization explicitly. Then, the model can be amenable to asymptotic analysis. It would be interesting to investigate this for the general setting of this paper in future work.
\end{appendices}

\end{document}